\documentclass[11pt, reqno]{amsart}
\usepackage{amsmath,amssymb}
\usepackage{mathtools}
\mathtoolsset{showonlyrefs}

\newcommand{\bZ}{\mathbb{Z}}
\newcommand{\bR}{\mathbb{R}}
\newcommand{\bC}{\mathbb{C}}
\newcommand{\bQ}{\mathbb{Q}}
\newcommand{\Fg}{\mathfrak{g}}
\newcommand{\Fh}{\mathfrak{h}}
\newcommand{\be}{\mathbf{e}}
\newcommand{\cO}{\mathcal{O}}
\newcommand{\A}{\mathcal{A}}
\newcommand{\sB}{\mathsf{B}}
\newcommand{\sQ}{\mathsf{Q}}
\newcommand{\vpi}{\varpi}
\newcommand{\ve}{\varepsilon}
\newcommand{\af}{\mathrm{af}}
\DeclareMathOperator{\wend}{end}
\DeclareMathOperator{\wt}{wt}
\DeclareMathOperator{\dn}{down}
\DeclareMathOperator{\Gr}{Gr}
\DeclareMathOperator{\IG}{IG}
\DeclareMathOperator{\Fl}{Fl}
\newcommand{\pair}[2]{\langle #1,\,#2 \rangle}
\newcommand{\edge}[1]{ \xrightarrow{\hspace{2pt}#1\hspace{2pt}} }
\newcommand{\mcr}[1]{\lfloor #1 \rfloor}

\DeclareMathOperator{\QB}{QB}

\newcommand{\vtl}{\vartriangleleft}

\newtheorem{thm}{Theorem}
\newtheorem{conj}[thm]{Conjecture}
\newtheorem{prop}[thm]{Proposition}
\newtheorem{cor}[thm]{Corollary}
\newtheorem{lem}[thm]{Lemma}
\theoremstyle{definition}
\newtheorem{dfn}[thm]{Definition}
\newtheorem{exa}[thm]{Example}

\newtheorem{rema}[thm]{Remark}

\newcommand{\bp}{\textbf{p}}
\DeclareMathOperator{\sign}{sign}

\makeatletter

\let\choose\@@choose

\makeatother

\usepackage[left=1in, right=1in, top=1in, bottom=1in]{geometry}

\makeatletter
\let\@wraptoccontribs\wraptoccontribs
\makeatother

\begin{document}

\title[Quantum $K$-theory Chevalley formulas in the parabolic case]{Quantum $K$-theory Chevalley formulas\\ in the parabolic case}

\author[T.~Kouno]{Takafumi Kouno}
\address[Takafumi Kouno]{Department of Mathematics, Faculty of Science and Engineering, Waseda University,
3-4-1 Okubo, Shinjuku-ku, Tokyo 169-8555, Japan}
\email{t.kouno@aoni.waseda.jp}

\author[C.~Lenart]{Cristian Lenart}
\address[Cristian Lenart]{Department of Mathematics and Statistics, State University of New York at Albany, 
Albany, NY 12222, U.S.A.}
\email{clenart@albany.edu}

\author[S.~Naito]{Satoshi Naito}
\address[Satoshi Naito]{Department of Mathematics, Tokyo Institute of Technology,
2-12-1 Oh-Okayama, Meguro-ku, Tokyo 152-8551, Japan}
\email{naito@math.titech.ac.jp}

\author[D.~Sagaki]{Daisuke Sagaki}
\address[Daisuke Sagaki]{Department of Mathematics, Faculty of Pure and Applied Sciences, University of Tsukuba,
1-1-1 Tennodai, Tsukuba, Ibaraki 305-8571, Japan}
\email{sagaki@math.tsukuba.ac.jp}

\address[W.~Xu]{Department of Mathematics, 460 McBryde Hall, Virginia Tech, 225 Stanger Street, Blacksburg, VA 24061-1026}
\email{weihong@vt.edu}

\contrib[with an appendix by]{Takafumi Kouno, Cristian Lenart, Satoshi Naito, Daisuke Sagaki, and Weihong Xu}

\begin{abstract} We derive cancellation-free Chevalley-type multiplication formulas for the $T$-equivariant quantum $K$-theory ring of Grassmannians of type $A$ and $C$, and also those of two-step flag manifolds of type $A$. They are obtained based on the uniform Chevalley formula  in the $T$-equivariant quantum $K$-theory ring of arbitrary flag manifolds $G/B$, which was derived earlier in terms of the quantum alcove model, by the last three authors.
\end{abstract}


\maketitle

\section{Introduction}

Y.-P.~Lee defined the (small) {quantum $K$-theory} of a smooth projective variety $X$, denoted by $QK(X)$ (see \cite{leeqkt}, and also \cite{givqkt}). This is a deformation of the ordinary $K$-ring of $X$, analogous to the relation between quantum cohomology and ordinary cohomology. The deformed product is defined in terms of certain generalizations of {Gromov-Witten invariants} (i.e., the structure constants in quantum cohomology), called {quantum $K$-invariants of Gromov-Witten type}. 

Given a simply-connected simple algebraic group $G$ over $\mathbb{C}$, with Borel subgroup $B$, and maximal torus $T \subset B$, we consider the corresponding flag manifold $G/B$, the $T$-equivariant $K$-theory $K_{T}(G/B)$, and the $T$-equivariant quantum $K$-ring $QK_{T}(G/B) := K_T(G/B) \otimes_{{\mathbb Z}[\Lambda]} \bZ[\Lambda][\![Q]\!]$, where $\bZ[\Lambda][\![Q]\!]$ is the ring of formal power series with coefficients in $\bZ[\Lambda]$ in the (Novikov) variables $Q_{i} = Q^{\alpha_i^{\vee}}$, $i \in I$, with $I$ the index set for the simple roots $\alpha_i$ of $G$; $QK_{T}(G/B)$ has a $\bZ[\Lambda][\![Q]\!]$-basis given by the classes $[{\mathcal O}^w]$ of the structure sheaves of the (opposite) Schubert varieties $X^{w} \subset G/B$ indexed by the elements $w$ of the Weyl group $W = \langle s_i := s_{\alpha_i} \mid i \in I \rangle$ of $G$. 
Also, given a (standard) parabolic subgroup $P_{J} \supset B$ corresponding to a subset $J$, we also consider the partial flag manifold $G/P_{J}$, the $T$-equivariant $K$-theory $K_{T}(G/P_{J})$, and the $T$-equivariant quantum $K$-ring $QK_{T}(G/P_{J}) := K_T(G/P_{J}) \otimes_{{\mathbb Z}[\Lambda]} \bZ[\Lambda][\![Q_{K}]\!]$, where $\bZ[\Lambda][\![Q_{K}]\!]$ is the ring of formal power series with coefficients in $\bZ[\Lambda]$ in the (Novikov) variables $Q_{k}$, $k \in K := I \setminus J$; $QK_{T}(G/P_{J})$ has a $\bZ[\Lambda][\![Q_{K}]\!]$-basis given by the (opposite) Schubert classes $[\cO^{y}_{J}]$, for $y \in W^{J}$, where $W^{J}$ denotes the set of minimal-length coset representatives for the cosets in $W/W_{J}$, where $W_{J} := \langle s_j \mid j \in J \rangle \subset W$. 
A Chevalley formula (in cohomology, $K$-theory, or their quantum versions) expresses the Schubert basis expansion of the product between an arbitrary Schubert class and the class of a line bundle, or a Schubert class indexed by a simple reflection (i.e., a divisor class). Having an explicit Chevalley formula in the quantum $K$-ring of an arbitrary flag manifold is important because this algebra is uniquely determined by products with divisor classes~\cite{bcmcfe}, together with its $K_{T}(\mathrm{pt})$-module structure; 
here, $K_{T}(\mathrm{pt}) = R(T)$, the representation ring of $T$, is identified with the group algebra $\bZ[\Lambda]$ of the weight lattice $\Lambda$ of $G$.

A cancellation-free Chevalley formula in the $T$-equivariant quantum $K$-theory of $G/B$ was recently given in~\cite{lnscfs} (see also \cite{lnsccf}); cf. the related conjecture in~\cite{lapawg}. This formula is expressed in terms of the so-called quantum alcove model, which was introduced in~\cite{lalgam}. It generalizes the formula in the $T$-equivariant $K$-theory of $G/B$ in~\cite{lapawg}, which can easily be restricted to the partial flag manifold $G/P_{J}$ for $J \subset I$. However, such a restriction does not work in quantum $K$-theory, because of the lack of functoriality. 
In contrast, we know (see \cite{katqkg}) that for a subset $J \subset I$, the ($\bZ[\Lambda]$-linear) push-forward $(\pi_{J})_{*} : K_{T}(G/B) \to K_{T}(G/P_{J})$, induced by the natural proection $\pi_{J} : G/B \to G/P_{J}$ with $P_{J}$ the (standard) parabolic subgroup of $G$ corresponding to $J$, yields a surjective $\bZ[\Lambda]$-module homomorphism from $QK_{T}^{\mathrm{poly}}(G/B) := K_T(G/B) \otimes_{{\mathbb Z}[\Lambda]} \bZ[\Lambda][Q] \subset QK_{T}(G/B)$ onto $QK_{T}^{\mathrm{poly}}(G/P_{J}) := K_T(G/P_{J}) \otimes_{{\mathbb Z}[\Lambda]} \bZ[\Lambda][Q_{K}]$ such that 
\begin{equation}
\Phi_{J}([\mathcal{O}^{w}] \cdot [\mathcal{O}_{G/B}(-\varpi_k)]) = [\mathcal{O}_{J}^{\lfloor w \rfloor}] \cdot [\mathcal{O}_{G/P_{J}}(-\varpi_k)]
\end{equation}
for $w \in W$ and $k \in K := I \setminus J$, where $\lfloor w \rfloor$ denotes the minimal-length coset representative for the coset $w W_{J}$ in $W/W_{J}$; 
here, $\bZ[\Lambda][Q]$ (resp., $\bZ[\Lambda][Q_{K}]$) is the polynomial ring with coefficients in $\bZ[\Lambda]$ in the (Novikov) variables $Q_i$, $i \in I$, (resp., $Q_k$, $k \in K = I \setminus J$).

Originally, in \cite{katqkg}, the fact above was proved by using the relationship between the $T$-equivariant $K$-group of a (full or partial) semi-infinite flag manifold and the $T$-equivariant quantum $K$-theory of a (full or partial) flag manifold.
Here we should mention that the existence of the surjective $\mathbb{Z}[\Lambda]$-algebra homomorphism $\Phi_{J}$ can also be verified by using the $K$-theoretic analog, conjectured in \cite{llmscnj}, of the Peterson homomorphism ($K$-Peterson homomorphism for short), which is a homomorphism of $\bZ[\Lambda]$-algebras from the $K$-homology of the affine Grassmannian to (the localization, with respect to the positive part $Q^{\vee,+}$ of the coroot lattice $Q^{\vee}$, of) the quantum $K$-ring of $G/P_{J}$; a (new) proof of the existence of the $K$-Peterson homomorphism has been given recently by \cite{newprf}. 
Indeed, as stated in the proof of \cite[Lemma~2.12]{newprf}, under the $K$-Peterson homomorphism (which is a $\mathbb{Z}[\Lambda]$-algebra homomorphism) in the case of the Borel subgroup $B$, the classes of the structure sheaves of Schubert varieties in the affine Grassmannian indexed by the minimal-length coset representatives for $W_{\mathrm{af}}/W$, with $W_{\mathrm{af}}$ the affine Weyl group and $W$ the finite Weyl group, are sent injectively to the corresponding (opposite) Schubert classes in $QK_{T}(G/B)$ multiplied by explicit monomials in the Novikov variables correspondong to anti-dominant coroots in $- Q^{\vee,+}$. Hence, by composing the inverse of the $K$-Peterson homomorphism in the case of $B$ with the $K$-Peterson homomorphism (which is also a $\mathbb{Z}[\Lambda]$-algebra homomorphism) in the case of $P_{J} \supset B$, we obtain the desired surjective $\bZ[\Lambda]$-algebra homomorphism $\Phi_{J}$; 
here we use the fact that all the (opposite) Schubert classes will lie in the image of the $K$-Peterson homomorphism in the case of $B$ if we multiply them by a monomial in the Novikov variables corresponding to a (fixed) regular anti-dominant coroot in $- Q^{\vee,+}$.
The details of these arguments are explained in Appendix~\ref{sec:A}. 

In this paper, on the basis of the fact above, we derive cancellation-free Chevalley formulas in the $T$-equivariant quantum $K$-ring $QK_T(G/P_{J})$ of the partial flag manifold $G/P_{J}$, where $P_{J} \supset B$ is the (standard) parabolic subgroup of $G$ corresponding to $J \subset I$ in the following two cases: (i) $G$ is of type $A$ or $C$ and $J = I \setminus \{ k \}$ for $k \in I$; (ii) $G$ is of type $A$ and $J = I \setminus \{ k_1, k_2 \}$ for $k_1, k_2 \in I$ with $k_1 \not= k_2$.
More precisely, the mentioned Chevalley formulas express the quantum multiplication in $QK_T(G/P_J)$ with the class of the line bundle associated to the anti-dominant fundamental weight $-\varpi_k$ for $k \in I \setminus J$. 
Our strategy is the following: start with the Chevalley formula for $QK_{T}^{\mathrm{poly}}(G/B) \subset QK_T(G/B)$ in~\cite{lnscfs}; apply the $\bZ[\Lambda]$-module surjection $\Phi_{J} : QK_{T}^{\mathrm{poly}}(G/B) \rightarrow QK_{T}^{\mathrm{poly}}(G/P_J)$ (which respects quantum multiplications) above; perform all cancellations, which arise via a sign-reversing involution. 
In addition, as an application of our Chevalley formulas, we prove the positivity property of certain structure constants of the quantum $K$-ring of a Grassmannian of type $C$ and that of a two-step flag manifold of type $A$, 
as well as that for an arbitrary full flag manifold. 

The resulting Chevalley formulas for Grassmannians of types $A$ and $C$ and also those for two-step flag manifolds of type $A$ are no longer uniform, and they might also involve several cases.
This fact validates our approach of deriving them from the uniform formula for $G/B$. Note that, in many cases, the opposite approach works better, namely the formulas for Grassmannians are obtained first, because they are easier. 

We now compare our work with two related papers. In~\cite{knscfa}, a quantum $K$-theory Chevalley formula is given in $QK_T(G/P_J)$, where $J=I\setminus\{k\}$, for the line bundle associated to $-\varpi_k$, assuming that $\varpi_k$ is a minuscule fundamental weight in type $A$, $D$, $E$, or $B$. The formulas are expressed in terms of the quantum Bruhat graph (on which the quantum alcove model is based). The approach in the present paper is simpler, and has the advantage of being easier to be extended to other partial flag manifolds; 
in fact, we also obtain a quantum $K$-theory Chevalley formula for two-step flag manifolds of type $A$. 
On another hand, the Chevalley formulas in~\cite{bcmcfe} for cominuscule varieties are of a different nature than the corresponding cases of the formulas in this paper. Indeed, the role of the quantum Bruhat graph is not transparent in~\cite{bcmcfe}.

\subsection*{Acknowledgements.}
T.K. was partly supported by JSPS Grant-in-Aid for Scientific Research 20J12058 and 22J00874. 
C.L. was partly supported by the NSF grant DMS-1855592.
S.N. was partly supported by JSPS Grant-in-Aid for Scientific Research (B) 16H03920. 
D.S. was partially supported by JSPS Grant-in-Aid for Scientific Research (C) 15K04803 and 19K03415.

\section{Background}

Consider a simply-connected simple algebraic group $G$ over $\mathbb{C}$, with Borel subgroup $B$, and maximal torus $T$. 
Let $\mathfrak{g}$ be the corresponding finite-dimensional simple Lie algebra over $\mathbb{C}$, and $W$ its Weyl group, with length function denoted by $\ell(\cdot)$. 
Let $\Phi$, $\Phi^{+}$, and $\Phi^{-}$ be 
the set of roots, positive roots, and negative roots of $\mathfrak{g}$, respectively, and let $\Lambda$ be the corresponding weight lattice. Let $\alpha_i$, $i\in I$, be the simple roots, $\Delta := \{\alpha_{i} \mid i \in I\}$ the set of all simple roots, $\theta$ the highest root, and $\alpha^\vee$ the coroot associated with the root $\alpha$. The reflection corresponding to $\alpha$ is denoted, as usual, by $s_\alpha$, and we let $s_i:=s_{\alpha_i}$, $i\in I$, be  the simple reflections. Set $\rho:=(1/2) \sum_{\alpha \in \Phi^{+}} \alpha$. 

Let $J$ be a subset of $I$. 
We denote by $W_J:= \langle s_{i} \mid i \in J \rangle$ the parabolic subgroup of $W$ corresponding to $J$, 
and we identify $W/W_J$ with the corresponding set of minimal coset representatives, denoted by $W^J$; 
note that if $J=\emptyset$, then $W^{J}=W^{\emptyset}$ is identical to $W$. 
For $w \in W$, we denote by $\mcr{w}=\mcr{w}^{J} \in W^{J}$ 
the minimal coset representative for the coset $w W_{J}$ in $W/W_{J}$.

\subsection{The quantum Bruhat graph}\label{sec:qbg}
We start with the definition of this graph, which plays a fundamental role in our combinatorial model.

\begin{dfn}
The quantum Bruhat graph $\QB(W)$ is 
the $\Phi^{+}$-labeled
directed graph whose vertices are the elements of $W$, and 
whose directed edges are of the form: $w \edge{\beta} v$ 
for $w,v \in W$ and $\beta \in \Phi^{+}$ 
such that $v= {ws_{\beta}}$, and such that either of 
the following holds: 
(i) $\ell(v) = \ell (w) + 1$; 
(ii) $\ell(v) = \ell (w) + 1 - 2 \pair{\rho}{\beta^{\vee}}$.
An edge satisfying (i) (resp., (ii)) is called a Bruhat (resp., quantum) edge. 
\end{dfn}

In \cite{bfpmbo}, it is proved that the quantum Bruhat graph $\QB(W)$ has the following property (called the \emph{shellability}): for all $x, y \in W$, there exists a unique directed path from in $\QB(W)$ from $x$ to $y$ whose edge labels are increasing with respect to an arbitrarily fixed reflection order on $\Phi^{+}$. 

We recall an explicit description of the edges of the quantum Bruhat graphs of types $A$ and $C$. These results generalize the well-known criteria for covers of the Bruhat order in these cases~\cite{babccg}. 

In type $A_{n-1}$, the Weyl group elements (i.e., permutations) $w\in W(A_{n-1})=S_n$ are written in one-line notation $w=[w(1), \ldots, w(n)]$. For simplicity, we use the same notation $(i,j)$ with $1\le i<j\le n$ for the root $\alpha_{ij}$ and the reflection $s_{\alpha_{ij}}$, which is the transposition $t_{ij}$ of $i$ and $j$. We have $\theta=(1,n)$. We recall a criterion for the edges of the type $A_{n-1}$ quantum Bruhat graph. We need the circular order $\prec_i$ on $[n]$ starting at $i$, namely $i\prec_i i+1\prec_i\cdots \prec_i n\prec_i 1\prec_i\cdots\prec_i i-1$. It is convenient to think of this order in terms of the numbers $1,\ldots,n$ arranged clockwise on a circle, in this order. We make the convention that, whenever we write $a\prec b\prec c\prec\cdots$; i.e., the leftmost of the chain $a\prec b\prec c\prec\cdots$ we are writing is $a$, we refer to the circular order $\prec=\prec_a$.
	  
	\begin{prop}[\cite{Lenart}]
		\label{prop:quantum_bruhat_order_type_A}
	For $w\in S_n$ and $1\leq i<j\leq n$, we have an edge $w \stackrel{(i,j)}{\longrightarrow} w(i,j)$ if and only if 
	there is no $k$ such that $i<k<j$ and $w(i) \prec w(k) \prec w(j)$.
	\end{prop}

If there is a position $k$ as above, we say that the transposition of $w(i)$ and $w(j)$ {straddles} $w(k)$. We also let $w[i,j]:=[w(i), w(i+1), \cdots, w(j)]$. We continue to use this terminology and notation for the other classical types. 

The Weyl group of type $C_n$ is the group of signed permutations. These are bijections $w$ from 
$[\overline{n} ]:= 
\{ 1 < 2 < \dots < n < \overline{n} < \overline{n-1} < \dots < \overline{1} \}$ to 
$[\overline{n}]$ satisfying $w(\overline{\imath}) = \overline{w(i)}$.
Here $\overline{\imath}$ is viewed as $-i$, so $\overline{\overline{\imath}}= i$,
$|\overline{\imath}| = i$, and $\sign(\overline{\imath})=-1$.
We use both the window notation $w=[w(1), \ldots, w(n)]$ and the full one-line notation
$w=[w(1), \ldots, w(n), w(\overline{n}), \ldots, w(\overline{1})]$ for signed permutations.
For simplicity, given $1 \leq i < j \leq n$, we denote by $(i,j)$ the root 
$\varepsilon_i - \varepsilon_j$ and the corresponding reflection, which is identified
with the composition of transpositions $t_{ij}t_{\overline{\jmath \imath}}$.
Similarly,  for $1\leq i<j \leq n$, we denote by $(i,\overline{\jmath})=(j,\overline{\imath})$ the root 
$\varepsilon_i + \varepsilon_j$ and the corresponding reflection, which is 
identified with the composition of transpositions 
$t_{i\overline{\jmath}} t_{j\overline{\imath}}$.
Finally, we denote by $(i,\overline{\imath})$ the root $2\varepsilon_i$ and the 
corresponding reflection, which is identified with the transposition $t_{i\overline{\imath}}$. We have $\theta=(1,\overline{1})$. 

We now recall the criterion for the edges of the type $C_n$ quantum Bruhat graph. We need the circular order $\prec_i$ on $[\overline{n}]$ starting at $i$, which is 
defined similarly to the circular order on $[n]$, by thinking of the numbers $1,2,\ldots,n,\overline{n},\overline{n-1},\ldots,\overline{1}$ arranged clockwise on a circle, in this order. 
We make the same convention as above that, whenever we write 
$a \prec b \prec c \prec \cdots,$  we refer to the circular order $\prec = \prec_a$.

\begin{prop}[\cite{Lenart}] Let $w\in W(C_n)$ be a signed permutation. \begin{enumerate}
		\item[{\rm (1)}] 
			 Given $1 \leq i < j \leq n$, we have an edge
			$w \stackrel{(i,j)}{\longrightarrow} w(i,j)$ if and only if
			there is no $k$ such that 
			$i < k < j$ and $w(i) \prec w(k) \prec w(j)$.
		\item[{\rm (2)}] 
			 Given $1 \leq i < j \leq n$, we have an edge 
			$w\stackrel{(i,\overline{\jmath})}{\longrightarrow}$ if and
			only if $w(i) < w(\overline{\jmath})$,  
			$\sign(w(i)) = \sign(w(\overline{\jmath}))$, 
			and there is no $k$ such that $i<k<\overline{\jmath}$ and 
			$w(i) < w(k) < w(\overline{\jmath})$.
		\item[{\rm (3)}] 
			 Given  $1 \leq i \leq n$, we have an edge 
			$w \stackrel{(i,\overline{\imath})}{\longrightarrow} 
			w(i,\overline{\imath})$ if and only if there is no $k$ such
			that $i < k < \overline{\imath}$ (or, equivalently, $i<k\leq n$) and 
			$w(i) \prec w(k) \prec w(\overline{\imath})$.
	\end{enumerate}
	\label{prop:quantum_bruhat_order_type_C}
\end{prop}

\subsection{The quantum alcove model}\label{sec:lamch}
We need basic notions related to the combinatorial model known as the \emph{alcove model}, which was defined in~\cite{lapawg}. In particular, we need the notion of a \emph{$\lambda$-chain of roots}, 
where $\lambda$ is a weight. 
In this section, we recall definitions of these notions from \cite{lapawg}. 

Let $\Lambda$ be the weight lattice of $G$ and set $\Fh^{\ast}_{\bR} := \Lambda \otimes_{\bZ} \bR$. 
For $\alpha \in \Phi$ and $k \in \bZ$, we define a hyperplane $H_{\alpha, k}$ by $H_{\alpha, k} := \{ \xi \in \Fh^{\ast}_{\bR} \mid \pair{\xi}{\alpha^{\vee}} = k \}$. We denote by $s_{\beta, k}$, $\beta \in \Phi$ and $k \in \bZ$, the reflection with respect to $H_{\alpha, k}$. 
Then, an \emph{alcove} is defined to be a connected component of the space
\begin{equation}
\Fh^{\ast}_{\bR} \setminus \bigcup_{\alpha \in \Phi, \ k \in \bZ} H_{\alpha, k}. 
\end{equation}
If two alcoves $A$ and $B$ have a common wall, then $A$ and $B$ are said to be \emph{adjacent}. 
Let us take adjacent alcoves $A$ and $B$. If the common wall of $A$ and $B$ is contained in a hyperplane $H_{\alpha, k}$ for some $\alpha \in \Phi$ and $k \in \bZ$, and the vector $\alpha$ points a direction from $A$ to $B$, then we write $A \xrightarrow{\alpha} B$. 
We define a specific alcove $A_{\circ}$, called \emph{the fundamental alcove}, by 
\begin{equation}
A_{\circ} := \{ \xi \in \Fh^{\ast}_{\bR} \mid \text{$\pair{\xi}{\alpha^{\vee}} \ge 0$ for all $\alpha \in \Phi^{+}$} \}. 
\end{equation}
In addition, for $\lambda \in \Lambda$, we define an alcove $A_{\lambda}$ by $A_{\lambda} := A_{\circ} + \lambda = \{ \xi + \lambda \mid \xi \in A_{\circ} \}$. 

\begin{dfn}[{\cite[Definitions~5.2, 5.4]{lapawg}}]
\begin{enumerate}
\item An \emph{alcove path} is a sequence $(A_{0}, A_{1}, \ldots, A_{m})$ of alcoves such that for each $0 \le k \le m-1$, $A_{k}$ and $A_{k+1}$ are adjacent. If an alcove path $\Pi = (A_{0}, \ldots, A_{m})$ is shortest among all alcove paths from $A_{0}$ to $A_{m}$, we say that $\Pi$ is \emph{reduced}. 
\item Let $\lambda \in \Lambda$. A \emph{$\lambda$-chain of roots} is a sequence $\Gamma = (\beta_{1}, \ldots, \beta_{m})$ of roots such that there exists an alcove path $\Pi = (A_{\circ} = A_{0}, \ldots, A_{m} = A_{-\lambda})$ such that 
\begin{equation}
A_{0} \xrightarrow{-\beta_{1}} A_{1} \xrightarrow{-\beta_{2}} \cdots \xrightarrow{-\beta_{m}} A_{m}. 
\end{equation}
If $\Pi$ is reduced, then we also say that $\Gamma$ is \emph{reduced}. 
\end{enumerate}
\end{dfn}

Let $\lambda \in \Lambda$. Take a $\lambda$-chain $\Gamma = (\beta_{1}, \ldots, \beta_{m})$ and corresponding alcove path $(A_{0}, \ldots, A_{m})$. Set $r_i:=s_{\beta_i}$, $i = 1, \ldots, m$. Below, we present an explicit description of the chains of roots corresponding to the anti-dominant fundamental weights in the classical types, i.e., $\lambda=-\varpi_k$. 

We also need to recall the more general \emph{quantum alcove model}~\cite{lalgam}. We refer to~\cite[Section~3.2]{lnscfs} for more details. 
In the next definition, we use the following notation: for $\beta \in \Phi$, 
\begin{equation}
|\beta| := \begin{cases}
\beta & \text{if $\beta \in \Phi^{+}$}, \\ 
-\beta & \text{if $\beta \in \Phi^{-}$}. 
\end{cases}
\end{equation}
\begin{dfn}[\cite{lalgam}]
\label{def:admissible}
	A subset 
	$A=\left\{ j_1 < j_2 < \cdots < j_s \right\}$ of $[m]:=\{1,\ldots,m\}$ (possibly empty)
 	is a $w$-\emph{admissible subset} if
	we have the following directed path in the quantum Bruhat graph $\QB(W)$:
	\begin{equation}
	\label{eqn:admissible}
	 \Pi(w,A):\;\;\;\;w\xrightarrow{|\beta_{j_1}|} w r_{{j_1}} 
	\xrightarrow{|\beta_{j_2}|}  wr_{{j_1}}r_{{j_2}} 
	\xrightarrow{|\beta_{j_3}|}  \cdots 
	\xrightarrow{|\beta_{j_s}|}  wr_{{j_1}}r_{{j_2}} \cdots r_{{j_s}}=:\wend(w,A)\,.
	\end{equation}
\end{dfn}

We denote by $A^-$ the subset of $A$ corresponding to quantum steps in $\Pi(w,A)$. Let $\A(w,\Gamma)$ be the collection of all $w$-admissible subsets corresponding to the $\lambda$-chain $\Gamma$, and $\A_{\lessdot}(w,\Gamma)$ its subset consisting of all those $A$ with $A^-=\emptyset$ (i.e., $\Pi(w,A)$ is a saturated chain in Bruhat order). For convenience, we identify an admissible subset $J=\{j_1<\cdots<j_s\}$ with the corresponding sequence of roots $\{\beta_{j_1},\ldots,\beta_{j_s}\}$ in the $\lambda$-chain $\Gamma$ (in case of multiple occurrences of a root in $\Gamma$, we specify which one is considered). 
Also, we define statistics $\dn(w, A)$ for $A \in \A(w, \Gamma)$ as follows: 
\begin{equation}
\dn(w, A) := \sum_{j \in A^{-}} |\beta_{j}|^{\vee}. 
\end{equation}
In addition, let $H_{\beta_{j}, -l_{j}}$, $j = 1 \ldots m$, be the hyperplane containing the common wall of $A_{j-1}$ and $A_{j}$.
Then we define $\wt(w, A)$ by 
\begin{equation} 
\wt(w, A) := -ws_{\beta_{j_{1}}, -l_{j_{1}}} \cdots s_{j_{s}, -l_{j_{s}}} (-\lambda). 
\end{equation}

We use the same notation as in Section~\ref{sec:qbg}, and we start with type $A_{n-1}$. It is proved in \cite[Corollary~15.4]{lapawg} that, for any $k=1,\ldots,n-1$,
we have the following reduced $(-\varpi_k)$-chain of roots, denoted by $\Gamma(k)$
(note that all the roots in this $(-\varpi_k)$-chain are negated for simplicity of notation, and hence they are all positive roots): 
\begin{equation}\label{omegakchain}\begin{array}{lllll}
(&\!\!\!\!(1,n),&(1,n-1),&\ldots,&(1,k+1)\,,\\
&\!\!\!\!(2,n),&(2,n-1),&\ldots,&(2,k+1)\,,\\
&&&\ldots\\
&\!\!\!\!(k,n),&(k,n-1),&\ldots,&(k,k+1)\,\,)\,.
\end{array}\end{equation}

In type $A_{n-1}$, we have the (Dynkin) diagram automorphism 
\begin{equation}
\omega: [n-1] \rightarrow [n-1], \quad l \mapsto n - l. 
\end{equation}
By applying the diagram automorphism $\omega$ to $\Gamma(n-k)$, we obtain another reduced $(-\varpi_{k})$-chain (with all the roots negated), denoted by $\Gamma^{\ast}(k)$: 
\begin{equation}\label{omegakchain*}\begin{array}{lllll}
(&\!\!\!\!(1,n),&(2,n),&\ldots,&(k,n)\,,\\
&\!\!\!\!(1,n-1),&(2,n-1),&\ldots,&(k,n-1)\,,\\
&&&\ldots\\
&\!\!\!\!(1,k+1),&(2,k+1),&\ldots,&(k,k+1)\,\,)\,.
\end{array}\end{equation}

In type $C_n$, let 
\begin{equation}\label{omegachain-c}\Gamma(k):=\Gamma_2'\cdots\Gamma_k'\Gamma_1(k)\cdots\Gamma_k(k)\,,\end{equation}
where
\begin{align}
&\;\:\Gamma_j':=((1,\overline{\jmath}),(2,\overline{\jmath}),\ldots,(j-1,\overline{\jmath}))\,,\nonumber \\
&\begin{array}{llllll}\Gamma_j(k):=(\!\!\!\!\!&(1,\overline{\jmath}),&(2,\overline{\jmath}),&\ldots,&(j-1,\overline{\jmath}),\\
&(j,\overline{k+1}),&(j,\overline{k+2}),&\ldots,&(j,\overline{n}),\\&(j,\overline{\jmath}),\\ 
&(j,n),&(j,n-1),&\ldots,&(j,k+1)\,)\,.\end{array}\label{omegachain2-c}\end{align}
It is proved in \cite[Lemma~4.1]{lenhhl} that $\Gamma(k)$ is a reduced $(-\varpi_k)$-chain (with all the roots negated), for $1\le k\le n$.

\subsection{The quantum $K$-theory of flag manifolds}\label{sec:qk}
In order to describe the (small) $T$-equivariant quantum $K$-ring $QK_T(G/B)$, for the finite-dimensional flag manifold $G/B$, we associate a variable $Q_i$ to each simple coroot $\alpha_i^{\vee}$, 
and set $\bZ[Q]:=\bZ[Q_{i} \mid i \in I]$, $\bZ[\![Q]\!] := \bZ[\![Q_{i} \mid i \in I]\!]$; 
for each $\xi= \sum_{i \in I} d_i \alpha_i^{\vee}$ in $Q^{\vee,+}$, 
we set $Q^\xi:= \prod_{i \in I} Q_i^{d_i}$. 
Also, we set $\bZ[\Lambda][Q]:=\bZ[\Lambda]\otimes_{{\mathbb Z}} \bZ[Q]$, 
$\bZ[\Lambda][\![Q]\!]:=\bZ[\Lambda]\otimes_{{\mathbb Z}} \bZ[\![Q]\!]$, 
where $\mathbb{Z}[\Lambda]$ is the group algebra of the weight lattice $\Lambda$ of $G$, 
and is identified with the representation ring $R(T) = K_{T}(\mathrm{pt})$.
Following \cite{leeqkt} (and also \cite{givqkt}), we define the quantum $K$-ring $QK_T(G/B)$ to be the $\bZ[\Lambda][\![Q]\!]$-module
$K_T(G/B) \otimes_{{\mathbb Z}[\Lambda]} \bZ[\Lambda][\![Q]\!]$, 
equipped with the quantum product $\star$ given in terms of quantum $K$-invariants of Gromov-Witten type. 
The quantum $K$-ring $QK_T(G/B)$ has a $\mathbb{Z}[\Lambda][\![Q]\!]$-basis given by the classes $[{\mathcal O}^w]$ of the structure sheaves of the (opposite) Schubert varieties $X^{w} \subset G/B$ of codimension $\ell(w)$, for $w \in W$. 

We consider the maximal (standard) parabolic subgroup of $P_J \supset B$ of $G$ corresponding to the subset $J := I \setminus \{ k \}$, for some $k\in I$.
The $T$-equivariant quantum $K$-ring $QK_{T}(G/P_{J})$ of the partial flag manifold $G/P_{J}$ is defined as $K_{T}(G/P_{J}) \otimes_{\mathbb{Z}[\Lambda]} \mathbb{Z}[\Lambda][\![Q_{k}]\!]$, 
where $K_{T}(G/P_{J})$ is the $T$-equivariant $K$-theory of $G/P_{J}$, and $\mathbb{Z}[\Lambda][\![Q_{k}]\!]$ is the ring of formal power series with coefficients in $\bZ[\Lambda]$ in the single (Novikov) variable $Q_{k} = Q^{\alpha_{k}^{\vee}}$ corresponding to the simple coroot $\alpha_{k}^{\vee}$. The (opposite) Schubert classes $[\mathcal{O}_{J}^{y}]$, for $y \in W^{J}$, form a $\mathbb{Z}[\Lambda][\![Q_k]\!]$-basis. 

We also consider the (standard) parabolic subgroup $P_J \supset B$ of $G$ corresponding to the subset $J := I \setminus \{k_{1}, k_{2}\}$, for some $k_{1}, k_{2} \in I$ with $k_{1} \not= k_{2}$. 
In this case, the $T$-equivariant quantum $K$-ring $QK_{T}(G/P_{J})$ is defined as $K_{T}(G/P_{J}) \otimes_{\mathbb{Z}[\Lambda]} \mathbb{Z}[\Lambda][\![Q_{k_{1}}, Q_{k_{2}}]\!]$, 
where $\bZ[\Lambda][\![Q_{k_1}, Q_{k_2}]\!]$ is the ring of formal power series with coefficients in $\bZ[\Lambda]$ in the two (Novikov) variables $Q_{k_1}, Q_{k_2}$. 
As in the maximal parabolic case, the (opposite) Schubert classes $[\cO^{y}_{J}]$, for $y \in W^{J}$, form a $\mathbb{Z}[\Lambda][\![Q_{k_{1}}, Q_{k_{2}}]\!]$-basis. 

For an arbitrary subset $J \subset I$, let $\pi_{J} : G/B \to G/P_{J}$ be the natural projection, and let $(\pi_{J})_{*} : K_{T}(G/B) \to K_{T}(G/P_{J})$ denote the induced push-forward, which is $\bZ[\Lambda]$-linear. Also, it is well-known that $\pi_{J}([\mathcal{O}^{w}]) = [\mathcal{O}_{J}^{\lfloor w \rfloor}]$ for each $w \in W$, where $\lfloor w \rfloor$ denotes the minimal-length coset representative for the coset $w W_{J}$ in $W/W_{J}$, 
and that $\pi_{J}([\mathcal{O}_{G/B}(- \varpi_{k})]) = [\mathcal{O}_{G/P_{J}}(- \varpi_{k})]$ for $k \in K = I \setminus J$ (see, for example, \cite[Section~9.2]{mnsdem}). 
Now, we set 
$QK_{T}^{\mathrm{poly}}(G/B) := K_T(G/B) \otimes_{{\mathbb Z}[\Lambda]} \bZ[\Lambda][Q] \subset QK_{T}(G/B)$, and 
$QK_{T}^{\mathrm{poly}}(G/P_{J}) := K_T(G/P_{J}) \otimes_{{\mathbb Z}[\Lambda]} \bZ[\Lambda][Q_{K}]$, 
where $\bZ[\Lambda][Q_{K}]$ is the ring of polynomials with coefficients in $\bZ[\Lambda]$ in the (Novikov) variables $Q_k = Q^{\alpha_k^{\vee}}$, $k \in K := I \setminus J$. 
Based on the finiteness result on the quantum multiplication in $QK_{T}(G/P_{J})$ with the line bundle classes $[\mathcal{O}_{G/P_{J}}(-\varpi_k)]$ for $k \in K = I \setminus J$ (see also \cite{ACT}), 
Kato proved (see \cite{katqkg}) that the ($\bZ[\Lambda]$-linear) push-forward $(\pi_{J})_{*} : K_{T}(G/B) \to K_{T}(G/P_{J})$ induces a surjective $\bZ[\Lambda]$-module homomorphism $\Phi_{J} : QK_{T}^{\mathrm{poly}}(G/B) \to QK_{T}^{\mathrm{poly}}(G/P_{J})$ 
such that for $w \in W$ and $k \in K = I \setminus J$, the following equality holds: 
\begin{equation}
\Phi_{J}([\mathcal{O}^{w}] \cdot [\mathcal{O}_{G/B}(-\varpi_k)]) = [\mathcal{O}_{J}^{\lfloor w \rfloor}] \cdot [\mathcal{O}_{G/P_{J}}(-\varpi_k)], 
\end{equation}
by defining 
$\Phi_{J}(Q^{\xi}) := Q^{[\xi]^{J}}$ for each $\xi \in Q^{\vee,+}$, where $[\xi]^{J} := \sum_{k \in I \setminus J} c_{k} \alpha_{k}^{\vee}$ for $\xi = \sum_{i \in I} c_{i} \alpha_{i}^{\vee} \in Q^{\vee,+}$. 
Namely, Kato proved the following. 

\begin{thm}[\cite{katqkg}]\label{thm:qksurj}
Let $J$ be an arbitrary subset of $I$. Then, 
the surjective $\mathbb{Z}[\Lambda]$-module homomorphism 
\begin{equation}
\Phi_{J} : QK_{T}^{\mathrm{poly}}(G/B) \to QK_{T}^{\mathrm{poly}}(G/P_{J})
\end{equation}
defined by 
$\Phi_{J}(Q^{\xi}[\mathcal{O}^{w}]) = Q^{[\xi]^J}[\mathcal{O}_{J}^{\lfloor w \rfloor}]$ for $w \in W$ and $\xi \in Q^{\vee,+}$, where $[\xi]^{J} = \sum_{k \in I \setminus J} c_{k} \alpha_{k}^{\vee}$ for $\xi = \sum_{i \in I} c_{i} \alpha_{i}^{\vee} \in Q^{\vee,+}$, 
has the following multiplicativity:
\begin{equation}
\Phi_{J}([\mathcal{O}^{w}] \cdot [\mathcal{O}_{G/B}(-\varpi_k)]) = [\mathcal{O}_{J}^{\lfloor w \rfloor}] \cdot [\mathcal{O}_{G/P_{J}}(-\varpi_k)] 
\end{equation}
for $w \in W$ and $k \in K = I \setminus J$. 
\end{thm}

In Appendix~\ref{sec:A}, we give another proof of the existence of the multiplicative $\bZ[\Lambda]$-module surjection $\Phi_{J}$ above by using the $K$-Peterson homomorphism, which is a homomorphism of $\bZ[\Lambda]$-algebras from the $K$-homology of the affine Grassmannian associated to $G$ to (the localization, with respect to $Q^{\vee,+}$, of) the quantum $K$-ring $QK_{T}(G/P_{J})$; a (new) proof of the existence of the $K$-Peterson homomorphism has been given by \cite{newprf}. 

We now recall the (cancellation-free) quantum $K$-theory Chevalley formula in \cite[Theorem~47]{lnscfs} (see also \cite[Theorem~12]{lnsccf}) for $G/B$, which is based on the quantum alcove model; in fact, we use the slight modification corresponding to the multiplication by the class $[\mathcal{O}(-\varpi_{k})]:=[\mathcal{O}_{G/B}(- \varpi_{k})]$ of the line bundle associated to $-\varpi_k$. Throughout this paper, we denote by $|S|$ for a set $S$ the cardinality of $S$. This formula is expressed in terms of a $(-\varpi_k)$-chain of roots, cf. Section~\ref{sec:lamch}.

\begin{thm}\label{qkchev} Let $k \in I$, and fix a reduced $(-\varpi_k)$-chain $\Gamma(k)$. Then, in $QK_{T}^{\mathrm{poly}}(G/B) \subset QK_{T}(G/B)$, we have {for $w \in W$},
\begin{equation}\label{qkchev-f}[\cO(-\varpi_k)]\cdot [{\mathcal O}^w] = \sum_{A\in{\mathcal A}(w,\Gamma(k))}
(-1)^{|A|} \,Q^{\dn(w,A)}\be^{-\wt(w,A)} [{\mathcal O}^{\wend(w,A)}]\,.\end{equation}
\end{thm}

\begin{rema} \label{rem:Chevalley_cancellation-free}
The right-hand side of equation~\eqref{qkchev-f} is cancellation-free. 
Indeed, suppose, for a contradiction, that there exist two admissible subsets $A, A' \in \A(w, \Gamma(k))$ satisfying $\wend(w, A) = \wend(w, A')$ and $(-1)^{|A|} = -(-1)^{|A'|}$ (together with $\dn(w, A) = \dn(w, A')$ and $\wt(w, A) = \wt(w, A')$). Here we know (see \cite{bfpmbo} and also \cite{posqbg}) that for directed paths $\bp_{1}$, $\bp_{2}$ in $\QB(W)$ starting from the same element $v \in W$ and ending at the same element $u \in W$, the equality $(-1)^{\ell(\bp_{1})} = (-1)^{\ell(\bp_{2})}$ holds, where $\ell(\cdot)$ denotes the length of a directed path. This contradicts the equality $(-1)^{|A|} = -(-1)^{|A'|}$, as desired. 
\end{rema}

Let $N_{u, v}^{w, \xi} \in \bZ[P]$, with $v, w, u \in W^{J}$, $\xi \in Q_{I \setminus J}^{\vee, +} := \sum_{i \in I \setminus J} \bZ_{\ge 0} \alpha_{i}^{\vee}$, denote the \emph{structure constants} of $QK_{T}(G/P_{J})$ defined by: 
\begin{equation}
[\cO^{v}] \cdot [\cO^{w}] = \sum_{u \in W^{J}, \ \xi \in Q_{I \setminus J}^{\vee, +}} N_{v, w}^{u, \xi} Q^{\xi} [\cO^{u}]. 
\end{equation}
Let $\rho_{J}$ be a half of the sum of all positive roots of $P_{J}$, 
and set $\deg(Q^{\xi}) := 2\pair{\rho - \rho_{J}}{\xi}$ for $\xi \in Q_{I \setminus J}^{\vee, +}$. 
It is expected that the structure constants of $QK_{T}(G/P_{J})$ have the following \emph{positivity property}. 

\begin{conj}[{\cite[Conjecture~2.2]{bcmcfe}}]
For $v, w, u \in W^{J}$ and $\xi \in Q_{I \setminus J}^{\vee, +}$, we have 
\begin{equation}
(-1)^{\ell(v) + \ell(w) + \ell(u) + \deg(Q^{\xi})}N_{v, w}^{u, \xi} \in \bZ_{\ge 0}[\be^{\gamma} - 1 \mid \gamma \in -\Delta]. 
\end{equation}
\end{conj}

The positivity property of the structure constants $N_{s_k, w}^{u, \xi}$, with $k \in K = I \setminus J$, is proved for cominuscule varieties $G/P_{J}$, which include Grassmannians of type $A$, by Buch-Chaput-Mihalcea-Perrin in \cite{bcmcfe} by writing explicitly the structure constants. 
Also, the positivity property of the structure constants $N_{v, w}^{u, \xi}$, with $\xi = 0$, is proved by Anderson-Griffeth-Miller \cite{agmpkt} since these are the structure constants of the ordinary $T$-equivariant $K$-theory $K_{T}(G/P_{J})$. 
In this paper, we prove the positivity property of the structure constants $N_{s_k, w}^{u, \xi}$, with $k \in K = I \setminus J$, for full flag manifolds of arbitrary types, two-step flag manifolds of type $A$, and Grassmannians of type $C$. 

Let us define $C_{w}^{u, \xi} \in \bZ[P]$, with $w, u \in W$, $\xi \in Q_{I \setminus J}^{\vee, +}$, by: 
\begin{equation}
[\cO(-\vpi_{k})] \cdot [\cO^{w}] = \sum_{u \in W^{J}, \ \xi \in Q_{I \setminus J}^{\vee, +}} C_{w}^{u, \xi} Q^{\xi} [\cO^{u}]. 
\end{equation}
Since it is well-known that $[\cO^{s_{k}}] = 1 - \be^{-\vpi_{k}} [\cO(-\vpi_{k})]$ for $k \in I \setminus J$, we see that 
\begin{align}
[\cO^{s_{k}}] \cdot [\cO^{w}] &= (1 - \be^{-\vpi_{k}}[\cO(-\vpi_{k})]) \cdot [\cO^{w}] \\ 
&= [\cO^{w}] - \be^{-\vpi_{k}} [\cO(-\vpi_{k})] \cdot [\cO^{w}] \\ 
\begin{split}
&= (1 - \be^{-\vpi_{k}} C_{w}^{w, 0}) [\cO^{w}] + \sum_{\xi \in Q_{I \setminus J}^{\vee, +} \setminus \{0\}} (-\be^{-\vpi_{k}}C_{w}^{w, \xi}) Q^{\xi} [\cO^{w}] \\ 
& \quad + \sum_{u \in W \setminus \{w\}, \ \xi \in Q_{I \setminus J}^{\vee, +}} (-\be^{-\vpi_{k}}C_{w}^{u, \xi}) Q^{\xi} [\cO^{u}]. 
\end{split}
\end{align}
Hence it follows that for $w, u \in W^{J}$ and $\xi \in Q_{I \setminus J}^{\vee, +}$, 
\begin{equation}
N_{s_{k}, w}^{u, \xi} = \begin{cases}
1-\be^{-\vpi_{k}}C_{w}^{w, 0} & \text{if $u = w$ and $\xi = 0$}, \\ 
-\be^{-\vpi_{k}}C_{w}^{u, \xi} & \text{otherwise}. 
\end{cases}
\end{equation}

For the proof of the positivity property, we need the following lemma. 

\begin{lem} \label{lem:wt_belonging}
Let $w \in W$. Let $\lambda \in \Lambda$ be a dominant weight, and take a reduced $(-\lambda)$-chain $\Gamma$. 
For $A \in \A(w, \Gamma)$, we have $\wt(w, A) \in -\lambda + Q^{+}$. 
\end{lem}
\begin{proof}
Let $A \in \A(w, \Gamma)$. We denote by $\Lambda_{\af}^{0}$ the set of all level-zero weights of the (untwisted) affine Lie algebra $\Fg_{\af} = (\Fg \otimes \bC[t, t^{-1}]) \oplus \bC c \oplus \bC d$ associated to $\Fg$; 
in the following, we regard $\lambda$ as an element of $\Lambda_{\af}^{0}$. 

We use \emph{quantum Lakshmibai-Seshadri} (QLS) \emph{paths} of shape $\lambda$, which are defined in \cite[Definition~3.1]{lnsumk}. 
We first assume that $\Gamma$ is the \emph{lex} $(-\lambda)$-chain, defined in \cite[Section~4.2]{lnsccf}. 
In this case, we know from \cite[Proposition~31]{lnscfs} that 
there exists a QLS path $\eta$ of shape $\lambda$ such that $\wt(w, A) = -\wt(\eta)$, where $\wt(\eta) := \eta(1)$. 
Let us write $\eta$ in the form $\eta = (\nu_{1}, \ldots, \nu_{s}; 0 = a_{0} < a_{1} < \cdots < a_{s} = 1)$, with 
$\nu_{1}, \ldots, \nu_{s} \in W\lambda$ and $a_{0}, \ldots, a_{s} \in \bQ$. 
Then we see that $\nu_{k} \in \lambda - Q^{+}$, $k = 1, \ldots, s$, since $\lambda \in \Lambda$ is dominant and  $W$ is the finite Weyl group. 
Hence we have 
\begin{equation}
\wt(\eta) = \eta(1) = \sum_{k = 1}^{s} (a_{k} - a_{k-1}) \nu_{k} \in \lambda - \sum_{j \in I} \bQ_{\ge 0} \alpha_{j}. 
\end{equation}

Also, we have
\begin{equation}
\wt(\eta) = \eta(1) = \nu_{s} + \sum_{k = 1}^{s-1} a_{k}(\nu_{k} - \nu_{k+1}). 
\end{equation}
Since $(\nu_{k}, \nu_{k+1})$ is an $a_{k}$-chain (see \cite[Section~4]{litpro}), it follows that $a_{k}(\nu_{k} - \nu_{k+1}) \in Q$ 
for $k = 1, \ldots, s-1$.
In addition, we have that $\nu_{s} \in \lambda - Q^{+}$. 
Hence we see that $\wt(\eta) \in \lambda + Q$. 
Therefore, we deduce that $\wt(\eta) \in \lambda - Q^{+}$, as desired. 

We next assume that $\Gamma$ is an arbitrary reduced $(-\lambda)$-chain. 
Then we know that $\Gamma$ can be deformed to the lex $(-\lambda)$-chain $\Gamma'$ by repeated application of \emph{Yang-Baxter transformations} in \cite[Section~3.1]{klnnsq} (see also \cite[Remark~40]{lnsccf}). 
In this situation, \cite[Theorems~3.2 and 3.4]{klnnsq} implies that there exists a bijection $Y: \A(w, \Gamma) \rightarrow \A(w, \Gamma')$, given by \emph{quantum Yang-Baxter moves}, such that $\wt(w, Y(A)) = \wt(w, A)$ for all $A \in \A(w, \Gamma)$. 
Here we note that \cite[Theorem~3.2]{klnnsq} states that $Y$ is a \emph{sijection} (\cite[Section~2]{fakbpa}), i.e., a ``signed bijection'', where $\A(w, \Gamma)$ and $\A(w, \Gamma')$ are regarded as signed sets equipped with sign functions. However, since $-\lambda$ is anti-dominant, we have no sign-reversing involution on any non-empty subset of $\A(w, \Gamma)$ or $\A(w, \Gamma')$. Therefore, $Y$ is, in fact, a bijection. 
Since $\Gamma'$ is the lex $(-\lambda)$-chain and $Y(A) \in \A(w, \Gamma')$, we deduce that $\wt(w, A) = \wt(w, Y(A)) \in -\lambda + Q^{+}$. 
This proves the lemma. 
\end{proof}

Note that if there exists an edge $x \rightarrow y$ in $\QB(W)$ for $x, y \in W$, then we have $\ell(y) \equiv \ell(x) + 1 \mod 2$ by the definition of $\QB(W)$. 
This implies that for $w \in w$ and $A \in \A(w, \Gamma(k))$, we have $(-1)^{|A|} = (-1)^{\ell(\wend(w, A)) - \ell(w)}$. 

In this section, we prove the positivity property of structure constants for full flag manifolds as a corollary of the Chevalley formula (Theorem~\ref{qkchev}). 
We will give a proof of the positivity property for Grassmannians of type $C$ (resp., two-step flag manifolds of type $A$) in Section~\ref{sec:type_C} (resp., Section~\ref{sec:two-step_positivity}). 
\begin{cor} \label{cor:positivity_full}
Let $G$ be of an arbitrary type, $J = \emptyset$ (hence $P_{J} = B$), and $k \in I$. 
Then, for $w, u \in W^{J}$ and $\xi \in Q^{\vee, +}$, we have 
\begin{equation}
(-1)^{1 + \ell(w) + \ell(u) + \deg(Q^{\xi})}N_{s_{k}, w}^{u, \xi} \in \bZ_{\ge 0}[\be^{\gamma} - 1 \mid \gamma \in -\Delta].  
\end{equation}
\end{cor}

\begin{proof}
Let $w \in W$. Take $A \in \A(w, \Gamma(k))$ such that $A^{-} = \emptyset$. 
If $A = \emptyset$, then we have 
\begin{equation}
(-1)^{|A|}Q^{\dn(w, A)}\be^{-\wt(w, A)}[\cO^{\wend(w, A)}] = \be^{w\vpi_{k}} [\cO^{w}]. 
\end{equation} 
Since there exists no $A \in \A(w, \Gamma(k))$ such that $\wend(w, A) = w$ and $\dn(w, A) = 0$ except for $A = \emptyset$, we have $C_{w}^{w, 0} = \be^{w\vpi_{k}}$. In addition, we have $\deg(Q^{0}) = 0$. Hence it follows that 
\begin{equation}
N_{s_{k}, w}^{w, 0} = 1 - \be^{w\vpi_{k} - \vpi_{k}} = (-1)^{1 + \ell(w) + \ell(w) + \deg(Q^{0})} (\be^{w\vpi_{k} - \vpi_{k}} - 1); 
\end{equation}
note that $w\vpi_{k} - \vpi_{k} \in -Q^{+}$. 
Since 
\begin{align}
\be^{-\mu} &= \prod_{i \in I} (\be^{-\alpha_{i}})^{c_{i}} \\ 
&= \prod_{i \in I} ((\be^{-\alpha_{i}} - 1) + 1)^{c_{i}} \\ 
&= \prod_{i \in I} \left( \sum_{k = 0}^{c_{i}} \binom{c_{i}}{k} (\be^{-\alpha_{i}} - 1)^{k} \right) \in \bZ_{\ge 0}[\be^{\gamma} - 1 \mid \gamma \in -\Delta]
\end{align}
for $\mu = \sum_{i \in I} c_{i} \alpha_{i} \in Q^{+}$, we deduce that 
\begin{equation}
(-1)^{1 + \ell(w) + \ell(w) + \deg(Q^{0})} N_{s_{k}, w}^{w, 0} \in \bZ_{\ge 0}[\be^{\gamma}-1 \mid \gamma \in -\Delta], 
\end{equation}
as desired. 

Next, take $A \in \A(w, \Gamma(k)) \setminus \{\emptyset\}$. Then we have 
\begin{equation}
(-1)^{|A|}Q^{\dn(w, A)}\be^{-\wt(w, A)}[\cO^{\wend(w, A)}] = (-1)^{\ell(\wend(w, A)) - \ell(w)}Q^{\dn(w, A)}\be^{-\wt(w, A)}[\cO^{\wend(w, A)}]. 
\end{equation}
Also, by Lemma~\ref{lem:wt_belonging}, we have $\wt(w, A) \in -\vpi_{k} + Q^{+}$ for $A \in \A(w, \Gamma(k))$. 
Here we set 
\begin{equation}
\A(w, \Gamma(k))_{u, \xi, \lambda} := \{A \in \A(w, \Gamma(k)) \mid \wend(w, A) = u, \ \dn(w, A) = \xi, \ \wt(w, A) = \lambda\}, 
\end{equation}
for $u \in W$, $\xi \in Q^{\vee, +}$, and $\lambda \in -\vpi_{k}+Q^{+}$.
Then by Theorem~\ref{qkchev}, we have 
\begin{align}
C_{w}^{u, \xi} &= \sum_{\lambda \in -\vpi_{k}+Q^{+}} \sum_{A \in \A(w, \Gamma(k))_{u, \xi, \lambda}} (-1)^{|A|} \be^{-\wt(w, A)} \\ 
&= (-1)^{\ell(u) - \ell(w)} \sum_{\lambda \in -\vpi_{k}+Q^{+}} |\A(w, \Gamma(k))_{u, \xi, \lambda}| \be^{-\lambda}. 
\end{align}
Since $\deg(Q_{j}) = 2\pair{\rho}{\alpha_{j}^{\vee}} = 2$ for all $j \in I$, we have $\deg(Q^{\xi}) \in 2\bZ$. 
Therefore, we see that 
\begin{align}
N_{s_{k}, w}^{u, \xi} &= -\be^{-\vpi_{k}} \cdot (-1)^{\ell(u) - \ell(w)} \sum_{\lambda \in -\vpi_{k}+Q^{+}} |\A(w, \Gamma(k))_{u, \xi, \lambda}| \be^{-\lambda}. \\ 
&= (-1)^{1 + \ell(w) + \ell(u) + \deg(Q^{\xi})} \sum_{\lambda \in -\vpi_{k}+Q^{+}} |\A(w, \Gamma(k))_{u, \xi, \lambda}| \be^{-\vpi_{k} - \lambda}. 
\end{align}
This implies that 
\begin{equation}
(-1)^{1 + \ell(w) + \ell(u) + \deg(Q^{\xi})} N_{s_{k}, w}^{u, \xi} \in \bZ_{\ge 0}[\be^{\gamma}-1 \mid \gamma \in -\Delta], 
\end{equation}
as desired. This proves the corollary. 
\end{proof}

\section{Quantum $K$-theory Chevalley formulas in the maximal parabolic case}\label{sec:max_parabolic}

Given a maximal parabolic subgroup $P_J$ for $J=I\setminus\{k\}$, we will derive cancellation-free parabolic Chevalley formulas for the quantum multiplication in $QK_T(G/P_J)$ with $[\mathcal{O}(-\varpi_{k})]:=[\mathcal{O}_{G/P_{J}}(- \varpi_{k})]$. Based on Theorem~\ref{thm:qksurj} explained in Section~\ref{sec:qk}, we obtain certain formulas from equation~\eqref{qkchev-f} in Theorem~\ref{qkchev} for $QK_{T}^{\mathrm{poly}}(G/B) \subset QK_{T}(G/B)$ by applying $\Phi_{J}$;
this argument works for an arbitrary fundamental weight $\varpi_{k}$ of $G$ of any type. 
However, upon applying $\Phi_{J}$, there are many terms to be canceled in the corresponding formula in $QK_{T}^{\mathrm{poly}}(G/P_{J}) \subset QK_{T}(G/P_{J})$.
For any fundamental weight $\varpi_{k}$ in types $A$ and $C$, we cancel out all these terms via a sign-reversing involution, and obtain a cancellation-free formula. We rely on the structure of the corresponding $(-\varpi_k)$-chain of roots $\Gamma(k)$ in Section~\ref{sec:lamch}, as well as the quantum Bruhat graph criteria in Section~\ref{sec:qbg}.

\begin{rema}\label{nocancel}
Upon applying the above procedure, there are no cancellations among the terms corresponding to $w$-admissible subsets $A$ with $A^-=\emptyset$, by Remark~\ref{rem:Chevalley_cancellation-free}.  
\end{rema}

\begin{rema}
If $G$ is of type $A_{n-1}$, then the partial flag manifold $G/P_{J}$ for $J = I \setminus \{k\}$ is isomorphic to the \emph{Grassmannian} $\Gr(k, n)$ defined as: 
\begin{equation}
\Gr(k, n) := \{ V \mid \text{$V$ is a subspace of $\bC^{n}$ such that $\dim V = k$} \}. 
\end{equation}

Also, if $G$ is of type $C_{n}$, then the partial flag manifold $G/P_{J}$ for $J = I \setminus \{k\}$ is isomorphic to the \emph{isotropic Grassmannian} $\IG(k, 2n)$ defined as: 
\begin{equation}
\IG(k, 2n) := \left\{ V \ \middle| \ \parbox{21em}{$V$ is a subspace of $\bC^{2n}$ such that $\dim V = k$, and $V$ is isotropic with respect to $(-,-)$} \right\}; 
\end{equation}
where $(-,-)$ denotes a non-degenerate skew symmetric bilinear form on $\bC^{2n}$. 
\end{rema}

\subsection{Type $A_{n-1}$}

We start with type $A_{n-1}$, and we fix the anti-dominant fundamental weight $-\varpi_k$. Note that $w\in W^J$ is equivalent to $w[1,k]$ and $w[k+1,n]$ being increasing sequences.

\begin{lem}\label{paredge} Consider $w\in W^J$. We have an edge $w \stackrel{(i,j)}{\longrightarrow} w(i,j)$ in the quantum Bruhat graph on $S_n$, with $i\le k<j$, if and only if one of the following two conditions holds:
\begin{enumerate}
\item the edge is a Bruhat cover, with $w(i)=a$, $w(j)=a+1$, and $w(i,j)\in W^J$;
\item the edge is a quantum one, and $(i,j)=\alpha_k$. 
\end{enumerate}
\end{lem}

\begin{proof} We implicitly use several times the quantum Bruhat graph criterion in Proposition~\ref{prop:quantum_bruhat_order_type_A}, as well as the fact that $w[1,k]$ and $w[k+1,n]$ are increasing sequences. Letting $a:=w(i)$, and assuming that the edge is a Bruhat cover, we cannot have $w(j)>a+1$ because the value $a+1$ would be straddled by the transposition $(i,j)$. Indeed, this would happen irrespective of $a+1$ being in $w[1,k]$ or $w[k+1,n]$. So we must have $w(j)=a+1$. Now assume that $w(i)>w(j)$. If $i<k$, then the value $w(k)$ would be straddled, while if $j>k+1$, then the value $w(k+1)$ would be straddled. So we must have $i=k$ and $j=k+1$.
\end{proof}

We can now give a short proof of \cite[Theorem~I]{knscfa} in type $A_{n-1}$, which is restated below in terms of the quantum alcove model. 

\begin{thm}\label{qkchev-a} In type $A_{n-1}$, consider $1\le k\le n-1$ and $w\in W^J$.

{\rm (1)} If $w\ge\lfloor s_\theta\rfloor$, then we have the following cancellation-free formula:
\begin{equation}\label{typea1-chev}[\cO(-\varpi_k)]\cdot[\cO^w]=
\mathbf{e}^{w\varpi_k}\sum_{A\in\A_{\lessdot}(w,\Gamma(k))}(-1)^{|A|}\left([\cO^{\wend(w,A)}]-Q_k[\cO^{\lfloor\wend(w,A)s_{k}\rfloor}]\right)\,.\end{equation}

{\rm (2)} If $w\not\ge\lfloor s_\theta\rfloor$, then we have the following cancellation-free formula:
\begin{equation}\label{typea2-chev}[\cO(-\varpi_k)]\cdot[\cO^w]=
\mathbf{e}^{w\varpi_k}\sum_{A\in\A_{\lessdot}(w,\Gamma(k))}(-1)^{|A|}[\cO^{\wend(w,A)}]\,.\end{equation}
\end{thm}

\begin{rema}
As will be seen in the proof below, for $w \in W^{J}$, the condition $w \geq \mcr{s_{\theta}}$ is equivalent to the condition $w(k) = n$ and $w(k+1) = 1$. 
\end{rema}

\begin{exa}
We give some examples of the Chevalley formula in the case that $n = 4$ and $k = 2$. Note that $\mcr{s_{\theta}} = s_{3}s_{1}s_{2}$. Also, we have $\Gamma(2) = ((1,4), (1,3), (2,4), (2,3))$ (with all roots negated). 
\begin{enumerate}

\item Let $w = s_{3}s_{1}s_{2} = \mcr{s_{\theta}}$. Table~\ref{tab:exm_typeA_Grassmannian} is the list of all admissible subsets $A \in \A(w, \Gamma(2))$ and their statistics $\wend(w, A)$, $\dn(w, A)$, together with $\mcr{\wend(w, A)}$; note that $\wt(w, A) = -s_{2}\vpi_{2}$ for all $A \in \A(w, \Gamma(2))$. 
\begin{table}[ht]
\caption{The list of all admissible subsets $A \in \A(s_{3}s_{1}s_{2}, \Gamma(2))$}
\label{tab:exm_typeA_Grassmannian}
\centering
\begin{tabular}{|c|ccc|} \hline
$A$ & $\wend(w, A)$ & $\mcr{\wend(w, A)}$ & $\dn(w, A)$ \\ \hline  
$\emptyset$ & $s_{3}s_{1}s_{2}$ & $s_{3}s_{1}s_{2}$ & $0$ \\ 
$\{1\}$ & $s_{2}s_{3}s_{1}s_{2}$ & $s_{2}s_{3}s_{1}s_{2}$ & $0$ \\ 
$\{4\}$ & $s_{3}s_{1}$ & $e$ & $\alpha_{2}^{\vee}$ \\ 
$\{1, 4\}$ & $s_{2}s_{3}s_{1}$ & $s_{2}$ & $\alpha_{2}^{\vee}$ \\ \hline
\end{tabular}
\end{table}

By Theorem~\ref{qkchev}, in $QK_{T}^{\mathrm{poly}}(G/B)$, we have: 
\begin{equation} \label{eq:Chevalley_full_example}
[\cO(-\vpi_{2})] \cdot [\cO^{s_{2}s_{3}s_{1}s_{2}}] = \be^{s_{2}s_{3}s_{1}s_{2}\vpi_{2}} \left( [\cO^{s_{3}s_{1}s_{2}}] - [\cO^{s_{2}s_{3}s_{1}s_{2}}] - Q_{2} [\cO^{s_{3}s_{1}}] + Q_{2} [\cO^{s_{2}s_{3}s_{1}}] \right)
\end{equation}
By applying the surjection $\Phi_{J}: QK_{T}^{\mathrm{poly}}(G/B) \rightarrow QK_{T}^{\mathrm{poly}}(G/P_{J})$, explained in Theorem~\ref{thm:qksurj}, to equation~\eqref{eq:Chevalley_full_example}, we obtain the following cancellation-free formula in $QK_{T}^{\mathrm{poly}}(G/P_{J}) \subset QK_{T}(G/P_{J})$: 
\begin{equation}
[\cO(-\vpi_{2})] \cdot [\cO^{s_{2}s_{3}s_{1}s_{2}}] = \be^{s_{2}s_{3}s_{1}s_{2}\vpi_{2}} \left( [\cO^{s_{3}s_{1}s_{2}}] - [\cO^{s_{2}s_{3}s_{1}s_{2}}] - Q_{2} [\cO^{e}] + Q_{2} [\cO^{s_{2}}] \right). 
\end{equation}

Also, we deduce that $\A_{\lessdot}(w, \Gamma(2)) = \{\emptyset, \{1\}\}$. Therefore, we see that 
\begin{align}
& \text{(RHS of equation~\eqref{typea1-chev})} \\ 
&= \be^{s_{2}s_{3}s_{1}s_{2}\vpi_{2}} \left( \left( [\cO^{s_{3}s_{1}s_{2}}] - Q_{2} [\cO^{\mcr{s_{3}s_{1}}}] \right) - \left( [\cO^{s_{2}s_{3}s_{1}s_{2}}] - Q_{2} [\cO^{\mcr{s_{2}s_{3}s_{1}}}] \right) \right) \\ 
&= \be^{s_{2}s_{3}s_{1}s_{2}\vpi_{2}} \left( [\cO^{s_{3}s_{1}s_{2}}] - Q_{2} [\cO^{e}] - [\cO^{s_{2}s_{3}s_{1}s_{2}}] + Q_{2} [\cO^{s_{2}}] \right) \\ 
&= [\cO(-\vpi_{2})] \cdot [\cO^{s_{2}s_{3}s_{1}s_{2}}]. 
\end{align}
Thus Theorem~\ref{qkchev-a}\,(1) holds in this case. 

\item Let $w = s_{2}$; note that $w \not\ge \mcr{s_{\theta}}$. Then we can give the list of all admissible subsets $A \in \A(w, \Gamma(2))$ and their statistics $\wend(w, A)$, $\dn(w, A)$, together with $\mcr{\wend(w, A)}$, as in Table~\ref{tab:exm_typeA_Grassmannian2}. 
Note that $\wt(w, A) = -s_{2}\vpi_{2}$ for all $A \in \A(w, \Gamma(2))$. 

\begin{table}[ht]
\caption{The list of all admissible subsets $A \in \A(s_{2}, \Gamma(2))$}
\label{tab:exm_typeA_Grassmannian2}
\centering
\begin{tabular}{|c|ccc|} \hline
$A$ & $\wend(w, A)$ & $\mcr{\wend(w, A)}$ & $\dn(w, A)$ \\ \hline
$\emptyset$ & $s_{2}$ & $s_{2}$ & $0$ \\ 
$\{2\}$ & $s_{1}s_{2}$ & $s_{1}s_{2}$ & $0$ \\ 
$\{3\}$ & $s_{3}s_{2}$ & $s_{3}s_{2}$ & $0$ \\ 
$\{4\}$ & $e$ & $e$ & $\alpha_{2}^{\vee}$ \\ 
$\{2, 3\}$ & $s_{3}s_{1}s_{2}$ & $s_{3}s_{1}s_{2}$ & $0$ \\ 
$\{2, 4\}$ & $s_{1}$ & $e$ & $\alpha_{2}^{\vee}$ \\ 
$\{3, 4\}$ & $s_{3}$ & $e$ & $\alpha_{2}^{\vee}$ \\ 
$\{2, 3, 4\}$ & $s_{3}s_{1}$ & $e$ & $\alpha_{2}^{\vee}$ \\ \hline
\end{tabular}
\end{table}

By Theorem~\ref{qkchev}, in $QK_{T}^{\mathrm{poly}}(G/B)$, we have: 
\begin{equation} \label{eq:Chevalley_full_example2}
\begin{split}
[\cO(-\vpi_{2})] \cdot [\cO^{s_{2}}] &= \be^{s_{2}\vpi_{2}} \left( [\cO^{s_{2}}] - [\cO^{s_{1}s_{2}}] - [\cO^{s_{3}s_{2}}] - Q_{2} [\cO^{e}] \right. \\ 
& \left. \quad {}+ [\cO^{s_{3}s_{1}s_{2}}] + Q_{2} [\cO^{s_{1}}] + Q_{2} [\cO^{s_{3}}] - Q_{2} [\cO^{s_{3}s_{1}}] \right). 
\end{split}
\end{equation}
By applying the surjection $\Phi_{J}: QK_{T}^{\mathrm{poly}}(G/B) \rightarrow QK_{T}^{\mathrm{poly}}(G/P_{J})$ to equation~\eqref{eq:Chevalley_full_example2}, we obtain the following cancellation-free formula in $QK_{T}^{\mathrm{poly}}(G/P_{J}) \subset QK_{T}(G/P_{J})$; here, the underlined terms in the first equality are canceled out: 
\begin{align}
\begin{split}
[\cO(-\vpi_{2})] \cdot [\cO^{s_{2}}] &= \be^{s_{2}\vpi_{2}} ([\cO^{s_{2}}] - [\cO^{s_{1}s_{2}}] - [\cO^{s_{3}s_{2}}] \underline{{}- Q_{2} [\cO^{e}]} \\ 
& \quad + [\cO^{s_{3}s_{1}s_{2}}] \underline{{}+ Q_{2} [\cO^{e}]} \underline{{}+ Q_{2} [\cO^{e}]} \underline{{}- Q_{2} [\cO^{e}]}) 
\end{split} \\ 
\begin{split}
&= \be^{s_{2}\vpi_{2}} ([\cO^{s_{2}}] - [\cO^{s_{1}s_{2}}] - [\cO^{s_{3}s_{2}}] + [\cO^{s_{3}s_{1}s_{2}}]).  
\end{split}
\end{align}

Also, we deduce that $\A_{\lessdot}(w, \Gamma(2)) = \{ \emptyset, \{2\}, \{3\}, \{2, 3\}\}$. Therefore, we see that 
\begin{align}
\text{(RHS of equation~\eqref{typea2-chev})} &= \be^{s_{2}\vpi_{2}} \left( [\cO^{s_{2}}] - [\cO^{s_{1}s_{2}}] - [\cO^{s_{3}s_{2}}] + [\cO^{s_{3}s_{1}s_{2}}] \right). \\ 
&= [\cO(-\vpi_{2})] \cdot [\cO^{s_{2}}]. 
\end{align}
Thus Theorem~\ref{qkchev-a}\,(2) holds in this case. 
\end{enumerate}
\end{exa}

\begin{proof}[Proof of Theorem~\ref{qkchev-a}] The result is clear when $w$ is the identity (indeed, a $w$-admissible subset is either empty or consists only of the transposition $(k,k+1)$); so we can assume that $w(k)>w(k+1)$. 

Let $A$ be a generic $w$-admissible subset in $\A(w,\Gamma(k))$. Given the structure of the $(-\varpi_k)$-chain $\Gamma(k)$ in~\eqref{omegakchain} and Lemma~\ref{paredge}, we can see that a quantum step in a path $\Pi(w,A)$ must correspond to the transposition $\alpha_k=(k,k+1)$, which is the last one in $\Gamma(k)$. All other steps are Bruhat covers of the form specified in Lemma~\ref{paredge}~(1). Moreover, the structure of $\Gamma(k)$ combined with the fact that $w\in W^J$ imply that $A$ contains at most one root labeling a Bruhat cover in $\Pi(w,A)$ of the following forms: $(i,\,\cdot\,)$ for each $i\le k$, and $(\,\cdot\,,j)$ for each $j>k$. All these facts will be used implicitly.

By Deodhar's criterion for the Bruhat order on the symmetric group~\cite[Theorem~2.6.3]{babccg}, we can see that $w\ge\lfloor s_\theta\rfloor=[2, 3, \ldots, k, n, 1, k+1, \ldots, n-1]$ (in one-line notation) if and only if $w(k)=n$ and $w(k+1)=1$. Thus, we consider the following cases; whenever there are terms to be canceled, we describe the sign-reversing involution mentioned above.

\emph{Case} 1: $w(k)<n$. Let $q>k+1$ be such that $w(q)=w(k)+1\le n$. We pair every $A$ containing $(k,k+1)$, but not $(k,j)$ with $j>k+1$, with $A':=A\cup\{(k,q)\}$. It is clear that $A'$ is also $w$-admissible, and in fact the root $(k,q)$ is the predecessor of $(k,k+1)$ in $A'$. Moreover, we have
\[\lfloor{\rm end}(w,A')\rfloor=\lfloor\cdots(k,q)(k,k+1)\rfloor=\lfloor\cdots(k,k+1)(k+1,q)\rfloor=\lfloor{\rm end}(w,A)\rfloor\,,\]
as well as ${\rm down}(w,A)={\rm down}(w,A')$ and ${\rm wt}(w,A)={\rm wt}(w,A')$. 
The latter property is a consequence of the fact that all the affine reflections in the definition of ${\rm wt}(w,\,\cdot\,)$ in \cite[Equation~(12)]{lnscfs} fix $\varpi_k$; for more details, see~\cite[Corollary~8.2]{lapawg} and the discussion preceding it. Finally, as the cardinalities of $A$ and $A'$ differ by $1$, their contributions to the parabolic Chevalley formula for $G/P_{J}$ have opposite signs. We have thus proved that the involution $A\leftrightarrow A'$ is sign-reversing. 

\emph{Case} 2: $w(k)=n$ and $w(k+1)>1$. Let $p< k$ be such that $w(p)=w(k+1)-1\ge 1$. We pair every $A$ containing $(k,k+1)$, but not $(i,k+1)$ with $i< k$, with $A':=A\cup\{(p,k+1)\}$. 
We continue the reasoning like in Case~1.

\emph{Case} 3: $w(k)=n$ and $w(k+1)=1$. It is clear that no $w$-admissible subset $A$ can contain transpositions of the form $(i,k+1)$ with $i<k$, and $(k,j)$ with $j>k+1$. Furthermore, there is a 2-to-1 correspondence between $\A(w,\Gamma(k))$ and $\A_\lessdot(w,\Gamma(k))$: every $A\in\A_\lessdot(w,\Gamma(k))$ corresponds to itself and $A\cup\{(k,k+1)\}$. Like above, we can check that $\wt(w,A)=\wt(w,A\cup\{(k,k+1)\})$. Finally, based on the above facts and Remark~\ref{nocancel}, we can see that there are no cancellations of terms corresponding to the elements of either $A_\lessdot(w,\Gamma(k))$ or $\A(w,\Gamma(k))\setminus A_\lessdot(w,\Gamma(k))$.

It is now easy to see that the uncanceled terms in the resulting combinatorial formula are precisely those in \eqref{typea1-chev} in Case~3, and those in \eqref{typea2-chev} in Cases~1 and~2.
\end{proof}

\subsection{Type $C_n$} \label{sec:type_C}

As we move beyond type $A$, we note that the following analogue of Lemma~\ref{paredge} exists: \cite[Lemma~5.1]{knscfa} for any simply laced type and $\varpi_k$ minuscule. Below we present the corresponding result in type $C_n$, which works for any $\varpi_k$; this result is easily proved based on the quantum Bruhat graph criterion in Section~\ref{sec:qbg}.

\begin{lem}\label{paredgec} Consider $1\le k\le n$ and $w\in W^J$ in type $C_n$. We have a quantum edge $w \stackrel{\alpha}{\longrightarrow} ws_\alpha$ in $\QB(W)$, with $\alpha\in\Phi^+\setminus\Phi_J^+$, if and only if $w\ne e$ and one of the following two conditions holds:
\begin{enumerate}
\item $\alpha=\alpha_k$;
\item $\alpha=(k,\overline{k})$, $w(k)=\overline{a}$ for $1\le a\le n$, and $w[k+1,n]\subseteq\{a+1,\ldots,n\}$ if $k<n$. 
\end{enumerate}
\end{lem}

Let us now turn to a short proof in the case of $\varpi_k$ in type $C_n$, where $1\le k\le n$. Note that $w\in W^J$ is equivalent to $w[1,k]$ and $w[k+1,n]$ being increasing sequences (with respect to the total order on $[\overline{n}]$), as well as $w[k+1,n]$ consisting of positive entries. We also need to introduce more notation. The $(-\varpi_k)$-chain $\Gamma(k)$ in \eqref{omegachain-c} has an obvious splitting $\Gamma(k)=\Gamma^1(k)\Gamma^2(k)$, where $\Gamma^1(k):=\Gamma_2'\cdots\Gamma_k'$ and $\Gamma^2(k):=\Gamma_1(k)\cdots\Gamma_k(k)$. This induces a splitting $A=A^1\sqcup A^2$ of any $w$-admissible subset $A$, where $A^i=A\cap\Gamma^i(k)$, for $i=1,2$. 

\begin{thm}\label{qkchev-c} In type $C_n$, given $w\in W^J$, we have the following cancellation-free formula in $QK_{T}^{\mathrm{poly}}(G/P_{J}) \subset QK_{T}(G/P_{J})$:
\begin{align}\label{typec-chev}[\cO(-\varpi_k)]\cdot[\cO^w]&=
\sum_{A\in\A_{\lessdot}(w,\Gamma(k))}(-1)^{|A|}\mathbf{e}^{-\wt(w,A)}[\cO^{\wend(w,A)}]\\
&-Q_k\sum_{\substack{A\in\A_{\lessdot}(w,\Gamma(k)) \\ \wend(w,A^1)\ge\lfloor s_\theta\rfloor}}(-1)^{|A|}\mathbf{e}^{-\wt(w,A)}[\cO^{\lfloor\wend(w,A)s_{2\varepsilon_k}\rfloor}]\,.\nonumber
\end{align}
\end{thm}

\begin{rema}
As will be seen in the proof below, for $w \in W^{J}$, the condition $w \geq \mcr{s_{\theta}}$ is equivalent to the condition $w(k) = \overline{1}$. 
\end{rema}

\begin{exa}
In this example, we consider the case that $n = 3$ and $k = 2$. Note that $\mcr{s_{\theta}} = s_{1}s_{2}s_{3}s_{2}$. Recall that $\Gamma(2) = ((1,\overline{2}), (1, \overline{3}), (1, \overline{1}), (1, 3), (1, \overline{2}), (2, \overline{3}), (2, \overline{2}), (2, 3))$ (with all roots negated). 
Let $w = s_{2}s_{3}s_{2}$. Then the list of all admissible subsets $A \in \A(w, \Gamma(2))$ and their statistics $\wend(w, A)$, $\dn(w, A)$, together with $\wend(w, A^{1})$, $\mcr{\wend(w, A)}$, is given in Table~\ref{tab:exm_typeC_Grassmannian}. Note that $\wt(w, A) = -s_{2}s_{3}s_{2}\vpi_{2}$ for all $A \in \A(w, \Gamma(2))$. 
\begin{table}[ht]
\caption{The list of all admissible subsets $A \in \A(s_{2}s_{3}s_{2}, \Gamma(2))$}
\label{tab:exm_typeC_Grassmannian}
\centering
\begin{tabular}{|c|cccc|} \hline 
$A$ & $\wend(w, A^{1})$ & $\wend(w, A)$ & $\mcr{\wend(w, A)}$ & $\dn(w, A)$ \\ \hline
$\emptyset$ & $s_{2}s_{3}s_{2}$ & $s_{2}s_{3}s_{2}$ & $s_{2}s_{3}s_{2}$ & $0$ \\ 
$\{1\}$ & $s_{1}s_{2}s_{3}s_{2}$ & $s_{1}s_{2}s_{3}s_{2}$ & $s_{1}s_{2}s_{3}s_{2}$ & $0$ \\ 
$\{4\}$ & $s_{2}s_{3}s_{2}$ & $s_{2}s_{3}s_{1}s_{2}$ & $s_{2}s_{3}s_{1}s_{2}$ & $0$ \\ 
$\{5\}$ & $s_{2}s_{3}s_{2}$ & $s_{1}s_{2}s_{3}s_{2}$ & $s_{1}s_{2}s_{3}s_{2}$ & $0$ \\ 
$\{7\}$ & $s_{2}s_{3}s_{2}$ & $e$ & $e$ & $\alpha_{2}^{\vee} + \alpha_{3}^{\vee}$ \\ 
$\{8\}$ & $s_{2}s_{3}s_{2}$ & $s_{2}s_{3}$ & $s_{2}$ & $\alpha_{2}^{\vee}$ \\ 
$\{1, 4\}$ & $s_{1}s_{2}s_{3}s_{2}$ & $s_{1}s_{2}s_{3}s_{1}s_{2}$ & $s_{1}s_{2}s_{3}s_{1}s_{2}$ & $0$ \\ 
$\{1, 7\}$ & $s_{1}s_{2}s_{3}s_{2}$ & $s_{1}$ & $e$ & $\alpha_{2}^{\vee} + \alpha_{3}^{\vee}$ \\ 
$\{1, 8\}$ & $s_{1}s_{2}s_{3}s_{2}$ & $s_{1}s_{2}s_{3}$ & $s_{1}s_{2}$ & $\alpha_{2}^{\vee}$ \\ 
$\{4, 6\}$ & $s_{2}s_{3}s_{2}$ & $s_{1}s_{2}s_{3}s_{1}s_{2}$ & $s_{1}s_{2}s_{3}s_{1}s_{2}$ & $0$ \\ 
$\{4, 8\}$ & $s_{2}s_{3}s_{2}$ & $s_{2}s_{3}s_{1}$ & $s_{2}$ & $\alpha_{2}^{\vee}$ \\ 
$\{5, 7\}$ & $s_{2}s_{3}s_{2}$ & $s_{1}$ & $e$ & $\alpha_{2}^{\vee} + \alpha_{3}^{\vee}$ \\ 
$\{5, 8\}$ & $s_{2}s_{3}s_{2}$ & $s_{1}s_{2}s_{3}$ & $s_{1}s_{2}$ & $\alpha_{2}^{\vee}$ \\ 
$\{7, 8\}$ & $s_{2}s_{3}s_{2}$ & $s_{2}$ & $s_{2}$ & $\alpha_{2}^{\vee} + \alpha_{3}^{\vee}$ \\ 
$\{1, 4, 7\}$ & $s_{1}s_{2}s_{3}s_{2}$ & $s_{2}s_{1}$ & $s_{2}$ & $\alpha_{2}^{\vee} + \alpha_{3}^{\vee}$ \\ 
$\{1, 4, 8\}$ & $s_{1}s_{2}s_{3}s_{2}$ & $s_{1}s_{2}s_{3}s_{1}$ & $s_{1}s_{2}$ & $\alpha_{2}^{\vee}$ \\ 
$\{1, 7, 8\}$ & $s_{1}s_{2}s_{3}s_{2}$ & $s_{1}s_{2}$ & $s_{1}s_{2}$ & $\alpha_{2}^{\vee} + \alpha_{3}^{\vee}$ \\ 
$\{4, 6, 7\}$ & $s_{2}s_{3}s_{2}$ & $s_{2}s_{1}$ & $s_{2}$ & $\alpha_{2}^{\vee} + \alpha_{3}^{\vee}$ \\ 
$\{4, 6, 8\}$ & $s_{2}s_{3}s_{2}$ & $s_{1}s_{2}s_{3}s_{1}$ & $s_{1}s_{2}$ & $\alpha_{2}^{\vee}$ \\ 
$\{5, 7, 8\}$ & $s_{2}s_{3}s_{2}$ & $s_{1}s_{2}$ & $s_{1}s_{2}$ & $\alpha_{2}^{\vee} + \alpha_{3}^{\vee}$ \\ 
$\{1, 4, 7, 8\}$ & $s_{1}s_{2}s_{3}s_{2}$ & $s_{1}s_{2}s_{1}$ & $s_{1}s_{2}$ & $\alpha_{2}^{\vee} + \alpha_{3}^{\vee}$ \\ 
$\{4, 6, 7, 8\}$ & $s_{2}s_{3}s_{2}$ & $s_{1}s_{2}s_{1}$ & $s_{1}s_{2}$ & $\alpha_{2}^{\vee} + \alpha_{3}^{\vee}$ \\ \hline
\end{tabular}
\end{table}

By Theorem~\ref{qkchev}, in $QK_{T}^{\mathrm{poly}}(G/B)$, we have: 
\begin{equation} \label{eq:Chevalley_full_typeC}
\begin{split}
& [\cO(-\vpi_{2})] \cdot [\cO^{s_{2}s_{3}s_{2}}] \\ 
&= \be^{s_{2}s_{3}s_{2}\vpi_{2}} ([\cO^{s_{2}s_{3}s_{2}}] - [\cO^{s_{1}s_{2}s_{3}s_{2}}] - [\cO^{s_{2}s_{3}s_{1}s_{2}}] - [\cO^{s_{1}s_{2}s_{3}s_{2}}] - Q_{2}Q_{3}[\cO^{e}] - Q_{2}[\cO^{s_{2}s_{3}}] \\ 
& \quad + [\cO^{s_{1}s_{2}s_{3}s_{1}s_{2}}] + Q_{2}Q_{3}[\cO^{s_{1}}] + Q_{2}[\cO^{s_{1}s_{2}s_{3}}] + [\cO^{s_{1}s_{2}s_{3}s_{1}s_{2}}] + Q_{2}[\cO^{s_{2}s_{3}s_{1}}] + Q_{2}Q_{3}[\cO^{s_{1}}] \\ 
& \quad + Q_{2}[\cO^{s_{1}s_{2}s_{3}}] + Q_{2}Q_{3}[\cO^{s_{2}}] - Q_{2}Q_{3}[\cO^{s_{2}s_{1}}] - Q_{2}[\cO^{s_{1}s_{2}s_{3}s_{1}}] - Q_{2}Q_{3}[\cO^{s_{1}s_{2}}] \\ 
& \quad - Q_{2}Q_{3}[\cO^{s_{2}s_{1}}] - Q_{2}[\cO^{s_{1}s_{2}s_{3}s_{1}}] - Q_{2}Q_{3}[\cO^{s_{1}s_{2}}] + Q_{2}Q_{3}[\cO^{s_{1}s_{2}s_{1}}] + Q_{2}Q_{3}[\cO^{s_{1}s_{2}s_{1}}]). 
\end{split}
\end{equation}
By applying the surjection $\Phi_{J}: QK_{T}^{\mathrm{poly}}(G/B) \rightarrow QK_{T}^{\mathrm{poly}}(G/P_{J})$ to equation~\eqref{eq:Chevalley_full_typeC}, we obtain the following cancellation-free formula in $QK_{T}^{\mathrm{poly}}(G/P_{J}) \subset QK_{T}(G/P_{J})$; here, the underlined terms in the first equality are canceled out: 
\begin{align}
\begin{split}
& [\cO(-\vpi_{2})] \cdot [\cO^{s_{2}s_{3}s_{2}}] \\ 
&= \be^{s_{2}s_{3}s_{2}\vpi_{2}} ([\cO^{s_{2}s_{3}s_{2}}] - [\cO^{s_{1}s_{2}s_{3}s_{2}}] - [\cO^{s_{2}s_{3}s_{1}s_{2}}] - [\cO^{s_{1}s_{2}s_{3}s_{2}}] \underline{{}- Q_{2}[\cO^{e}]} \underline{{}- Q_{2}[\cO^{s_{2}}]} \\ 
& \quad + [\cO^{s_{1}s_{2}s_{3}s_{1}s_{2}}] \underline{{}+ Q_{2}[\cO^{e}]} \underline{{}+ Q_{2}[\cO^{s_{1}s_{2}}]} + [\cO^{s_{1}s_{2}s_{3}s_{1}s_{2}}] \underline{{}+ Q_{2}[\cO^{s_{2}}]} + Q_{2}[\cO^{e}] \\ 
& \quad \underline{{}+ Q_{2}[\cO^{s_{1}s_{2}}]} \underline{{}+ Q_{2}[\cO^{s_{2}}]} \underline{{}- Q_{2}[\cO^{s_{2}}]} \underline{{}- Q_{2}[\cO^{s_{1}s_{2}}]} \underline{{}- Q_{2}[\cO^{s_{1}s_{2}}]} \\ 
& \quad - Q_{2}[\cO^{s_{2}}] \underline{{}- Q_{2}[\cO^{s_{1}s_{2}}]} \underline{{}- Q_{2}[\cO^{s_{1}s_{2}}]} \underline{{}+ Q_{2}[\cO^{s_{1}s_{2}}]} \underline{{}+ Q_{2}[\cO^{s_{1}s_{2}}]}) 
\end{split} \\ 
&= \be^{s_{2}s_{3}s_{2}\vpi_{2}} ([\cO^{s_{2}s_{3}s_{2}}] - 2[\cO^{s_{1}s_{2}s_{3}s_{2}}] - [\cO^{s_{2}s_{3}s_{1}s_{2}}] + 2[\cO^{s_{1}s_{2}s_{3}s_{1}s_{2}}] + Q_{2}[\cO^{e}] - Q_{2}[\cO^{s_{2}}]). 
\end{align}

Also, we deduce that $\A_{\lessdot}(w, \Gamma(2)) = \{\emptyset, \{1\}, \{4\}, \{5\}, \{1, 4\}, \{4, 6\}\}$; note that only two elements $A = \{1\}, \{1, 4\}$ of $\A_{\lessdot}(w, \Gamma(2))$ satisfy $\wend(w, A^{1}) \ge \mcr{s_{\theta}}$. Therefore, we see that 
\begin{align}
& \text{(RHS of equation~\eqref{typec-chev})} \\ 
\begin{split}
&= \be^{s_{2}s_{3}s_{2}\vpi_{2}} ([\cO^{s_{2}s_{3}s_{2}}] - [\cO^{s_{1}s_{2}s_{3}s_{2}}] - [\cO^{s_{2}s_{3}s_{1}s_{2}}] - [\cO^{s_{1}s_{2}s_{3}s_{2}}] + [\cO^{s_{1}s_{2}s_{3}s_{1}s_{2}}] + [\cO^{s_{1}s_{2}s_{3}s_{1}s_{2}}]) \\ 
& \quad - Q_{2} \be^{s_{2}s_{3}s_{2}\vpi_{2}} (- [\cO^{\mcr{s_{1}}}] + [\cO^{\mcr{s_{2}s_{1}}}])
\end{split} \\ 
&= \be^{s_{2}s_{3}s_{2}\vpi_{2}} ([\cO^{s_{2}s_{3}s_{2}}] - 2[\cO^{s_{1}s_{2}s_{3}s_{2}}] - [\cO^{s_{2}s_{3}s_{1}s_{2}}] + 2[\cO^{s_{1}s_{2}s_{3}s_{1}s_{2}}] + Q_{2}[\cO^{e}] - Q_{2}[\cO^{s_{2}}]) \\ 
&= [\cO(-\vpi_{2})] \cdot [\cO^{s_{2}s_{3}s_{2}}]; 
\end{align}
here, we have used $s_{2\ve_{2}} = s_{2}s_{3}s_{2}$ for the first equality. 
Thus Theorem~\ref{qkchev-c} holds in this case. 
\end{exa}

\begin{proof}[Proof of Theorem~\ref{qkchev-c}] We assume that $w$ is not the identity, as this case is trivial. We follow the same procedure outlined above, and describe the sign-reversing involution canceling terms obtained from the Chevalley formula for $G/B$. 

We carry out the proof in the case $k<n$, and refer to $k=n$ at the end. Consider a generic $w$-admissible subset $A$ in $\A(w,\Gamma(k))$, corresponding to a term in the Chevalley formula for $G/B$.  Like in type $A$, the structure of $\Gamma(k)$ combined with the fact that $w\in W^J$ imply that $A$ contains at most one root labeling a Bruhat cover in $\Pi(w,A)$ from each row in the display of $\Gamma_j(k)$ in \eqref{omegachain2-c}. 

We focus on those $A$ with $A^-\ne\emptyset$. By Lemma~\ref{paredgec}, we have $A^-\subseteq\{\alpha_k=(k,k+1),\,2\varepsilon_k=(k,\overline{k})\}$. Note that both of these roots appear only once in the $(-\varpi_k)$-chain $\Gamma(k)$, with $(k,k+1)$ being the last one, while $(k,\overline{k})$ appears in the last segment $\Gamma_k(k)$. In fact,  we have either $A^-=\{(k,k+1)\}$ or $A^-=\{(k,\overline{k})\}$. Indeed, assuming that $(k,\overline{k})\in A^-$, and considering the signed permutation $u$ in $\Pi(w,A)$ to which $(k,\overline{k})$ is applied, we have $u\in W^J$ and $u[k+1,n]\subseteq\{a+1,\ldots,n\}$, where $a:=|u(k)|=u(\overline{k})$; therefore, it is impossible for $(k,k+1)$ to correspond to a quantum step in $\Pi(w,A)$.

Now assume that $A^-=\{(k,k+1)\}$, and let $v:=\wend(w,A\setminus\{(k,k+1)\})$. We clearly have $v\in W^J$. We will pair $A$ with another $w$-admissible subset $A'$, such that their contributions to the parabolic Chevalley formula for $G/P_J$ cancel out. We must have one of the following cases, where $1\le a<b\le n$.

\emph{Case} 1: $v(k)=b$, $v(k+1)=a$, and $A$ does not contain $(k,j)$ with $j>k+1$.

\emph{Subcase} 1.1: $b<n$.  This case is completely similar to Case~1 in the type $A$ proof. Indeed, there clearly exists $q>k+1$ such that $v(q)=b+1\le n$. We let $A':=A\cup\{(k,q)\}$, so $(A')^-=\{(k,k+1)\}$, and continue the reasoning as above.

\emph{Subcase} 1.2: $b=n$. We let $A':=A\cup\{(k,\overline{k})\}$, and we have $(A')^-=\{(k,k+1)\}$. 

\emph{Case} 2: $v(k)=\overline{a}$, $v(k+1)=b$.  We let $A':=A\cup\{(k,\overline{k})\}$,  and we have $(A')^-=\{(k,\overline{k})\}$. 

\emph{Case} 3: $v(k)=\overline{b}$, $v(k+1)=a$, and $A$ contains neither $(k,\overline{k})$, nor $(k,\overline{\jmath})$ for $j>k+1$. 

\emph{Subcase} 3.1: $k+2\le n$ and $v(k+2)<b$. Consider $q>k+1$ largest such that $v(q)<b$. We let $A':=A\cup\{(k,\overline{q})\}$, and we have $(A')^-=\{(k,k+1)\}$. 

\emph{Subcase} 3.2: $k+2> n$ or $v(k+2)>b$. We let $A':=(A\setminus\{(k,k+1)\}) \cup\{(k,\overline{k+1}),\,(k,\overline{k})\}$, and we have $(A')^-=\{(k,\overline{k})\}$. 

We claim that in all cases, 
\[A'\in\A(w,\Gamma(k))\,,\;\;\;\lfloor\wend(w,A)\rfloor=\lfloor\wend(w,A')\rfloor\,,\;\;\;\mbox{and}\;\;\;\wt(w,A)=\wt(w,A')\,.\]
Furthermore, it is not hard to check that these cases completely pair up all $w$-admissible subsets $A$ with $A^-=\{(k,k+1)\}$, either among themselves (in Cases~1.1, 1.2, and 3.1),  or with $A$ satisfying $A^-=\{(k,\overline{k})\}$ (in Cases~2 and 3.2); see below for a discussion of the latter $A$ which are not paired up above. 

Indeed, let us consider, for instance, Case~2. We cannot have $(k,\overline{k})\in A$, because the corresponding up step in Bruhat order would not be a cover (by the classical part of the criterion in Proposition~\ref{prop:quantum_bruhat_order_type_C}~(3)). Moreover, $A$ cannot contain any root of the form $(k,j)$ with $j>k+1$, as the corresponding reflection would bring a positive entry to position $k$, whereas $v(k)$ is negative. Therefore, the roots $(k,\overline{k})$ and $(k,k+1)$ are the last two in $A'$, while the step corresponding to $(k,\overline{k})$ is a quantum one (by the criterion in Proposition~\ref{prop:quantum_bruhat_order_type_C}~(3)). Moreover, we have
\[\lfloor{\rm end}(w,A')\rfloor=\lfloor v(k,\overline{k})(k,k+1)\rfloor=\lfloor v(k,k+1)(k+1,\overline{k+1})\rfloor=\lfloor{\rm end}(w,A)\rfloor\,.\]
The weight preservation is verified by noting that all affine hyperplanes corresponding to the roots in $\Gamma^2(k)$ contain $\varpi_k$; so the corresponding affine reflections fix $\varpi_k$, and are thus irrelevant for the weight computation. 

On another hand, in the Chevalley formula for $G/B$, the quantum steps corresponding to both roots $(k,k+1)$ and $(k,\overline{k})$ contribute the variable $Q_k$. Indeed, as indicated above, we have the following coroot splitting: $(2\varepsilon_k)^\vee=\alpha_k^\vee+(\alpha_{k+1}^\vee+\cdots+\alpha_n^\vee)$. Finally, since the cardinalities of $A$ and $A'$ differ by $1$, we conclude that the involution $A\leftrightarrow A'$ is indeed sign-reversing. In this way, the contributions to the parabolic Chevalley formula for $G/P_{J}$ of all $A$ with $A^-=\{(k,k+1)\}$ are canceled. 

We have now exhausted all $w$-admissible sets $A$ with $A^-=\{(k,k+1)\}$. Thus, it remains to discuss the contributions of the remaining $A$ with $A^-\ne\emptyset$, i.e., $A^-=\{(k,\overline{k})\}$ and $A$ is not among the $A'$ in Cases~2 and 3.2. So from now on we work under this assumption. We previously considered the signed permutation $u\in W^J$ in $\Pi(w,A)$ to which $(k,\overline{k})$ is applied, and observed that $u[k+1,n]\subseteq\{a+1,\ldots,n\}$, where $a:=|u(k)|$. If $(k,\overline{k})$ is followed by another root in $A$, then this can only be $(k,k+1)$; but this situation was considered in Case~2 above, which means that $(k,\overline{k})$ must be the last root in $A$. Moreover, $A$ cannot contain any root of the form $(k,\overline{\jmath})$ with $j>k$, because we would be in Case~3.2. The following two cases cover all remaining possibilities, and we continue to use the above notation.

\emph{Case} 4: $u(k)\ne\overline{1}$ (i.e., $a\ne 1$), and  $A^2$ contains no root $(i,\overline{k})$ with $i<k$. There clearly exists $p<k$ such that $u(p)=a-1$. We let $A':=A\cup\{(p,\overline{k})\}$, where the root $(p,\overline{k})$ is taken from $\Gamma^2(k)$. We have $(A')^-=\{(k,\overline{k})\}$. Like above, we verify that the terms corresponding to $A$ and $A'$ cancel out, so we can extend the sign-reversing involution above by pairing $A$ with $A'$. 

Now recall that, in general, $A^2$ contains at most one root  $(i,\overline{k})$ with $i<k$. Whenever it contains one, the values in positions $i$ and $k$ of the signed permutation to which this reflection is applied are of the form $b-1$ and $\overline{b}$, respectively. Thus, the remaining case consists of the following $w$-admissible subsets $A$. 

\emph{Case} 5: $u(k)=\overline{1}$ (i.e., $a=1$), and $A^2$ contains no root $(i,\overline{k})$ with $i<k$. We clearly have $A\setminus\{(k,\overline{k})\}\in\A_\lessdot(w,\Gamma(k))$, where we recall that $(k,\overline{k})$ is the last root in $A$. Now let $u':=\wend(w,A^1)$. Based on the structures of $\Gamma(k)$ and $A$, we have $u'(k)=u(k)=\overline{1}$. But this is equivalent to $u'\ge\lfloor s_\theta\rfloor=[2, 3, \ldots, k, \overline{1}, k+1, \ldots, n]$ (in the window notation), by Deodhar's criterion for the type $C$ Bruhat order~\cite[Chapter~8,~Exercise~6]{babccg}. In the same way as above, we can see that $\wt(w,A)=\wt(A\setminus\{(k,\overline{k})\})$. The above facts imply that the terms corresponding to this case make up the second sum in~\eqref{typec-chev}. By Remark~\ref{nocancel}, there are no cancellations between these terms. 

We conclude by considering $k=n$, and noting that the proof reduces to Cases~4 and 5 above. 
\end{proof}

We now prove the positivity property of structure constants for isotropic Grassmannians as a corollary of Theorem~\ref{qkchev-c}. 

\begin{cor}
Let $G$ be of type $C_{n}$, and $J = I \setminus \{k\}$ for an arbitrary fixed $1 \le k \le n$. 
Then, for $w, u \in W^{J}$ and $\xi \in Q_{I \setminus J}^{\vee, +}$, we have 
\begin{equation}
(-1)^{1 + \ell(w) + \ell(u) + \deg(Q^{\xi})}N_{s_{k}, w}^{u, \xi} \in \bZ_{\ge 0}[\be^{\gamma} - 1 \mid \gamma \in -\Delta].  
\end{equation}
\end{cor}

\begin{proof}
Take $A \in \A_{\lessdot}(w, \Gamma(k))$ such that $\wend(w, A^{1}) \ge \mcr{s_{\theta}}$, and set $v := \wend(w, A)$. 
Recall from the proof of Theorem~\ref{qkchev-c} that there exists a quantum edge $v \xrightarrow{2\ve_{k}} vs_{2\ve_{k}} \in \QB(W)$. 
Also, by Case~5 in the proof of Theorem~\ref{qkchev-c}, we have $v(k) = \overline{1}$. 
Note that $v \in W^{J}$, and hence that $v(1) < \cdots < v(k)$, $v(k+1) < \cdots < v(n)$. 
It follows that $1 = vs_{2\ve_{k}}(k) < vs_{2\ve_{k}}(1) < \cdots < vs_{2\ve_{k}}(k-1)$ and $vs_{2\ve_{k}}(k+1) < \cdots < vs_{2\ve_{k}}(n)$. 
Therefore, if we take a cyclic permutation $\sigma := (1, k,  k-1, \ldots, 2) \in W$, then we have $\mcr{vs_{2\ve_{k}}} = vs_{2\ve_{k}}\sigma$. 
Hence we see that 
\begin{align}
|A|+1 &= |A \cup \{(k, \overline{k})\}| \\ 
&\equiv \ell(vs_{2\ve_{k}}) - \ell(w) \\ 
&\equiv (\ell(vs_{2\ve_{k}}\sigma) - \ell(\sigma)) - \ell(w) \\ 
&\equiv \ell(\mcr{vs_{2\ve_{k}}}) - (k-1) - \ell(w) \\ 
&\equiv \ell(w) + \ell(\mcr{vs_{2\ve_{k}}}) + k - 1
\end{align}
modulo 2. Thus, we obtain 
\begin{equation}
(-1)^{|A|+1} = (-1)^{\ell(w) + \ell(\mcr{vs_{2\ve_{k}}}) + k - 1}. 
\end{equation}

It is easy to check (see, for example, \cite[Section~3.1.5, Exercise~4]{gwrep}) that 
\begin{equation}
2\rho_{J} = \begin{cases}
\displaystyle \sum_{i = 1}^{k-1} i(k-1) \alpha_{i} + \sum_{i = 1}^{n-k-1} i(2(n-k)-i+1) \alpha_{k+1} + \dfrac{(n-k)(n-k+1)}{2} \alpha_{n} & \text{if $k \not= n$}, \\ 
\displaystyle \sum_{i = 1}^{n-1} i(n-i) \alpha_{i} & \text{if $k = n$}. 
\end{cases} 
\end{equation}
Since 
\begin{equation}
\pair{\alpha_{i}}{\alpha_{j}^{\vee}} = \begin{cases}
2 & \text{if $i = j$,} \\ 
-1 & \text{if $|i - j| = 1$ and $i \not= n$}, \\ 
-2 & \text{if $i = n$ and $j = n-1$}, \\ 
0 & \text{otherwise},
\end{cases}
\end{equation}
we have 
\begin{align}
2\pair{\rho_{J}}{\alpha_{k}^{\vee}} &= \begin{cases} 
-(k-1) - 2(n-k) & \text{if $k \not= n-1, n$}, \\ 
-(n-2) - 2 \cdot \dfrac{1 \cdot 2}{2} & \text{if $k = n-1$}, \\ 
-(n-1) & \text{if $k = n$} 
\end{cases} \\ 
&\equiv k-1 \mod 2.  
\end{align}
In addition, we have $2\pair{\rho}{\alpha_{k}^{\vee}} = 2$. Therefore, we see that 
\begin{equation}
\deg(Q_{k}) = 2\pair{\rho}{\alpha_{k}^{\vee}} - 2\pair{\rho_{J}}{\alpha_{k}^{\vee}} \equiv 2 - (k-1) \equiv k-1 \mod 2. 
\end{equation}

We set 
\begin{equation}
\A(w, \Gamma(k))_{u, \alpha_{k}^{\vee}, \lambda}^{0} := \left\{ A \in \A_{\lessdot}(w, \Gamma(k)) \ \middle| \ \begin{aligned} &\wend(w, A^{(1)}) \ge \mcr{s_{\theta}}, \ \mcr{\wend(w, A)s_{2\ve_{k}}} = u, \\ &\wt(A \cup \{(k, \overline{k})\}) = \lambda \end{aligned} \right\}.
\end{equation}
Then, since $\wt(w, A) \in -\vpi_{k}+Q^{+}$ by Lemma~\ref{lem:wt_belonging}, we deduce from Theorem~\ref{qkchev-c} that 
\begin{align}
C_{w}^{u, \alpha_{k}^{\vee}} &= \sum_{\lambda \in -\vpi_{k}+Q^{+}} \sum_{A \in \A(w, \Gamma)_{u, \alpha_{k}^{\vee}, \lambda}^{0}} (-1)^{|A|+1} Q_{k} \be^{-\wt(w, A\cup\{(k, \overline{k})\})} \\ 
&= (-1)^{\ell(w) + \ell(u) + k-1} Q_{k} \sum_{\lambda \in -\vpi_{k}+Q^{+}} |\A(w, \Gamma(k))_{u, \alpha_{k}^{\vee}, \lambda}^{0}| \be^{-\lambda}. 
\end{align}
Therefore, we obtain 
\begin{equation}
N_{s_{k}, w}^{u, \alpha_{k}^{\vee}} = (-1)^{1 + \ell(w) + \ell(u) + k-1} Q_{k} \sum_{\lambda \in -\vpi_{k}+Q^{+}} |\A(w, \Gamma(k))_{u, \alpha_{k}^{\vee}, \lambda}^{0}| \be^{-\vpi_{k}-\lambda}. 
\end{equation}
This implies that 
\begin{equation} \label{eq:positivity_typeC}
(-1)^{1 + \ell(w) + \ell(u) + \deg(Q_{k})} N_{s_{k}, w}^{u, \alpha_{k}^{\vee}} \in \bZ_{\ge 0}[\be^{\gamma}-1 \mid \gamma \in -\Delta], 
\end{equation}
as desired. 
Equation \eqref{eq:positivity_typeC}, together with the positivity property of $N_{u, v}^{w, 0}$ for $u, v, w \in W^{J}$, proves the corollary. 
\end{proof}

\section{Quantum $K$-theory Chevalley formulas for two-step flag manifolds} \label{sec:2-step}
In this section, we concentrate on the case of type $A_{n-1}$; note that $I = [n-1]$ in this case. 
Let us consider the (standard) parabolic subgroup $P_{J} \supset B$ corresponding to $J = I \setminus \{k_{1}, k_{2}\}$ for some $k_{1},  k_{2} \in I$ with $k_{1} < k_{2}$; 
the resulting partial flag manifold $G/P_{J}$ is isomorphic to a \emph{two-step flag manifold} $\Fl(k_{1}, k_{2}; n)$ defined as:
\begin{equation}
\Fl(k_{1}, k_{2}; n) := \left\{ (V_{1}, V_{2}) \ \middle| \ \parbox{21em}{$V_{1}$ and $V_{2}$ are subspaces of $\bC^{n}$ such that $V_{1} \subset V_{2}$, $\dim V_{1} = k_{1}$, and $\dim V_{2} = k_{2}$} \right\}. 
\end{equation}
The purpose of this section is to derive cancellation-free parabolic Chevalley formulas for the quantum multiplication in $QK_T(G/P_J)$ with $[\cO(-\varpi_{k})]$, for $k = k_{1}$ and $k = k_{2}$. 
For this purpose, as in Section~\ref{sec:max_parabolic}, we examine all the terms to be canceled in certain formulas obtained from equation~\eqref{qkchev-f} in Theorem~\ref{qkchev}, in $QK_T^{\mathrm{poly}}(G/B) \subset QK_{T}(G/B)$, by applying the map $\Phi_{J} : QK_T^{\mathrm{poly}}(G/B) \rightarrow QK_T^{\mathrm{poly}}(G/P_J)$. 

\subsection{Some lemmas on admissible subsets} \label{subsec:lemmas}

Note that for $w \in W = W(A_{n-1}) = S_n$, 
$w \in W^{J}$ is equivalent to $w[1, k_{1}]$, $w[k_{1}+1, k_{2}]$, and $w[k_{2}+1, n]$ being increasing sequences (see \cite[Lemma~2.4.7]{babccg}). 

We first consider the case $k = k_{1}$. 
We will make repeated use of the following. 

\begin{lem} \label{lem:quantum_edge_2-step}
Consider $w \in W^{J}$. We have an edge $w \xrightarrow{(i, j)} w(i, j)$ in the quantum Bruhat graph $\QB(W)$, with $i \le k_{1} < j$, if and only if one of the following two conditions holds: 
\begin{enumerate}
\item the edge above is a Bruhat cover, and $w(i, j) \in W^{J}$; 
\item the edge above is a quantum one, and $(i, j) = (k_{1}, k_{2}+1)$ or $(i, j) = \alpha_{k_{1}}$. 
\end{enumerate}
\end{lem}

\begin{proof}
As in the proof of Lemma~\ref{paredge}, we implicitly use Proposition~\ref{prop:quantum_bruhat_order_type_A}, as well as the fact that $w[1, k_{1}]$, $w[k_{1}+1, k_{2}]$, and $w[k_{2}+1, n]$ are increasing sequences. 
Assume first that the edge above is a Bruhat cover. Then, since $(i, j) \notin W_{J}$, \cite[Corollary~2.5.2]{babccg} implies that $w(i, j) \in W^{J}$, as desired. 
Assume next that the edge above is a quantum one; note that $w(i) > w(j)$ in this case. If $i < k_{1}$, then the value $w(k_{1})$ would be straddled between $w(i)$ and $w(j)$. Hence we must have $i = k_{1}$. Also, if $k_{1}+1 < j \leq k_{2}$, then the value $w(k_{1}+1)$ would be straddled between $w(k_{1})$ and $w(j)$; if $j > k_{2}+1$, then the value $w(k_{2}+1)$ would be straddled between $w(k_{1})$ and $w(j)$. Hence we must have $j = k_{1}+1$ or $j = k_{2}+1$. This proves the lemma. 
\end{proof}

\begin{lem} \label{lem:quantum_edge_2-step_2}
Consider $w \in W^{J}$, and assume that we have a quantum edge $w \xrightarrow{(k_{1}, k_{2}+1)} w(k_{1}, k_{2}+1)$ in $\QB(W)$. 
Then, for $k_{1}+1 \leq j \leq k_{2}$, we have an edge $w(k_{1}, k_{2}+1) \xrightarrow{(k_{1}, j)} w(k_{1}, k_{2}+1)(k_{1}, j)$ in $\QB(W)$ if and only if $j = k_{1}+1$. 
In this case, the edge $w(k_{1}, k_{2}+1) \xrightarrow{(k_{1}, j)} w(k_{1}, k_{2}+1)(k_{1}, j)$ is a Bruhat cover. 
\end{lem}

\begin{proof}
Set $v := w(k_{1}, k_{2}+1)$. Since we have a quantum edge $w \xrightarrow{(k_{1}, k_{2}+1)} v$ in $\QB(W)$, Proposition~\ref{prop:quantum_bruhat_order_type_A} implies the following: 
\begin{equation}
v(k_{1}) < v(k_{1}+1) < v(k_{1}+2) < \cdots < v(k_{2}) < v(k_{2}+1). 
\end{equation}
If $k_{1}+1 < j \le k_{2}$, then the value $v(k_{1}+1)$ would be straddled between $v(k_{1})$ and $v(j)$. Hence we must have $j = k_{1}+1$. In this case, by Proposition~\ref{prop:quantum_bruhat_order_type_A}, we have a Bruhat edge $v \xrightarrow{(k_{1}, k_{1}+1)} v(k_{1}, k_{1}+1)$. This proves the lemma. 
\end{proof}

As a corollary of Lemmas~\ref{lem:quantum_edge_2-step} and \ref{lem:quantum_edge_2-step_2}, we immediately obtain the following. 

\begin{lem} \label{lem:containing_quantum_steps}
Let $w \in W^{J}$, and take $A = \{j_{1} < \cdots < j_{s}\} \in \A(w, \Gamma(k_{1}))$. If the directed path $\Pi(w, A)$ contains a quantum edge, then $\Pi(w, A)$ is one of the following forms; here $\xrightarrow[\sB]{}$ indicates a Bruhat edge, while $\xrightarrow[\sQ]{}$ indicates a quantum edge: 
\begin{enumerate}
\item $\Pi(w, A): w = w_{0} \xrightarrow[\sB]{} \cdots \xrightarrow[\sB]{} w_{s-1} \xrightarrow[\sQ]{(k_{1}, k_{1}+1)} w_{s}$; 
\item $\Pi(w, A): w = w_{0} \xrightarrow[\sB]{} \cdots \xrightarrow[\sB]{} w_{s-1} \xrightarrow[\sQ]{(k_{1}, k_{2}+1)} w_{s}$; 
\item $\Pi(w, A): w = w_{0} \xrightarrow[\sB]{} \cdots \xrightarrow[\sB]{} w_{s-2} \xrightarrow[\sQ]{(k_{1}, k_{2}+1)} w_{s-1} \xrightarrow[\sB]{(k_{1}, k_{1}+1)} w_{s}$. 
\end{enumerate}
\end{lem}

In view of this lemma, we divide the set $\A(w, \Gamma(k_{1}))$ into the disjoint union of the following four subsets: 
\begin{enumerate}
\item $\A_{\lessdot}(w, \Gamma(k_{1}))$ (defined in Section~\ref{sec:qbg}); 
\item $\A_{1}(w, \Gamma(k_{1})) := \{ A \in \A(w, \Gamma(k_{1})) \mid \text{$\Pi(w, A)$ is of the form (1) in Lemma~\ref{lem:containing_quantum_steps}} \}$; 
\item $\A_{2}(w, \Gamma(k_{1})) := \{ A \in \A(w, \Gamma(k_{1})) \mid \text{$\Pi(w, A)$ is of the form (2) in Lemma~\ref{lem:containing_quantum_steps}} \}$; 
\item $\A_{3}(w, \Gamma(k_{1})) := \{ A \in \A(w, \Gamma(k_{1})) \mid \text{$\Pi(w, A)$ is of the form (3) in Lemma~\ref{lem:containing_quantum_steps}} \}$. 
\end{enumerate}
Then it follows that 
\begin{equation}
\A(w, \Gamma(k_{1})) = \A_{\lessdot}(w, \Gamma(k_{1})) \sqcup \A_{1}(w, \Gamma(k_{1})) \sqcup \A_{2}(w, \Gamma(k_{1})) \sqcup \A_{3}(w, \Gamma(k_{1})). 
\end{equation}
Also, we can verify the following: 
\begin{itemize}
\item if $A \in \A_{\lessdot}(w, \Gamma(k_{1}))$, then $\dn(w, A) = 0$, and hence $Q^{[\dn(w, A)]^{J}} = 0$; 
\item if $A \in \A_{1}(w, \Gamma(k_{1}))$, then $\dn(w, A) = \alpha_{k_{1}}^{\vee}$, and hence $Q^{[\dn(w, A)]^{J}} = Q_{k_{1}}$; 
\item if $A \in \A_{2}(w, \Gamma(k_{1}))$ or $A \in \A_{3}(w, \Gamma(k_{1}))$, then $\dn(w, A) = \alpha_{k_{1}}^{\vee} + \cdots + \alpha_{k_{2}}^{\vee}$, and hence $Q^{[\dn(w, A)]^{J}} = Q^{\alpha_{k_{1}}^{\vee} + \alpha_{k_{2}}^{\vee}} = Q_{k_{1}}Q_{k_{2}}$. 
\end{itemize}
Therefore, by equation~\eqref{qkchev-f}, we deduce that 
\begin{align}
[\cO(-\varpi_{k_{1}})] \cdot [\cO^{w}] &= \be^{w\varpi_{k_{1}}} \sum_{A \in \A_{\lessdot}(w, \Gamma(k_{1}))} (-1)^{|A|} [\cO^{\wend(w, A)}] \\ 
& \quad + \be^{w\varpi_{k_{1}}} \sum_{A \in \A_{1}(w, \Gamma(k_{1}))} (-1)^{|A|} Q_{k_{1}} [\cO^{\mcr{\wend(w, A)}}] \\ 
& \quad + \be^{w\varpi_{k_{1}}} \sum_{A \in \A_{2}(w, \Gamma(k_{1}))} (-1)^{|A|} Q_{k_{1}} Q_{k_{2}} [\cO^{\mcr{\wend(w, A)}}] \\ 
& \quad + \be^{w\varpi_{k_{1}}} \sum_{A \in \A_{3}(w, \Gamma(k_{1}))} (-1)^{|A|} Q_{k_{1}} Q_{k_{2}} [\cO^{\mcr{\wend(w, A)}}]\,. 
\end{align}

Next, we consider the case $k = k_{2}$. In this case, we use $\Gamma^{\ast}(k_{2})$ instead of $\Gamma(k_{2})$. 
From Lemma~\ref{lem:containing_quantum_steps}, by applying the diagram automorphism $\omega$, we obtain the following. 

\begin{lem} \label{lem:containing_quantum_steps_2}
Let $w \in W^{J}$, and take $A = \{j_{1} < \cdots < j_{s}\} \in \A(w, \Gamma^{\ast}(k_{2}))$. If the directed path $\Pi(w, A)$ contains a quantum edge, then $\Pi(w, A)$ is one of the following forms; here, $\xrightarrow[\sB]{}$ indicates a Bruhat edge, while $\xrightarrow[\sQ]{}$ indicates a quantum edge: 
\begin{enumerate}
\item $\Pi(w, A): w = w_{0} \xrightarrow[\sB]{} \cdots \xrightarrow[\sB]{} w_{s-1} \xrightarrow[\sQ]{(k_{2}, k_{2}+1)} w_{s}$; 
\item $\Pi(w, A): w = w_{0} \xrightarrow[\sB]{} \cdots \xrightarrow[\sB]{} w_{s-1} \xrightarrow[\sQ]{(k_{1}, k_{2}+1)} w_{s}$; 
\item $\Pi(w, A): w = w_{0} \xrightarrow[\sB]{} \cdots \xrightarrow[\sB]{} w_{s-2} \xrightarrow[\sQ]{(k_{1}, k_{2}+1)} w_{s-1} \xrightarrow[\sB]{(k_{2}, k_{2}+1)} w_{s}$. 
\end{enumerate}
\end{lem}

In view of this lemma, we divide the set $\A(w, \Gamma^{\ast}(k_{2}))$ into the disjoint union of the following four subsets: 
\begin{enumerate}
\item $\A_{\lessdot}(w, \Gamma^{\ast}(k_{2}))$ (already defined); 
\item $\A_{1}(w, \Gamma^{\ast}(k_{2})) := \{ A \in \A(w, \Gamma^{\ast}(k_{2})) \mid \text{$\Pi(w, A)$ is of the form (1) in Lemma~\ref{lem:containing_quantum_steps_2}} \}$; 
\item $\A_{2}(w, \Gamma^{\ast}(k_{2})) := \{ A \in \A(w, \Gamma^{\ast}(k_{2})) \mid \text{$\Pi(w, A)$ is of the form (2) in Lemma~\ref{lem:containing_quantum_steps_2}} \}$; 
\item $\A_{3}(w, \Gamma^{\ast}(k_{2})) := \{ A \in \A(w, \Gamma^{\ast}(k_{2})) \mid \text{$\Pi(w, A)$ is of the form (3) in Lemma~\ref{lem:containing_quantum_steps_2}} \}$. 
\end{enumerate}

\subsection{Parabolic Chevalley formulas for two-step flag manifolds} \label{sec:two-step_positivity}

We state cancellation-free parabolic Chevalley formulas for the equivariant quantum $K$-theory of the two-step flag manifold $G/P_{J}$; 
the proofs of these results will be given in Sections~\ref{subsec:proof1} and \ref{subsec:proof2}. 
First, we assume that $k = k_{1}$. Take and fix $w \in W^{J}$. 

\begin{thm} \label{thm:2-step_1}
If $w(k_{1}) < w(k_{1}+1)$, then we have the following cancellation-free formula: 
\begin{equation}
[\cO(-\varpi_{k_{1}})] \cdot [\cO^{w}] = \be^{w\varpi_{k_{1}}} \sum_{A \in \A_{\lessdot}(w, \Gamma(k_{1}))} (-1)^{|A|} [\cO^{\wend(w, A)}]\,. 
\end{equation}
\end{thm}

We consider the following condition: 
\begin{itemize}
\item[(Q)] $w(k_{1}) > w(k_{2})$ and $w(k_{1}+1) > w(k_{2}+1)$. 
\end{itemize}

\begin{rema} \label{rem:k1>k1+1}
As mentioned at the beginning of Section~\ref{subsec:lemmas}, $w[k_{1}+1, k_{2}]$ is an increasing sequence for $w \in W^{J}$. 
Hence condition (Q) implies that $w (k_{1}) > w(k_{2}) \geq w(k_{1}+1) > w(k_{2} + 1)$. 
\end{rema}

\begin{thm} \label{thm:2-step_2}
Assume that $w(k_{1}) > w(k_{1}+1)$, and assume that condition \textup{(Q)} does not hold. 
\begin{enumerate}
\item If $w(1) < w(k_{1}+1)$ or $w(k_{1}) < w(k_{2})$, then we have the following cancellation-free formula: 
\begin{equation}
[\cO(-\varpi_{k_{1}})] \cdot [\cO^{w}] = \be^{w\varpi_{k_{1}}} \sum_{A \in \A_{\lessdot}(w, \Gamma(k_{1}))} (-1)^{|A|} [\cO^{\wend(w, A)}]\,. 
\end{equation}

\item If $w(1) > w(k_{1}+1)$ and $w(k_{1}) > w(k_{2})$, then we have the following cancellation-free formula: 
\begin{equation}
[\cO(-\varpi_{k_{1}})] \cdot [\cO^{w}] = \be^{w\varpi_{k_{1}}} \sum_{A \in \A_{\lessdot}(w, \Gamma(k_{1}))} (-1)^{|A|} \left( [\cO^{\wend(w, A)}] - Q_{k_{1}} [\cO^{\mcr{\wend(w, A)s_{k_{1}}}}] \right)\,. 
\end{equation}
\end{enumerate}
\end{thm}

Also, we consider the following condition: 
\begin{itemize}
\item[(Full)] both of the following hold: 
\begin{enumerate}
\item $w(k_{1}) = n$ and $w(k_{2}+1) = 1$; and 
\item $w(k_{1}+1)$ is the minimum element in the sequence $w[1, k_{2}]$. 
\end{enumerate}
\end{itemize}

\begin{rema}
Condition (Full) holds if and only if condition (Q) holds and $w(1) > w(k_{1}+1)$, $w(k_{1}) > w(n)$; 
note that the inequality $w(1) > w(k_{1}+1)$, together with condition (Q), implies that $w(1) > w(k_{2}+1)$. 
\end{rema}

\begin{thm} \label{thm:2-step_3}
Assume condition \textup{(Q)}. 
\begin{enumerate}
\item Assume that $w(k_{1}) < w(n)$. 
\begin{enumerate}
\item If $w(1) < w(k_{1}+1)$, then we have the following cancellation-free formula: 
\begin{equation}
[\cO(-\varpi_{k_{1}})] \cdot [\cO^{w}] = \be^{w\varpi_{k_{1}}} \sum_{A \in \A_{\lessdot}(w, \Gamma(k_{1}))} (-1)^{|A|} [\cO^{\wend(w, A)}]\,. 
\end{equation}

\item If $w(1) > w(k_{1}+1)$, then we have the following cancellation-free formula: 
\begin{equation}
[\cO(-\varpi_{k_{1}})] \cdot [\cO^{w}] = \be^{w\varpi_{k_{1}}} \sum_{A \in \A_{\lessdot}(w, \Gamma(k_{1}))} (-1)^{|A|} \left( [\cO^{\wend(w, A)}] - Q_{k_{1}} [\cO^{\mcr{\wend(w, A)s_{k_{1}}}}] \right)\,. 
\end{equation}
\end{enumerate}

\item Assume that $w(k_{1}) > w(n)$. 
\begin{enumerate}
\item If $w(1) < w(k_{2}+1)$, then we have the following cancellation-free formula: 
\begin{equation}
[\cO(-\varpi_{k_{1}})] \cdot [\cO^{w}] = \be^{w\varpi_{k_{1}}} \sum_{A \in \A_{\lessdot}(w, \Gamma(k_{1}))} (-1)^{|A|} [\cO^{\wend(w, A)}]\,. 
\end{equation}

\item If $w(k_{2}+1) < w(1) < w(k_{1}+1)$, then we have the following cancellation-free formula: 
\begin{equation}
[\cO(-\varpi_{k_{1}})] \cdot [\cO^{w}] = \be^{w\varpi_{k_{1}}} \sum_{A \in \A_{\lessdot}(w, \Gamma(k_{1}))} (-1)^{|A|} \left( [\cO^{\wend(w, A)}] - Q_{k_{1}}Q_{k_{2}} [\cO^{\mcr{\wend(w, A)(k_{1}, k_{2}+1)}}] \right)\,. 
\end{equation}
\end{enumerate}

\item If condition \textup{(Full)} holds, then we have the following cancellation-free formula: 
\begin{equation}
\begin{split}
[\cO(-\varpi_{k_{1}})] \cdot [\cO^{w}] &= \be^{w\varpi_{k_{1}}} \sum_{A \in \A_{\lessdot}(w, \Gamma(k_{1}))} (-1)^{|A|} \left( [\cO^{\wend(w, A)}] - Q_{k_{1}} [\cO^{\mcr{\wend(w, A)s_{k_{1}}}}] \right. \\ 
& \hspace*{20mm} - \left. Q_{k_{1}}Q_{k_{2}} \left( [\cO^{\mcr{\wend(w, A)(k_{1}, k_{2}+1)}}] - [\cO^{\mcr{\wend(w, A)(k_{1}, k_{2}+1)s_{k_{1}}}}] \right) \right)\,. 
\end{split}
\end{equation}
\end{enumerate}
\end{thm}

Next, we assume that $k = k_{2}$. 

\begin{thm} \label{thm:2-step_4}
If $w(k_{2}) < w(k_{2}+1)$, then we have the following cancellation-free formula: 
\begin{equation}
[\cO(-\varpi_{k_{2}})] \cdot [\cO^{w}] = \be^{w\varpi_{k_{2}}} \sum_{A \in \A_{\lessdot}(w, \Gamma^{\ast}(k_{2}))} (-1)^{|A|} [\cO^{\wend(w, A)}]\,. 
\end{equation}
\end{thm}

Recall condition (Q) above. 

\begin{thm} \label{thm:2-step_5}
Assume that $w(k_{2}) > w(k_{2}+1)$, and assume that condition \textup{(Q)} does not hold. 
\begin{enumerate}
\item If $w(k_{2}) < w(n)$ or $w(k_{1}+1) < w(k_{2}+1)$, then we have the following cancellation-free formula: 
\begin{equation}
[\cO(-\varpi_{k_{2}})] \cdot [\cO^{w}] = \be^{w\varpi_{k_{2}}} \sum_{A \in \A_{\lessdot}(w, \Gamma^{\ast}(k_{2}))} (-1)^{|A|} [\cO^{\wend(w, A)}]\,. 
\end{equation}

\item If $w(k_{2}) > w(n)$ and $w(k_{1}+1) > w(k_{2}+1)$, then we have the following cancellation-free formula: 
\begin{equation}
[\cO(-\varpi_{k_{2}})] \cdot [\cO^{w}] = \be^{w\varpi_{k_{2}}} \sum_{A \in \A_{\lessdot}(w, \Gamma^{\ast}(k_{2}))} (-1)^{|A|} \left( [\cO^{\wend(w, A)}] - Q_{k_{2}} [\cO^{\mcr{\wend(w, A)s_{k_{2}}}}] \right)\,. 
\end{equation}
\end{enumerate}
\end{thm}

We consider the following analog of condition (Full): 
\begin{itemize}
\item[(Full)$^{\ast}$] both of the following hold: 
\begin{enumerate}
\item $w(k_{1}) = n$ and $w(k_{2}+1) = 1$; and 
\item $w(k_{2})$ is the maximum element in the sequence $w[k_{1}+1, n]$. 
\end{enumerate}
\end{itemize}

\begin{rema}
Condition (Full)$^{\ast}$ holds if and only if condition (Q) holds and $w(n) < w(k_{2})$, $w(k_{2}+1) < w(1)$; 
note that the inequality $w(n) < w(k_{2})$, together with condition (Q), implies that $w(n) < w(k_{1})$. 
\end{rema}

\begin{thm} \label{thm:2-step_6}
Assume condition \textup{(Q)}. 
\begin{enumerate}
\item Assume that $w(1) < w(k_{2}+1)$. 
\begin{enumerate}
\item If $w(k_{2}) < w(n)$, then we have the following cancellation-free formula: 
\begin{equation}
[\cO(-\varpi_{k_{2}})] \cdot [\cO^{w}] = \be^{w\varpi_{k_{2}}} \sum_{A \in \A_{\lessdot}(w, \Gamma^{\ast}(k_{2}))} (-1)^{|A|} [\cO^{\wend(w, A)}]\,. 
\end{equation}

\item If $w(k_{2}) > w(n)$, then we have the following cancellation-free formula: 
\begin{equation}
[\cO(-\varpi_{k_{2}})] \cdot [\cO^{w}] = \be^{w\varpi_{k_{2}}} \sum_{A \in \A_{\lessdot}(w, \Gamma^{\ast}(k_{2}))} (-1)^{|A|} \left( [\cO^{\wend(w, A)}] - Q_{k_{2}} [\cO^{\mcr{\wend(w, A)s_{k_{2}}}}] \right)\,. 
\end{equation}
\end{enumerate}

\item Assume that $w(1) > w(k_{2}+1)$. 
\begin{enumerate}
\item If $w(k_{1}) < w(n)$, then we have the following cancellation-free formula: 
\begin{equation}
[\cO(-\varpi_{k_{2}})] \cdot [\cO^{w}] = \be^{w\varpi_{k_{2}}} \sum_{A \in \A_{\lessdot}(w, \Gamma^{\ast}(k_{2}))} (-1)^{|A|} [\cO^{\wend(w, A)}]\,. 
\end{equation}

\item If $w(k_{2}) < w(n) < w(k_{1})$, then we have the following cancellation-free formula: 
\begin{equation}
[\cO(-\varpi_{k_{2}})] \cdot [\cO^{w}] = \be^{w\varpi_{k_{2}}} \sum_{A \in \A_{\lessdot}(w, \Gamma^{\ast}(k_{2}))} (-1)^{|A|} \left( [\cO^{\wend(w, A)}] - Q_{k_{1}}Q_{k_{2}} [\cO^{\mcr{\wend(w, A)(k_{1}, k_{2}+1)}}] \right)\,. 
\end{equation}
\end{enumerate}

\item If condition \textup{(Full)$^{\ast}$} holds, then we have the following cancellation-free formula: 
\begin{equation}
\begin{split}
[\cO(-\varpi_{k_{2}})] \cdot [\cO^{w}] &= \be^{w\varpi_{k_{2}}} \sum_{A \in \A_{\lessdot}(w, \Gamma^{\ast}(k_{2}))} (-1)^{|A|} \left( [\cO^{\wend(w, A)}] - Q_{k_{2}} [\cO^{\mcr{\wend(w, A)s_{k_{2}}}}] \right. \\ 
& \hspace*{20mm} \left. - Q_{k_{1}}Q_{k_{2}} \left( [\cO^{\mcr{\wend(w, A)(k_{1}, k_{2}+1)}}] - [\cO^{\mcr{\wend(w, A)(k_{1}, k_{2}+1)s_{k_{2}}}}] \right) \right)\,. 
\end{split}
\end{equation}

\end{enumerate}
\end{thm}

\begin{exa}
In this example, we consider the case that $n = 6$ and $(k_{1}, k_{2}) = (2, 4)$. Let $w = s_{4}s_{1}s_{2}s_{3}s_{5}s_{4}s_{3}s_{2}$. Then, $w$ satisfies condition (Q), and we see that $w(k_{2}+1) (= w(5)) {}< w(1) < w(k_{1}+1) (= w(3))$. Recall that $\Gamma(2) = ((1, 6), (1, 5), (1, 4), (1, 3), (2, 6), (2, 5), (2, 4), (2, 3))$. Then Table~3 is the list of all admissible subset $A \in \A(w, \Gamma(5))$ and their statistics $\wend(w, A)$, $\dn(w, A)$, together with $\mcr{\wend(w, A)}$. 
\begin{table}[ht]
\caption{The list of all admissible subsets $A \in \A(s_{4}s_{1}s_{2}s_{3}s_{5}s_{4}s_{3}s_{2}, \Gamma(2))$}
\label{tab:exm_2-step}
\centering
\begin{tabular}{|c|ccc|} \hline
$A$ & $\wend(w, A)$ & $\mcr{\wend(w, A)}$ & $\dn(w, A)$ \\ \hline
$\emptyset$ & $s_{4}s_{5}s_{1}s_{2}s_{3}s_{4}s_{3}s_{2}$ & $s_{4}s_{5}s_{1}s_{2}s_{3}s_{4}s_{3}s_{2}$ & $0$ \\ 
$\{4\}$ & $s_{4}s_{5}s_{1}s_{2}s_{3}s_{4}s_{3}s_{1}s_{2}$ & $s_{4}s_{5}s_{1}s_{2}s_{3}s_{4}s_{3}s_{1}s_{2}$ & $0$ \\ 
$\{6\}$ & $s_{4}s_{5}s_{1}$ & $s_{4}$ & $\alpha_{2}^{\vee} + \alpha_{3}^{\vee} + \alpha_{4}^{\vee}$ \\ 
$\{8\}$ & $s_{4}s_{5}s_{1}s_{2}s_{3}s_{4}s_{3}$ & $s_{4}s_{5}s_{1}s_{2}s_{3}s_{4}$ & $\alpha_{2}^{\vee}$ \\ 
$\{4, 6\}$ & $s_{4}s_{5}s_{2}s_{1}$ & $s_{4}s_{2}$ & $\alpha_{2}^{\vee} + \alpha_{3}^{\vee} + \alpha_{4}^{\vee}$ \\ 
$\{4, 8\}$ & $s_{4}s_{5}s_{1}s_{2}s_{3}s_{4}s_{3}s_{1}$ & $s_{4}s_{5}s_{1}s_{2}s_{3}s_{4}$ & $\alpha_{2}^{\vee}$ \\ 
$\{6, 8\}$ & $s_{4}s_{5}s_{1}s_{2}$ & $s_{4}s_{1}s_{2}$ & $\alpha_{2}^{\vee} + \alpha_{3}^{\vee} + \alpha_{4}^{\vee}$ \\ 
$\{4, 6, 8\}$ & $s_{4}s_{5}s_{1}s_{2}s_{1}$ & $s_{4}s_{1}s_{2}$ & $\alpha_{2}^{\vee} + \alpha_{3}^{\vee} + \alpha_{4}^{\vee}$ \\ \hline
\end{tabular}
\end{table}

By Theorem~\ref{qkchev}, in $QK_{T}^{\mathrm{poly}}(G/B)$, we have: 
\begin{equation} \label{eq:Chevalley_full_example_2-step}
\begin{split}
& [\cO(-\vpi_{2})] \cdot [\cO^{s_{4}s_{5}s_{1}s_{2}s_{3}s_{4}s_{3}s_{2}}] \\ 
&= \be^{s_{4}s_{5}s_{1}s_{2}s_{3}s_{4}s_{3}s_{2}\vpi_{2}} ( [\cO^{s_{4}s_{5}s_{1}s_{2}s_{3}s_{4}s_{3}s_{2}}] - [\cO^{s_{4}s_{5}s_{1}s_{2}s_{3}s_{4}s_{3}s_{1}s_{2}}] \\ 
& \quad - Q_{2}Q_{3}Q_{4} [\cO^{s_{4}s_{5}s_{1}}] - Q_{2} [\cO^{s_{4}s_{5}s_{1}s_{2}s_{3}s_{4}s_{3}}] + Q_{2}Q_{3}Q_{4} [\cO^{s_{4}s_{5}s_{2}s_{1}}] \\ 
& \quad + Q_{2} [\cO^{s_{4}s_{5}s_{1}s_{2}s_{3}s_{4}s_{3}s_{1}}] + Q_{2}Q_{3}Q_{4} [\cO^{s_{4}s_{5}s_{1}s_{2}}] - Q_{2}Q_{3}Q_{4} [\cO^{s_{4}s_{5}s_{1}s_{2}s_{1}}]). 
\end{split}
\end{equation}
By applying the surjection $\Phi_{J}: QK_{T}^{\mathrm{poly}}(G/B) \rightarrow QK_{T}^{\mathrm{poly}}(G/P_{J})$ to equation~\eqref{eq:Chevalley_full_example_2-step}, we obtain the following cancellation-free formula in $QK_{T}^{\mathrm{poly}}(G/P_{J}) \subset QK_{T}(G/P_{J})$; here, the underlined terms in the first equality are canceled out: 
\begin{align}
\begin{split}
& [\cO(-\vpi_{2})] \cdot [\cO^{s_{4}s_{5}s_{1}s_{2}s_{3}s_{4}s_{3}s_{2}}] \\ 
&= \be^{s_{4}s_{5}s_{1}s_{2}s_{3}s_{4}s_{3}s_{2}\vpi_{2}} ( [\cO^{s_{4}s_{5}s_{1}s_{2}s_{3}s_{4}s_{3}s_{2}}] - [\cO^{s_{4}s_{5}s_{1}s_{2}s_{3}s_{4}s_{3}s_{1}s_{2}}] \\ 
& \quad - Q_{2}Q_{3}Q_{4} [\cO^{s_{4}}] \underline{{}- Q_{2} [\cO^{s_{4}s_{5}s_{1}s_{2}s_{3}s_{4}}]} + Q_{2}Q_{3}Q_{4} [\cO^{s_{4}s_{2}}] \\ 
& \quad \underline{{}+ Q_{2} [\cO^{s_{4}s_{5}s_{1}s_{2}s_{3}s_{4}}]} \underline{{}+ Q_{2}Q_{3}Q_{4} [\cO^{s_{4}s_{1}s_{2}}]} \underline{{}- Q_{2}Q_{3}Q_{4} [\cO^{s_{4}s_{1}s_{2}}]}). 
\end{split} \\ 
&= \be^{s_{4}s_{5}s_{1}s_{2}s_{3}s_{4}s_{3}s_{2}\vpi_{2}} ( [\cO^{s_{4}s_{5}s_{1}s_{2}s_{3}s_{4}s_{3}s_{2}}] - [\cO^{s_{4}s_{5}s_{1}s_{2}s_{3}s_{4}s_{3}s_{1}s_{2}}] - Q_{2}Q_{4} [\cO^{s_{4}}] + Q_{2}Q_{4} [\cO^{s_{4}s_{2}}]). 
\end{align}

Also, we deduce that $\A_{\lessdot}(w, \Gamma(2)) = \{ \emptyset, \{4\} \}$. Therefore, we see that 
\begin{align}
& \text{(RHS of the equation in Theorem~\ref{thm:2-step_3}\,(2)(b))} \\ 
\begin{split}
&= \be^{s_{4}s_{5}s_{1}s_{2}s_{3}s_{4}s_{3}s_{2}\vpi_{2}} ( ([\cO^{s_{4}s_{5}s_{1}s_{2}s_{3}s_{4}s_{3}s_{2}}] - Q_{2}Q_{4} [\cO^{\mcr{s_{4}s_{5}s_{1}}}]) \\ 
& \quad - ([\cO^{s_{4}s_{5}s_{1}s_{2}s_{3}s_{4}s_{3}s_{1}s_{2}}] - Q_{2}Q_{4} [\cO^{\mcr{s_{4}s_{5}s_{2}s_{1}}}]) )
\end{split} \\ 
&= \be^{s_{4}s_{5}s_{1}s_{2}s_{3}s_{4}s_{3}s_{2}\vpi_{2}} ([\cO^{s_{4}s_{5}s_{1}s_{2}s_{3}s_{4}s_{3}s_{2}}] - Q_{2}Q_{4} [\cO^{s_{4}}] - [\cO^{s_{4}s_{5}s_{1}s_{2}s_{3}s_{4}s_{3}s_{1}s_{2}}] + Q_{2}Q_{4} [\cO^{s_{4}s_{2}}]) \\ 
&= [\cO(-\vpi_{2})] \cdot [\cO^{s_{4}s_{5}s_{1}s_{2}s_{3}s_{4}s_{3}s_{2}}]. 
\end{align}
Thus Theorem~\ref{thm:2-step_3}\,(2)(b) holds in this case. 
\end{exa}

\begin{rema}
In \cite[Theorem~4.5]{Xu}, Xu obtained a Chevalley formula for \emph{incidence varieties}, that is, for the two-step flag manifold $G/P_{J}$ in the case that $J = I \setminus \{1, n-1\}$, by a completely different method of proof than ours of Theorems~\ref{thm:2-step_1}, \ref{thm:2-step_2}, \ref{thm:2-step_3}, \ref{thm:2-step_4}, \ref{thm:2-step_5}, and \ref{thm:2-step_6}. We can verify that in this case, our Chevalley formula coincides with the one in \cite[Theorem~4.5]{Xu} for incidence varieties. As an example, we compare Theorem~20\,(2) with \cite[Equation\,(9) of Theorem~4.5]{Xu}; this is a most complicated case. 
As for Theorems~\ref{thm:2-step_1} and \ref{thm:2-step_2}\,(1), we can also compare our formulas and Xu's ones by the same argument as below. As for Theorem~\ref{thm:2-step_3}, $w$ should be be the unique element of $W^{J}$ such that $w(1) = n$ and $w(n) = 1$, and hence we can compare the formulas by direct calculation. 
As for Theorems~\ref{thm:2-step_4}, \ref{thm:2-step_5}, and \ref{thm:2-step_6}, we can show the coincidence of the formulas from that of the formulas in Theorems~\ref{thm:2-step_1}, \ref{thm:2-step_2}, and \ref{thm:2-step_3} by applying the diagram automorphism $\omega$ (see Section~\ref{subsec:proof1}). 

Throughout this remark, we assume that $k_{1} = 1$, $k_{2} = n-1$. 
Note that under this assumption, for $1 \le i, j \le n$ with $i \not= j$, there exists a unique $w \in W^{J}$ such that $w(1) = i$ and $w(n) = j$; in such a case, we write $w = [i, j]$, as in \cite{Xu}. 

We assume that $w \in W^{J}$ satisfies the following: 
\begin{itemize}
\item $w(k_{1}) > w(k_{1}+1)$, 
\item condition (Q) does not hold, 
\item $w(1) > w(k_{1}+1)$, 
\item $w(k_{1}) > w(k_{2})$, 
\end{itemize}
and set $i := w(1)$, $j := w(n)$ (i.e., $w = [i, j]$). Under these assumptions, we see that $i+1 \equiv j \mod n$ if and only if $i = n-1$ and $j =n$ (i.e., $w = [n-1, n]$). 

Let us compute the product $[\cO^{s_{1}}] \cdot [\cO^{w}]$ by our Chevalley formula. 
Recall that 
\begin{equation}
\Gamma(1) = ((1,n), (1,n-1), \ldots, (1,2))
\end{equation}
(with all roots negated). First, assume that $w = [n-1, n]$. Then, by Proposition~2, we deduce that $\A_{\lessdot}([n-1, n], \Gamma(1)) = \{ \emptyset, \{(1, n)\}\}$. By Theorem~20\,(2), we compute: 
\begin{align}
& [\cO(-\vpi_{1})] \cdot [\cO^{[n-1, n]}] \\ 
&= \be^{[n-1, n]\vpi_{1}} \sum_{A \in \A_{\lessdot}([n-1, n], \Gamma(1))} (-1)^{|A|} \left( [\cO^{\wend([n-1, n], A)}] - Q_{1} [\cO^{\mcr{\wend([n-1, n], A)s_{1}}}] \right) \\ 
&= \be^{\ve_{n-1}} \left( \underbrace{\left( [\cO^{[n-1, n]}] - Q_{1}[\cO^{[1,n]}] \right)}_{A = \emptyset} - \underbrace{\left( [\cO^{[n, n-1]}] - Q_{1}[\cO^{[1, n-1]}] \right)}_{A = \{(1, n)\}} \right). 
\end{align}
By the well-known formula $[\cO^{s_{1}}] = 1 - \be^{-\vpi_{1}} [\cO(-\vpi_{1})]$ (cf., \cite[Section~4.1]{bcmcfe}), we see that 
\begin{align}
& [\cO^{s_{1}}] \cdot [\cO^{[n-1, n]}] \\ 
&= (1 - \be^{-\ve_{1}} [\cO(-\vpi_{1})]) \cdot [\cO^{[n-1, n]}] \\ 
&= [\cO^{[n-1, n]}] - \be^{-\ve_{1}} [\cO(-\vpi_{1})] \cdot [\cO^{[n-1, n]}] \\ 
&= [\cO^{[n-1, n]}] - \be^{\ve_{n-1} - \ve_{1}} \left( \left( [\cO^{[n-1, n]}] - Q_{1}[\cO^{[1,n]}] \right) - \left( [\cO^{[n, n-1]}] - Q_{1}[\cO^{[1, n-1]}] \right) \right) \\ 
&= (1- \be^{\ve_{n-1} - \ve_{1}}) [\cO^{[n-1, n]}] + \be^{\ve_{n-1}-\ve_{1}} \left( Q_{1} [\cO^{[1, n]}] + [\cO^{[n, n-1]}] - Q_{1}[\cO^{[1, n-1]}] \right). 
\end{align}
This result agrees with the second equation of \cite[Equation\,(9) of Theorem~4.5]{Xu}. 

Next, we consider the case $w = [i, j] \not= [n-1, n]$. In this case, we see that $i+1 \not\equiv j \mod n$. Since condition (Q) does not hold, we have $w(n) \not = 1$. Also, we have $w(n) \not= n$; this is because if $w(n) = n$, then $w$ must be $[n-1, n]$ under our assumptions. These facts imply that $w(1) = i = n$ and $w(2) = 1$. By Proposition~2, we deduce that $\A_{\lessdot}(w, \Gamma(1)) = \{ \emptyset \}$. Therefore, we compute: 
\begin{align}
& [\cO(-\vpi_{1})] \cdot [\cO^{[n, j]}] \\ 
&= \be^{w\vpi_{1}} \sum_{A \in \A_{\lessdot}([n, j], \Gamma(1))} (-1)^{|A|} \left( [\cO^{\wend([n, j], A)}] - Q_{1} [\cO^{\mcr{\wend([n, j], A)s_{1}}}] \right) \\ 
&= \be^{\ve_{n}} \left( [\cO^{[n, j]}] - Q_{1}[\cO^{[1, j]}] \right). 
\end{align}
Again, since $[\cO^{s_{1}}] = 1 - \be^{-\vpi_{1}} [\cO(-\vpi_{1})]$, we see that 
\begin{align}
& [\cO^{s_{1}}] \cdot [\cO^{[n, j]}] \\ 
&= (1 - \be^{-\vpi_{1}} [\cO(-\vpi_{1})]) \cdot [\cO^{[n, j]}] \\ 
&= [\cO^{[n, j]}] - \be^{-\ve_{1}} [\cO(-\vpi_{1})] \cdot [\cO^{[n, j]}] \\ 
&= [\cO^{[n, j]}] - \be^{\ve_{n} - \ve_{1}} \left( [\cO^{[n, j]}] - Q_{1}[\cO^{[1, j]}] \right) \\ 
&= (1 - \be^{\ve_{n} - \ve_{1}}) [\cO^{[n, j]}] + \be^{\ve_{n} - \ve_{1}} Q_{1} [\cO^{[1, j]}]. 
\end{align}
This result agrees with the first equation of \cite[Equation\,(9) of Theorem~4.5]{Xu}. 
\end{rema}

\subsection{Proofs of parabolic Chevalley formulas: Part~1} \label{subsec:proof1}

In this and the next subsection, we give proofs of the results stated in the previous subsection. 
Since Theorems~\ref{thm:2-step_4}, \ref{thm:2-step_5}, and \ref{thm:2-step_6} follow from Theorems~\ref{thm:2-step_1}, \ref{thm:2-step_2}, and \ref{thm:2-step_3}, respectively, by applying the diagram automorphism $\omega: [n-1] \rightarrow [n-1]$, it suffices to prove Theorems~\ref{thm:2-step_1}, \ref{thm:2-step_2}, and \ref{thm:2-step_3}. 
Note that the diagram automorphism $\omega$ induces a group automorphism $\omega: W \xrightarrow{\sim} W$, $s_{l} \mapsto s_{\omega(l)}$, together with a linear automorphism $\omega: \mathfrak{h}_{\mathbb{R}}^{*} \xrightarrow{\sim} \mathfrak{h}_{\mathbb{R}}^{*}$, $\varpi_{l} \mapsto \varpi_{\omega(l)}$, and also an isomorphism $\omega: G/P_{J} \xrightarrow{\sim} G/P_{\omega(J)}$ of varieties; recall that $G$ is simply-connected. 
Hence, as mentioned in \cite[Sections 8.1 and 8.3]{mnsdem}, we see that there exists a $\bZ$-module isomorphism $\omega: QK_{T}(G/P_{J}) \xrightarrow{\sim} QK_{T}(G/P_{\omega(J)})$ such that
\begin{equation}
\be^{\mu} [\cO^{w}] \mapsto \be^{\omega(\mu)} [\cO^{\omega(w)}]  
\end{equation}
for $\mu \in \Lambda$, $w \in W^{J}$, and such that $\omega(Q_{l}) = Q_{\omega(l)}$ for $l \in I \setminus J$. 
In this subsection, we give proofs of Theorems~\ref{thm:2-step_1} and \ref{thm:2-step_2}. 

By Remark~\ref{rem:Chevalley_cancellation-free}, we obtain the following. 
\begin{lem} \label{lem:2-step_cancellation-free_Bruhat}
The sum 
\begin{equation}
\be^{w\varpi_{k_{1}}} \sum_{A \in \A_{\lessdot}(w, \Gamma(k_{1}))} (-1)^{|A|} [\cO^{\wend(w, A)}]
\end{equation}
is cancellation-free.
\end{lem}

Also, by making use of  Proposition~\ref{prop:quantum_bruhat_order_type_A}, 
we can verify the following.
\begin{lem} \label{lem:condition_nonempty}
The following hold. 
\begin{enumerate}
\item We have $\A_{1}(w, \Gamma(k_{1})) \not= \emptyset$ if and only if $w(k_{1}) > w(k_{1}+1)$. 
\item We have $\A_{2}(w, \Gamma(k_{1})) \not= \emptyset$ if and only if condition \textup{(Q)} holds. 
\end{enumerate}
\end{lem}

\begin{rema}
It is obvious that $\A_{2}(w, \Gamma(k_{1})) \not= \emptyset$ if and only if $\A_{3}(w, \Gamma(k_{1})) \not= \emptyset$. 
\end{rema}

\begin{rema}
If $w(k_{1}) > w(k_{1}+1)$, then we have 
\begin{equation} \label{eq:A1}
\A_{1}(w, \Gamma(k_{1})) = \{ A \sqcup \{(k_{1}, k_{1}+1)\} \mid A \in \A_{\lessdot}(w, \Gamma(k_{1})) \}\,. 
\end{equation}
Also, if condition (Q) holds, then we have 
\begin{align}
\A_{2}(w, \Gamma(k_{1})) &= \{ A \sqcup \{(k_{1}, k_{2}+1)\} \mid A \in \A_{\lessdot}(w, \Gamma(k_{1})) \}\,, \label{eq:A2} \\ 
\A_{3}(w, \Gamma(k_{1})) &= \{ A \sqcup \{(k_{1}, k_{2}+1), (k_{1}, k_{1}+1)\} \mid A \in \A_{\lessdot}(w, \Gamma(k_{1})) \}\,. \label{eq:A3} 
\end{align}
\end{rema}

\begin{proof}[Proof of Theorem~\ref{thm:2-step_1}]
By Lemma~\ref{lem:condition_nonempty}, we have $\A(w, \Gamma(k_{1})) = \A_{\lessdot}(w, \Gamma(k_{1}))$. Therefore, the theorem follows from Lemma~\ref{lem:2-step_cancellation-free_Bruhat}. 
\end{proof}

In the rest of this subsection, we assume that $w(k_{1}) > w(k_{1}+1)$, and assume that condition (Q) does not hold. 
Hence we have $\A_{2}(w, \Gamma(k_{1})) = \A_{3}(w, \Gamma(k_{1})) = \emptyset$. 

First, assume that $w(1) < w(k_{1}+1)$. Take the maximal $1 \leq p \leq k_{1}$ such that $w(p) < w(k_{1}+1)$. Then, we can define an involution $\iota$ on $\A_{1}(w, \Gamma(k_{1}))$ as follows: set 
\begin{align}
\A_{1}^{1}(w, \Gamma(k_{1})) &:= \{ A \in \A_{1}(w, \Gamma(k_{1})) \mid (p, k_{1}+1) \in A \}\,, \\ 
\A_{1}^{2}(w, \Gamma(k_{1})) &:= \{ A \in \A_{1}(w, \Gamma(k_{1})) \mid (p, k_{1}+1) \notin A \}\,, 
\end{align}
and define $\iota$ by 
\begin{align}
A \in \A_{1}^{2}(w, \Gamma(k_{1})) \ &\mapsto \ \iota(A) := A \sqcup \{(p, k_{1}+1)\} \in \A_{1}^{1}(w, \Gamma(k_{1}))\,, \\ 
A \in \A_{1}^{1}(w, \Gamma(k_{1})) \ &\mapsto \ \iota(A) := A \setminus \{(p, k_{1}+1)\} \in \A_{1}^{2}(w, \Gamma(k_{1}))\,. 
\end{align}
This $\iota$ has the following properties: 
\begin{itemize}
\item $\wend(w, \iota(A)) = \wend(w, A)(p, k_{1})$ (and hence $\mcr{\wend(w, \iota(A))} = \mcr{\wend(w, A)}$); 
\item $|\iota(A)| = |A| \pm 1$. 
\end{itemize}
By using the involution $\iota$, we obtain the following. 

\begin{lem} \label{lem:2-step_cancellation_2-1-1}
Assume that $w(k_{1}) > w(k_{1}+1)$, and assume that condition \textup{(Q)} does not hold. If $w(1) < w(k_{1}+1)$, then 
\begin{equation}
\be^{w\varpi_{k_{1}}} \sum_{A \in \A_{1}(w, \Gamma(k_{1}))} (-1)^{|A|} Q_{k_{1}} [\cO^{\mcr{\wend(w, A)}}] = 0\,. 
\end{equation}
\end{lem}

\begin{rema} \label{rem:2-step_cancellation_A1}
Even if we assume condition (Q), the identity in Lemma~\ref{lem:2-step_cancellation_2-1-1} is still valid, if all the conditions of this lemma other than the negation of condition (Q) hold. This is because the involution $\iota$ above is well-defined whether or not condition (Q) holds. 
\end{rema}

Next, assume that $w(k_{1}) < w(k_{2})$. Take the minimal $k_{1}+1 \leq q \leq k_{2}$ such that $w(k_{1}) < w(q)$. Then, we can define an involution $\iota$ on $\A_{1}(w, \Gamma(k_{1}))$ as follows: set 
\begin{align}
\A_{1}^{1}(w, \Gamma(k_{1})) &:= \{ A \in \A_{1}(w, \Gamma(k_{1})) \mid (k_{1}, q) \in A \}\,, \\ 
\A_{1}^{2}(w, \Gamma(k_{1})) &:= \{ A \in \A_{1}(w, \Gamma(k_{1})) \mid (k_{1}, q) \notin A \}\,, \\ 
\end{align}
and define $\iota$ by 
\begin{align}
A \in \A_{1}^{2}(w, \Gamma(k_{1})) \ &\mapsto \ \iota(A) := A \sqcup \{(k_{1}, q)\} \in \A_{1}^{1}(w, \Gamma(k_{1}))\,, \\ 
A \in \A_{1}^{1}(w, \Gamma(k_{1})) \ &\mapsto \ \iota(A) := A \setminus \{(k_{1}, q)\} \in \A_{1}^{2}(w, \Gamma(k_{1}))\,. 
\end{align}
This $\iota$ has the following properties: 
\begin{itemize}
\item $\wend(w, \iota(A)) = \wend(w, A)(k_{1}+1, q)$ (and hence $\mcr{\wend(w, \iota(A))} = \mcr{\wend(w, A)}$); 
\item $|\iota(A)| = |A| \pm 1$. 
\end{itemize}
By using the involution $\iota$, we obtain the following. 

\begin{lem} \label{lem:2-step_cancellation_2-1-2}
Assume that $w(k_{1}) > w(k_{1}+1)$, and assume that condition \textup{(Q)} does not hold. If $w(k_{1}) < w(k_{2})$, then 
\begin{equation}
\be^{w\varpi_{k_{1}}} \sum_{A \in \A_{1}(w, \Gamma(k_{1}))} (-1)^{|A|} Q_{k_{1}} [\cO^{\mcr{\wend(w, A)}}] = 0\,. 
\end{equation}
\end{lem}

\begin{proof}[Proof of Theorem~\ref{thm:2-step_2}\,(1)]
By Lemmas~\ref{lem:2-step_cancellation_2-1-1} and \ref{lem:2-step_cancellation_2-1-2}, we deduce that 
\begin{equation}
\be^{w\varpi_{k_{1}}} \sum_{A \in \A_{1}(w, \Gamma(k_{1}))} (-1)^{|A|} Q_{k_{1}} [\cO^{\mcr{\wend(w, A)}}] = 0\,. 
\end{equation}
Therefore, we obtain the desired cancellation-free formula from Lemma~\ref{lem:2-step_cancellation-free_Bruhat}, together with the fact that $\A_{2}(w, \Gamma(k_{1})) = \A_{3}(w, \Gamma(k_{1})) = \emptyset$. 
\end{proof}

We assume that $w(1) > w(k_{1}+1)$ and $w(k_{1}) > w(k_{2})$ until the end of 
this subsection. 
Let $A \in \A_{1}(w, \Gamma(k_{1}))$, and set $y := \wend(w, A \setminus \{(k_{1}, k_{1}+1)\})$. Since $A \setminus \{(k_{1}, k_{1}+1)\}$ contains only Bruhat steps, we see that $y(k_{1}+1) < y(1)$ and $y(k_{2}) < y(k_{1})$. Therefore, if we set $z := ys_{k_{1}} = \wend(w, A)$, then we have 
\begin{itemize}
\item $z(k_{1}) < z(1) < z(2) < \cdots < z(k_{1}-1)$, 
\item $z(k_{1}+2) < z(k_{1}+3) < \cdots < z(k_{2}) < z(k_{1}+1)$, and 
\item $z(k_{2}+1) < z(k_{2}+2) < \cdots < z(n)$; 
\end{itemize}
hence, if we take cyclic permutations $\sigma_{1} := (1, k_{1}, k_{1}-1, \ldots, 2)$ (if $k_{1} = 1$, then we take $\sigma_{1} := e$, the identity permutation) 
and $\sigma_{2} := (k_{1}+1, k_{1}+2, \ldots, k_{2})$ (if $k_{1} + 1 = k_{2}$, then we take $\sigma_{2} := e$), 
then we deduce that $\mcr{z} = z\sigma_{1}\sigma_{2}$. 
Note that the definitions of $\sigma_{1}$ and $\sigma_{2}$ do not depend on the choice of $A$. Thus, for $A, B \in \A_{1}(w, \Gamma(k_{1}))$ with $A \not= B$, it follows that 
\begin{equation}
\mcr{\wend(w, A)} = \wend(w, A)\sigma_{1}\sigma_{2} \not= \wend(w, B)\sigma_{1}\sigma_{2} = \mcr{\wend(w, B)}\,. 
\end{equation}
Since the right-hand side of equation \eqref{qkchev-f} in Theorem~\ref{qkchev} is cancellation-free, as mentioned in Remark~\ref{rem:Chevalley_cancellation-free},
this proves the following. 

\begin{lem} \label{lem:2-step_cancellation-free_2-2}
Assume that $w(k_{1}) > w(k_{1}+1)$, and assume that condition \textup{(Q)} does not hold. If $w(1) > w(k_{1}+1)$ and $w(k_{1}) > w(k_{2})$, then the sum 
\begin{equation}\label{eq:2-step_A1}
\be^{w\varpi_{k_{1}}} \sum_{A \in \A_{1}(w, \Gamma(k_{1}))} (-1)^{|A|} Q_{k_{1}} [\cO^{\mcr{\wend(w, A)}}] 
\end{equation}
is cancellation-free. 
\end{lem}

\begin{rema} \label{rem:2-step_cancellation-free_A1}
Note that we do not use the negation of condition (Q) in the proof of Lemma~\ref{lem:2-step_cancellation-free_2-2}. Hence the sum \eqref{eq:2-step_A1} is cancellation-free whether or not we assume condition (Q), if all the conditions of Lemma~\ref{lem:2-step_cancellation-free_2-2} other than the negation of (Q) hold. 
\end{rema}

\begin{rema} \label{rem:2-step_deformation_A1}
If $\A_{1}(w, \Gamma(k_{1})) \not= \emptyset$ (or equivalently, $w(k_{1}) > w(k_{1}+1)$), then equation~\eqref{eq:A1} shows that 
\begin{align}
& \be^{w\varpi_{k_{1}}} \sum_{A \in \A_{1}(w, \Gamma(k_{1}))} (-1)^{|A|} Q_{k_{1}} [\cO^{\mcr{\wend(w, A)}}] \\ 
&= - \be^{w\varpi_{k_{1}}} \sum_{A \in \A_{\lessdot}(w, \Gamma(k_{1}))} (-1)^{|A|} Q_{k_{1}} [\cO^{\mcr{\wend(w, A)s_{k_{1}}}}]\,. 
\end{align}
\end{rema}

\begin{proof}[Proof of Theorem~\ref{thm:2-step_2}\,(2)]
The desired identity follows from Lemmas~\ref{lem:2-step_cancellation-free_Bruhat}, \ref{lem:2-step_cancellation-free_2-2}, and Remark~\ref{rem:2-step_deformation_A1}, together with the fact that $\A_{2}(w, \Gamma(k_{1})) = \A_{3}(w, \Gamma(k_{1})) = \emptyset$. 
\end{proof}

\subsection{Proofs of parabolic Chevalley formulas: Part~2} \label{subsec:proof2}

In this subsection, we give a proof of Theorem~\ref{thm:2-step_3}; since we assume condition (Q), we have $w(k_{1}) > w(k_{1}+1)$; see Remark~\ref{rem:k1>k1+1}. 

First, assume that $w(k_{1}) < w(n)$. Then, we can take the minimal $k_{2}+1 \leq q \leq n$ such that $w(k_{1}) < w(q)$, and define an involution $\iota$ on $\A_{l}(w, \Gamma(k_{1}))$, $l = 2, 3$, as follows: for each $l = 2, 3$, we set 
\begin{align}
\A_{l}^{1}(w, \Gamma(k_{1})) &:= \{ A \in \A_{l}(w, \Gamma(k_{1})) \mid (k_{1}, q) \in A \}\,, \\ 
\A_{l}^{2}(w, \Gamma(k_{1})) &:= \{ A \in \A_{l}(w, \Gamma(k_{1})) \mid (k_{1}, q) \notin A \}\,, 
\end{align}
and define $\iota$ by 
\begin{align}
A \in \A_{l}^{2}(w, \Gamma(k_{1})) \ &\mapsto \ \iota(A) := A \sqcup \{(k_{1}, q)\} \in \A_{l}^{1}(w, \Gamma(k_{1}))\,, \\ 
A \in \A_{l}^{1}(w, \Gamma(k_{1})) \ &\mapsto \ \iota(A) := A \setminus \{(k_{1}, q)\} \in \A_{l}^{2}(w, \Gamma(k_{1}))\,. 
\end{align}
This $\iota$ has the following properties: 
\begin{itemize}
\item $\wend(w, \iota(A)) = \wend(w, A)(k_{2}+1, q)$ (and hence $\mcr{\wend(w, \iota(A))} = \mcr{\wend(w, A)}$); 
\item $|\iota(A)| = |A| \pm 1$. 
\end{itemize}
By using the involution $\iota$, we obtain the following. 

\begin{lem} \label{lem:2-step_cancellation_3-1}
Assume condition \textup{(Q)}. If $w(k_{1}) < w(n)$, then for $l = 2, 3$, 
\begin{equation}
\be^{w\varpi_{k_{1}}} \sum_{A \in \A_{l}(w, \Gamma(k_{1}))} (-1)^{|A|} Q_{k_{1}}Q_{k_{2}} [\cO^{\mcr{\wend(w, A)}}] = 0\,. 
\end{equation}
\end{lem}

\begin{proof}[Proof of Theorem~\ref{thm:2-step_3}\,(1)]
By Lemma~\ref{lem:2-step_cancellation_3-1}, we deduce that 
\begin{equation}
\begin{split}
[\cO(-\varpi_{k_{1}})] \cdot [\cO^{w}] &= \be^{w\varpi_{k_{1}}} \sum_{A \in \A_{\lessdot}(w, \Gamma(k_{1}))} (-1)^{|A|} [\cO^{\wend(w, A)}] \\ 
& \quad + \be^{w\varpi_{k_{1}}} \sum_{A \in \A_{1}(w, \Gamma(k_{1}))} (-1)^{|A|} Q_{k_{1}} [\cO^{\mcr{\wend(w, A)}}]\,. 
\end{split}
\end{equation}

If $w(1) < w(k_{1}+1)$, then 
\begin{equation}
\be^{w\varpi_{k_{1}}} \sum_{A \in \A_{1}(w, \Gamma(k_{1}))} (-1)^{|A|} Q_{k_{1}} [\cO^{\mcr{\wend(w, A)}}] = 0\, 
\end{equation}
by Remark~\ref{rem:2-step_cancellation_A1}. Therefore, Theorem~\ref{thm:2-step_3}\,(1)\,(a) follows from Lemma~\ref{lem:2-step_cancellation-free_Bruhat}. 

Assume now that $w(1) > w(k_{1}+1)$. Note that $w(k_{1}) > w(k_{2})$ by condition (Q). Hence Remark~\ref{rem:2-step_cancellation-free_A1} implies that the sum 
\begin{equation}
\be^{w\varpi_{k_{1}}} \sum_{A \in \A_{1}(w, \Gamma(k_{1}))} (-1)^{|A|} Q_{k_{1}} [\cO^{\mcr{\wend(w, A)}}] 
\end{equation}
is cancellation-free. Therefore, Theorem~\ref{thm:2-step_3}\,(1)\,(b) follows from Lemma~\ref{lem:2-step_cancellation-free_Bruhat} and Remark~\ref{rem:2-step_deformation_A1}. 
\end{proof}

Next, assume that $w(k_{1}) > w(n)$. We consider the following auxiliary condition: 
\begin{itemize}
\item[(Q-A)] there exists $1 \leq l \leq k_{1}$ such that $w(k_{2}+1) < w(l) < w(k_{1}+1)$. 
\end{itemize}

Assume condition (Q-A), and that $w(1) < w(k_{2}+1)$. We take the maximal $1 \leq p_{\A_{2}} \leq k_{1}$ such that $w(p_{\A_{2}}) < w(k_{2}+1)$. Then, we can define an involution $\iota_{\A_{2}}$ on $\A_{2}(w, \Gamma(k_{1}))$ as follows: set 
\begin{align}
\A_{2}^{1}(w, \Gamma(k_{1})) &:= \{ A \in \A_{2}(w, \Gamma(k_{2})) \mid (p_{\A_{2}}, k_{2}+1) \in A \}\,, \\ 
\A_{2}^{2}(w, \Gamma(k_{1})) &:= \{ A \in \A_{2}(w, \Gamma(k_{2})) \mid (p_{\A_{2}}, k_{2}+1) \notin A \}\,, 
\end{align}
and define $\iota_{\A_{2}}$ by 
\begin{align}
A \in \A_{2}^{2}(w, \Gamma(k_{1})) \ &\mapsto \ \iota_{\A_{2}}(A) := A \sqcup \{(p_{\A_{2}}, k_{2}+1)\} \in \A_{2}^{1}(w, \Gamma(k_{1}))\,, \\ 
A \in \A_{2}^{1}(w, \Gamma(k_{1})) \ &\mapsto \ \iota_{\A_{2}}(A) := A \setminus \{(p_{\A_{2}}, k_{2}+1)\} \in \A_{2}^{2}(w, \Gamma(k_{1}))\,. 
\end{align}
\begin{rema}
If condition (Q-A) does not hold, then the above $\iota_{\A_{2}}: \A_{2}^{1}(w, \Gamma(k_{1})) \rightarrow \A_{2}^{2}(w, \Gamma(k_{1}))$ is not well-defined; we will explain this situation later. 
\end{rema}

This $\iota_{\A_{2}}$ has the following properties: 
\begin{itemize}
\item $\wend(w, \iota_{\A_{2}}(A)) = \wend(w, A)(p_{\A_{2}}, k_{1})$ (and hence $\mcr{\wend(w, \iota_{\A_{2}}(A))} = \mcr{\wend(w, A)}$); 
\item $|\iota_{\A_{2}}(A)| = |A| \pm 1$. 
\end{itemize}
By using the involution $\iota_{\A_{2}}$, we obtain the following. 

\begin{lem} \label{lem:2-step_cancellation_3-2-1-A2}
Assume condition \textup{(Q)}. If $w(1) < w(k_{2}+1)$ and condition \textup{(Q-A)} hold, then 
\begin{equation}
\be^{w\varpi_{k_{1}}} \sum_{A \in \A_{2}(w, \Gamma(k_{1}))} (-1)^{|A|} Q_{k_{1}}Q_{k_{2}} [\cO^{\mcr{\wend(w, A)}}] = 0\,. 
\end{equation}
\end{lem}

Also, we take the maximal $1 \leq p_{\A_{3}} \leq k_{1}$ such that $w(p_{\A_{3}}) < w(k_{1}+1)$. We can define an involution $\iota_{\A_{3}}$ on $\A_{3}(w, \Gamma(k_{1}))$ as follows: set 
\begin{align}
\A_{3}^{1}(w, \Gamma(k_{1})) &:= \{ A \in \A_{3}(w, \Gamma(k_{1})) \mid (p_{\A_{3}}, k_{1}+1) \in A \}\,, \\ 
\A_{3}^{2}(w, \Gamma(k_{1})) &:= \{ A \in \A_{3}(w, \Gamma(k_{1})) \mid (p_{\A_{3}}, k_{1}+1) \notin A \}\,, 
\end{align}
and define $\iota_{\A_{3}}$ by 
\begin{align}
A \in \A_{3}^{2}(w, \Gamma(k_{1})) \ &\mapsto \ \iota_{\A_{3}}(A) := A \sqcup \{(p_{\A_{3}}, k_{1}+1)\} \in \A_{3}^{1}(w, \Gamma(k_{1}))\,, \\ 
A \in \A_{3}^{1}(w, \Gamma(k_{1})) \ &\mapsto \ \iota_{\A_{3}}(A) := A \setminus \{(p_{\A_{3}}, k_{1}+1)\} \in \A_{3}^{2}(w, \Gamma(k_{1}))\,. 
\end{align}

\begin{rema}
If condition (Q-A) does not hold, then the above $\iota_{\A_{3}}: \A_{3}^{2}(w, \Gamma(k_{1})) \rightarrow \A_{3}^{1}(w, \Gamma(k_{1}))$ is not well-defined for the same reason as $\iota_{\A_{2}}$. 
\end{rema}

This $\iota_{\A_{3}}$ has the following properties: 
\begin{itemize}
\item $\wend(w, \iota_{\A_{3}}(A)) = \wend(w, A)(p_{\A_{3}}, k_{1})$ (and hence $\mcr{\wend(w, \iota_{\A_{3}}(A))} = \mcr{\wend(w, A)}$); 
\item $|\iota_{\A_{3}}(A)| = |A| \pm 1$. 
\end{itemize}
By using the involution $\iota_{\A_{3}}$, we obtain the following. 

\begin{lem} \label{lem:2-step_cancellation_3-2-1-A3}
Assume condition \textup{(Q)}. If $w(1) < w(k_{2}+1)$ and condition \textup{(Q-A)} holds, then 
\begin{equation}
\be^{w\varpi_{k_{1}}} \sum_{A \in \A_{3}(w, \Gamma(k_{1}))} (-1)^{|A|} Q_{k_{1}}Q_{k_{2}} [\cO^{\mcr{\wend(w, A)}}] = 0\,. 
\end{equation}
\end{lem}

Next, assume that condition (Q-A) does not hold, but assume that $w(1) < w(k_{2}+1)$. Take the maximal $1 \leq p \leq k_{1}$ such that $w(p) < w(k_{2}+1)$. Set 
\begin{align}
\A_{2}'(w, \Gamma(k_{1})) &:= \{ A \in \A_{2}(w, \Gamma(k_{1})) \mid (p, k_{1}+1) \in A \}\,, \\ 
\A_{2}'^{C}(w, \Gamma(k_{1})) &:= \A_{2}(w, \Gamma(k_{1})) \setminus \A_{2}'(w, \Gamma(k_{1}))\,. 
\end{align}
Observe that if $A \in \A_{2}'(w, \Gamma(k_{1}))$, then we must have $(p, k_{2}+1) \in A$; if not, then $A$ cannot contain a quantum step $(k_{1}, k_{2}+1)$, which contradicts the definition of $\A_{2}(w, \Gamma(k_{1}))$. Thus, the above $\iota_{\A_{2}}: \A_{2}^{1}(w, \Gamma(k_{1})) \rightarrow \A_{2}^{2}(w, \Gamma(k_{1}))$ is not well-defined. Hence we need another involution. 

In fact, we can define an involution on $\A_{2}'^{C}(w, \Gamma(k_{1}))$ similar to $\iota_{\A_{2}}$ as follows. We set 
\begin{align}
\A_{2}'^{C, 1}(w, \Gamma(k_{1})) &:= \{ A \in \A_{2}'^{C}(w, \Gamma(k_{1})) \mid (p, k_{2}+1) \in A \}\,, \\ 
\A_{2}'^{C, 2}(w, \Gamma(k_{1})) &:= \{ A \in \A_{2}'^{C}(w, \Gamma(k_{1})) \mid (p, k_{2}+1) \notin A \}\,. 
\end{align}
Then we can define an involution $\iota$ on $\A_{2}'^{C}(w, \Gamma(k_{1}))$ by 
\begin{align}
A \in \A_{2}'^{C, 2}(w, \Gamma(k_{1})) \ &\mapsto \ \iota(A) := A \sqcup \{(p, k_{2}+1)\} \in \A_{2}'^{C, 1}(w, \Gamma(k_{1}))\,, \\ 
A \in \A_{2}'^{C, 1}(w, \Gamma(k_{1})) \ &\mapsto \ \iota(A) := A \setminus \{(p, k_{2}+1)\} \in \A_{2}'^{C, 2}(w, \Gamma(k_{1}))\,. 
\end{align}
This $\iota$ has the following properties:
\begin{itemize}
\item $\wend(w, \iota(A)) = \wend(w, A)(p, k_{1})$ (and hence $\mcr{\wend(w, \iota(A))} = \mcr{\wend(w, A)}$); 
\item $|\iota(A)| = |A| \pm 1$. 
\end{itemize}
By using the involution $\iota$, we obtain the following. 

\begin{lem} \label{lem:2-step_cancellation_3-2-1-A2'C}
Assume condition \textup{(Q)}. If $w(1) < w(k_{2}+1)$, and if condition \textup{(Q-A)} does not hold, then 
\begin{equation}
\be^{w\varpi_{k_{1}}} \sum_{A \in \A_{2}'^{C}(w, \Gamma(k_{1}))} (-1)^{|A|} Q_{k_{1}}Q_{k_{2}} [\cO^{\mcr{\wend(w, A)}}] = 0\,. 
\end{equation}
\end{lem}

Similarly, we set 
\begin{align}
\A_{3}'(w, \Gamma(k_{1})) &:= \{ A \in \A_{3}(w, \Gamma(k_{1})) \mid (p, k_{2}+1) \notin A \}\,, \\ 
\A_{3}'^{C}(w, \Gamma(k_{1})) &:= \A_{3}(w, \Gamma(k_{1})) \setminus \A_{3}'(w, \Gamma(k_{1}))\,. 
\end{align}
Observe that if $A \in \A_{3}'(w, \Gamma(k_{1}))$, then we must have $(p, k_{1}+1) \notin A$; if not, then $A$ cannot contain a quantum step $(k_{1}, k_{2}+1)$. 
However, we can define an involution on $\A_{3}'(w, \Gamma(k_{1}))$ similar to $\iota_{\A_{3}}$ as follows. 
We set 
\begin{align}
\A_{3}'^{C, 1}(w, \Gamma(k_{1})) &:= \{ A \in \A_{3}'^{C}(w, \Gamma(k_{1})) \mid (p, k_{1}+1) \in A \}\,, \\ 
\A_{3}'^{C, 2}(w, \Gamma(k_{1})) &:= \{ A \in \A_{3}'^{C}(w, \Gamma(k_{1})) \mid (p, k_{1}+1) \notin A \}\,. 
\end{align}
Then we can define an involution $\iota$ on $\A_{3}'^{C}(w, \Gamma(k_{1}))$ by 
\begin{align}
A \in \A_{3}'^{C, 2}(w, \Gamma(k_{1})) \ &\mapsto \ \iota(A) := A \sqcup \{(p, k_{1}+1)\} \in \A_{3}'^{C, 1}(w, \Gamma(k_{1}))\,, \\ 
A \in \A_{3}'^{C, 1}(w, \Gamma(k_{1})) \ &\mapsto \ \iota(A) := A \setminus \{(p, k_{1}+1)\} \in \A_{3}'^{C, 2}(w, \Gamma(k_{1}))\,. 
\end{align}
This $\iota$ has the following properties: 
\begin{itemize}
\item $\wend(w, \iota(A)) = \wend(w, A)(p, k_{1})$ (and hence $\mcr{\wend(w, \iota(A))} = \mcr{\wend(w, A)}$); 
\item $|\iota(A)| = |A| \pm 1$. 
\end{itemize}
By using the involution $\iota$, we obtain the following. 

\begin{lem} \label{lem:2-step_cancellation_3-2-1-A3'C}
Assume condition \textup{(Q)}. If $w(1) < w(k_{2}+1)$, and if condition \textup{(Q-A)} does not hold, then 
\begin{equation}
\be^{w\varpi_{k_{1}}} \sum_{A \in \A_{3}'^{C}(w, \Gamma(k_{1}))} (-1)^{|A|} Q_{k_{1}} Q_{k_{2}} [\cO^{\mcr{\wend(w, A)}}] = 0\,. 
\end{equation}
\end{lem}

It remains to examine cancellations for the set 
\begin{equation}
\A_{23}'(w, \Gamma(k_{1})) := \A_{2}'(w, \Gamma(k_{1})) \sqcup \A_{3}'(w, \Gamma(k_{1}))\,. 
\end{equation}
The desired involution on $\A_{23}'(w, \Gamma(k_{1}))$ is given as follows: 
\begin{align}
A \in \A_{2}'(w, \Gamma(k_{1})) \ &\mapsto \ \iota(A) := (A \setminus \{(p, k_{2}+1), (p, k_{1}+1)\}) \sqcup \{(k_{1}, k_{1}+1)\} \in \A_{3}'(w, \Gamma(k_{1}))\,, \\ 
A \in \A_{3}'(w, \Gamma(k_{1})) \ &\mapsto \ \iota(A) := (A \setminus \{(k_{1}, k_{1}+1)\}) \sqcup \{(p, k_{2}+1), (p, k_{1}+1)\} \in \A_{2}'(w, \Gamma(k_{1}))\,. 
\end{align}
This $\iota$ has the following properties: 
\begin{itemize}
\item $\wend(w, \iota(A)) = \wend(w, A)(p, k_{1})$ (and hence $\mcr{\wend(w, \iota(A))} = \mcr{\wend(w, A)}$); 
\item $|\iota(A)| = |A| \pm 1$. 
\end{itemize}
By using the involution $\iota$, we obtain the following. 
\begin{lem} \label{lem:2-step_cancellation_3-2-1-A23'}
Assume condition \textup{(Q)}. If $w(1) < w(k_{2}+1)$, and if condition \textup{(Q-A)} does not hold, then 
\begin{equation}
\be^{w\varpi_{k_{1}}} \sum_{A \in \A_{2}'(w, \Gamma(k_{1}))} (-1)^{|A|} Q_{k_{1}}Q_{k_{2}} [\cO^{\mcr{\wend(w, A)}}] + \be^{w\varpi_{k_{1}}} \sum_{A \in \A_{3}'(w, \Gamma(k_{1}))} (-1)^{|A|} Q_{k_{1}}Q_{k_{2}} [\cO^{\mcr{\wend(w, A)}}] = 0\,. 
\end{equation}
\end{lem}

\begin{proof}[Proof of Theorem~\ref{thm:2-step_3}\,(2)\,(a)]
By Lemmas~\ref{lem:2-step_cancellation_3-2-1-A2}, \ref{lem:2-step_cancellation_3-2-1-A3}, \ref{lem:2-step_cancellation_3-2-1-A2'C}, \ref{lem:2-step_cancellation_3-2-1-A3'C}, and \ref{lem:2-step_cancellation_3-2-1-A23'}, we have 
\begin{equation}
\be^{w\varpi_{k_{1}}} \sum_{A \in \A_{2}(w, \Gamma(k_{1}))} (-1)^{|A|} Q_{k_{1}}Q_{k_{2}} [\cO^{\mcr{\wend(w, A)}}] + \be^{w\varpi_{k_{1}}} \sum_{A \in \A_{3}(w, \Gamma(k_{1}))} (-1)^{|A|} Q_{k_{1}}Q_{k_{2}} [\cO^{\mcr{\wend(w, A)}}] = 0\,. 
\end{equation}
Also, since $w(1) < w(k_{2}+1) < w(k_{1}+1)$, Remark~\ref{rem:2-step_cancellation_A1} implies that 
\begin{equation}
\be^{w\varpi_{k_{1}}} \sum_{A \in \A_{1}(w, \Gamma(k_{1}))} (-1)^{|A|} Q_{k_{1}} [\cO^{\mcr{\wend(w, A)}}] = 0\,. 
\end{equation}
These observations, together with Lemma~\ref{lem:2-step_cancellation-free_Bruhat}, prove the desired cancellation-free identity. 
\end{proof}

\begin{rema}
In the proof of Theorem~\ref{thm:2-step_3}\,(2)\,(a), we do not use the assumption that $w(k_{1}) > w(n)$. 
\end{rema}

Now, we assume that $w(k_{1}) > w(n)$ and $w(k_{2}+1) < w(1) < w(k_{1}+1)$, which are the assumptions of Theorem~\ref{thm:2-step_3}\,(2)\,(b); 
note that $w(1) < w(k_{1}+1) < w(k_{1})$ by condition (Q), and hence $k_{1} \not= 1$. 
In this case, the same proof as that of Lemma~\ref{lem:2-step_cancellation_3-2-1-A3} yields the following. 

\begin{lem} \label{lem:2-step_cancellation_3-2-2-A3}
Assume condition \textup{(Q)}. If $w(k_{2}+1) < w(1) < w(k_{1}+1)$, then 
\begin{equation}
\be^{w\varpi_{k_{1}}} \sum_{A \in \A_{3}(w, \Gamma(k_{1}))} (-1)^{|A|} Q_{k_{1}} Q_{k_{2}} [\cO^{\mcr{\wend(w, A)}}] = 0\,. 
\end{equation}
\end{lem}

\begin{rema}
We do not need the assumption that $w(k_{1}) > w(n)$ for Lemma~\ref{lem:2-step_cancellation_3-2-2-A3}. 
\end{rema}

In contrast, the sum over $\A_{2}(w, \Gamma(k_{1}))$ is cancellation-free. Indeed, let $A \in \A_{2}(w, \Gamma(k_{1}))$, and set $y := \wend(w, A \setminus \{(k_{1}, k_{2}+1)\})$. Note that $A \setminus \{(k_{1}, k_{2}+1)\}$ contains only Bruhat steps. Hence we see that $y(k_{2}+1) < y(1)$ and $y(n) < y(k_{1})$. Therefore, if we set $z := y(k_{1}, k_{2}+1) = \wend(w, A)$, then 
\begin{itemize}
\item $z(k_{1}) < z(1) < z(2) < \cdots < z(k_{1}-1)$, 
\item $z(k_{1}+1) < z(k_{1}+2) < \cdots < z(k_{2})$, and 
\item $z(k_{2}+2) < z(k_{2}+3) < \cdots < z(n) < z(k_{2}+1)$; 
\end{itemize}
hence, if we take cyclic permutations $\sigma_{1} := (1, k_{1}, k_{1}-1, \ldots , 2)$ and $\sigma_{2} := (k_{2}+1, k_{2}+2, \ldots, n)$
(if $k_{2}+1 = n$, then we take $\sigma_{2} := e$), 
then we have $\mcr{\wend(w, A)} = \wend(w, A)\sigma_{1}\sigma_{2}$. 
Note that these $\sigma_{1}$ and $\sigma_{2}$ does not depends on the choice of $A$. Thus, for $A, B \in \A_{2}(w, \Gamma(k_{1}))$ with $A \not= B$, it follows that 
\begin{equation}
\mcr{\wend(w, A)} = \wend(w, A)\sigma_{1}\sigma_{2} \not= \wend(w, B)\sigma_{1}\sigma_{2} = \mcr{\wend(w, B)}\,. 
\end{equation}
This, together with Remark~\ref{rem:Chevalley_cancellation-free}, proves the following. 

\begin{lem} \label{lem:2-step_cancellation-free_3-2-2-A2}
Assume condition \textup{(Q)}. If $w(k_{1}) > w(n)$ and $w(k_{2}+1) < w(1) < w(k_{1}+1)$, then the sum 
\begin{equation}
\be^{w\varpi_{k_{1}}} \sum_{A \in \A_{2}(w, \Gamma(k_{1}))} (-1)^{|A|} Q_{k_{1}}Q_{k_{2}} [\cO^{\mcr{\wend(w, A)}}] 
\end{equation}
is cancellation-free. 
\end{lem}

\begin{rema} \label{rem:2-step_deformation_A2}
If $\A_{2}(w, \Gamma(k_{1})) \not= \emptyset$ (or equivalently, if condition (Q) holds), then equation~\eqref{eq:A2} shows that 
\begin{align}
& \be^{w\varpi_{k_{1}}} \sum_{A \in \A_{2}(w, \Gamma(k_{1}))} (-1)^{|A|} Q_{k_{1}}Q_{k_{2}} [\cO^{\mcr{\wend(w, A)}}] \\ 
&= -\be^{w\varpi_{k_{1}}} \sum_{A \in \A_{\lessdot}(w, \Gamma(k_{1}))} (-1)^{|A|} Q_{k_{1}}Q_{k_{2}} [\cO^{\mcr{\wend(w, A)(k_{1}, k_{2}+1)}}]\,. 
\end{align}
\end{rema}

\begin{proof}[Proof of Theorem~\ref{thm:2-step_3}\,(2)\,(b)]
Since $w(1) < w(k_{1}+1)$, Remark~\ref{rem:2-step_cancellation_A1} implies that 
\begin{equation}
\be^{w\varpi_{k_{1}}} \sum_{A \in \A_{1}(w, \Gamma(k_{1}))} (-1)^{|A|} Q_{k_{1}} [\cO^{\mcr{\wend(w, A)}}] = 0\,. 
\end{equation}
Hence, by Lemmas~\ref{lem:2-step_cancellation-free_Bruhat}, \ref{lem:2-step_cancellation_3-2-2-A3}, \ref{lem:2-step_cancellation-free_3-2-2-A2}, and Remark~\ref{rem:2-step_deformation_A2}, we obtain the desired cancellation-free formula. 
\end{proof}

It only remains to prove Theorem~\ref{thm:2-step_3}\,(3). To do so, we assume condition (Full). By the same argument as in the proof of Lemma~\ref{lem:2-step_cancellation-free_3-2-2-A2}, we see that for $A \in \A_{2}(w, \Gamma(k_{1}))$, $\mcr{\wend(w, A)} = \wend(w, A)\sigma_{1}\sigma_{2}$, where $\sigma_{1}$ and $\sigma_{2}$ are the cyclic permutations defined above (if $k_{1} = 1$, then we take $\sigma_{1} := e$). 
In addition, since $w(k_{1}+1) < w(1)$ by condition (Full)\,(2), it follows that $\wend(w, A)(k_{1}) < \wend(w, A)(k_{1}+1) < \wend(w, A)(1)$
(if $k_{1} = 1$, then we need only the inequality $\wend(w, A)(k_{1}+1) < \wend(w, A)(1)$). 
Therefore, for $A \in \A_{3}(w, \Gamma(k_{1}))$ (note that $A \setminus \{(k_{1}, k_{1}+1)\} \in \A_{2}(w, \Gamma(k_{1}))$), the following hold: 
\begin{itemize}
\item $\wend(w, A)(k_{1}) < \wend(w, A)(1) < \wend(w, A)(2) < \cdots < \wend(w, A)(k_{1}-1)$, 
\item $\wend(w, A)(k_{1}+1) < \wend(w, A)(k_{1}+2) < \cdots < \wend(w, A)(k_{2})$, and 
\item $\wend(w, A)(k_{2}+2) < \wend(w, A)(k_{2}+3) < \cdots < \wend(w, A)(n) < \wend(w, A)(k_{2}+1)$. 
\end{itemize}
Thus we conclude that $\mcr{\wend(w, A)} = \wend(w, A)\sigma_{1}\sigma_{2}$, and hence that for $A, B \in \A_{2}(w, \Gamma(k_{1})) \sqcup \A_{3}(w, \Gamma(k_{1}))$ with $A \not= B$, 
\begin{equation}
\mcr{\wend(w, A)} = \wend(w, A)\sigma_{1}\sigma_{2} \not= \wend(w, B)\sigma_{1}\sigma_{2} = \mcr{\wend(w, B)}\,. 
\end{equation}
This, together with Remark~\ref{rem:Chevalley_cancellation-free}, proves the following. 

\begin{lem} \label{lem:2-step_cancellation-free_3-3-A23}
Assume conditions \textup{(Q)} and \textup{(Full)}. Then the sum 
\begin{equation}
\be^{w\varpi_{k_{1}}} \sum_{A \in \A_{2}(w, \Gamma(k_{1}))} (-1)^{|A|} Q_{k_{1}}Q_{k_{2}} [\cO^{\mcr{\wend(w, A)}}] + \be^{w\varpi_{k_{1}}} \sum_{A \in \A_{3}(w, \Gamma(k_{1}))} (-1)^{|A|} Q_{k_{1}}Q_{k_{2}} [\cO^{\mcr{\wend(w, A)}}]
\end{equation}
is cancellation-free. 
\end{lem}

\begin{rema} \label{rem:2-step_deformation_A3}
If $\A_{3}(w, \Gamma(k_{1})) \not= \emptyset$ (or equivalently, if condition (Q) holds), then equation~\eqref{eq:A3} shows that 
\begin{align}
&\be^{w\varpi_{k_{1}}} \sum_{A \in \A_{3}(w, \Gamma(k_{1}))} (-1)^{|A|} Q_{k_{1}}Q_{k_{2}} [\cO^{\mcr{\wend(w, A)}}] \\ 
&= \be^{w\varpi_{k_{1}}} \sum_{A \in \A_{\lessdot}(w, \Gamma(k_{1}))} (-1)^{|A|} Q_{k_{1}}Q_{k_{2}} [\cO^{\mcr{\wend(w, A)(k_{1}, k_{2}+1)s_{k_{1}}}}]\,. 
\end{align}
Therefore, by Remark~\ref{rem:2-step_deformation_A2}, we deduce that 
\begin{align}
& \be^{w\varpi_{k_{1}}} \sum_{A \in \A_{2}(w, \Gamma(k_{1}))} (-1)^{|A|} Q_{k_{1}}Q_{k_{2}} [\cO^{\mcr{\wend(w, A)}}] + \be^{w\varpi_{k_{1}}} \sum_{A \in \A_{3}(w, \Gamma(k_{1}))} (-1)^{|A|} Q_{k_{1}}Q_{k_{2}} [\cO^{\mcr{\wend(w, A)}}] \\ 
&= - \be^{w\varpi_{k_{1}}} \sum_{A \in \A_{\lessdot}(w, \Gamma(k_{1}))} (-1)^{|A|} Q_{k_{1}}Q_{k_{2}} \left( [\cO^{\mcr{\wend(w, A)(k_{1}, k_{2}+1)}}] - [\cO^{\mcr{\wend(w, A)(k_{1}, k_{2}+1)s_{k_{1}}}}] \right)\,. 
\end{align}
\end{rema}

\begin{proof}[Proof of Theorem~\ref{thm:2-step_3}\,(3)]
The desired identity follows from Lemmas~\ref{lem:2-step_cancellation-free_Bruhat}, \ref{lem:2-step_cancellation-free_3-3-A23} and Remarks~\ref{rem:2-step_cancellation-free_A1}, \ref{rem:2-step_deformation_A1}, \ref{rem:2-step_deformation_A3}. 
\end{proof}

\subsection{The positivity property}
We prove the positivity property of structure constants for two-step flag manifolds in type $A$, which is a corollary of Chevalley formulas (Theorems~\ref{thm:2-step_1}, \ref{thm:2-step_2}, \ref{thm:2-step_3}, \ref{thm:2-step_4}, \ref{thm:2-step_5}, and \ref{thm:2-step_6}). 
\begin{cor} \label{cor:positivity_2-step}
Let $G$ be of type $A_{n-1}$, $J = I \setminus \{k_{1}, k_{2}\}$ for arbitrarily fixed $1 \le k_{1} < k_{2} \le n-1$, and $k = k_{1}$ or $k = k_{2}$. 
Then, for $w, u \in W^{J}$ and $\xi \in Q_{I \setminus J}^{\vee, +}$, we have 
\begin{equation}
(-1)^{1 + \ell(w) + \ell(u) + \deg(Q^{\xi})}N_{s_{k}, w}^{u, \xi} \in \bZ_{\ge 0}[\be^{\gamma} - 1 \mid \gamma \in -\Delta].  
\end{equation}
\end{cor}
\begin{proof}
We give a proof of Corollary~\ref{cor:positivity_2-step} under the assumptions of Theorems~\ref{thm:2-step_1}, \ref{thm:2-step_2} and \ref{thm:2-step_3}. 
The positivity property under the assumptions of Theorems~\ref{thm:2-step_4}, \ref{thm:2-step_5}, and \ref{thm:2-step_6} follows by the same arguments as those for Theorems~\ref{thm:2-step_1}, \ref{thm:2-step_2} and \ref{thm:2-step_3}. 
Note that the positivity property under the assumptions of Theorems~\ref{thm:2-step_1}, \ref{thm:2-step_2}\,(1), \ref{thm:2-step_3}\,(1)\,(a) and (2)\,(a) has already been known because of the positivity property of $N_{u, v}^{w, 0}$ for $u, v, w \in W^{J}$. 
First, it is easy to check (see, for example, \cite[Section~3.1.5, Exercise~4]{gwrep}) that 
\begin{equation}
2\rho_{J} = \sum_{i = 1}^{k_{1}-1} i (k_{1} - i) \alpha_{i} + \sum_{i = 1}^{k_{2} - k_{1} - 1} i (k_{2} - k_{1} - i) \alpha_{k_{1}+i} + \sum_{i = 1}^{n - k_{2} - 1} i (n - k_{2} - i) \alpha_{k_{2}+i}. 
\end{equation}
Since 
\begin{equation}
\pair{\alpha_{i}}{\alpha_{j}^{\vee}} = \begin{cases}
2 & \text{if $i = j$}, \\ 
-1 & \text{if $|i - j| = 1$}, \\ 
0 & \text{otherwise}, 
\end{cases}
\end{equation}
we have 
\begin{align}
2\pair{\rho_{J}}{\alpha_{k_{1}}^{\vee}} &= -(k_{1} - 1) - (k_{2} - k_{1} - 1) = 2 - k_{2}, \\ 
2\pair{\rho_{J}}{\alpha_{k_{2}}^{\vee}} &= -(k_{2} - k_{1} - 1) - (n - k_{2} - 1) = 2 - n + k_{1}. 
\end{align}
In addition, we know that $2\pair{\rho}{\alpha_{k_{1}}^{\vee}} = 2$. Therefore, 
\begin{align}
\deg(Q_{k_{1}}) &= 2\pair{\rho - \rho_{J}}{\alpha_{k_{1}}^{\vee}} = -k_{2} \\ 
\deg(Q_{k_{2}}) &= 2\pair{\rho - \rho_{J}}{\alpha_{k_{2}}^{\vee}} = n - k_{1}, 
\end{align}
and hence 
\begin{align}
\deg(\underbrace{Q_{k_{1}}Q_{k_{2}}}_{= Q^{\alpha_{k_{1}}^{\vee} + \alpha_{k_{2}}^{\vee}}}) = 2\pair{\rho - \rho_{J}}{\alpha_{k_{1}}^{\vee} + \alpha_{k_{2}}^{\vee}} = n - k_{1} + k_{2}. 
\end{align}

Let us consider the structure constants $N_{s_{k_{1}}, w}^{u, \xi}$ with $\xi \not= 0$ under the assumptions of Theorems~\ref{thm:2-step_2}\,(2) and \ref{thm:2-step_3}\,(1)\,(b). 
We maintain the setting of Lemma~\ref{lem:2-step_cancellation-free_2-2} except for the negation of condition (Q) (see Remark~\ref{rem:2-step_cancellation-free_A1}). 
Take $A \in \A_{1}(w, \Gamma(k_{1}))$, and set $u := \mcr{\wend(w, A)}$, $u_{0} := \wend(w, A)$. 
Then, by the proof of Lemma~\ref{lem:2-step_cancellation-free_2-2}, we have $u = u_{0}\sigma_{1}\sigma_{2}$, where 
$\sigma_{1} = (1, k_{1}, k_{1}-1, \ldots, 2)$ and $\sigma_{2} = (k_{1}+1, k_{1}+2, \ldots, k_{2})$. 
Therefore, we see that 
\begin{align}
& (-1)^{|A|} \be^{w\vpi_{k_{1}}} Q_{k_{1}} [\cO^{\mcr{\wend(w, A)}}] \\ 
&= (-1)^{\ell(u_{0}) - \ell(w)} \be^{w\vpi_{k_{1}}} Q_{k_{1}} [\cO^{u}] \\ 
&= (-1)^{(\ell(u_{0}\sigma_{1}\sigma_{2}) - \ell(\sigma_{1}) - \ell(\sigma_{2})) - \ell(w)} \be^{w\vpi_{k_{1}}} Q_{k_{1}} [\cO^{u}] \\ 
&= (-1)^{\ell(u) - (k_{1} - 1) - (k_{2} - k_{1} - 1) - \ell(w)} \be^{w\vpi_{k_{1}}} Q_{k_{1}} [\cO^{u}] \\ 
&= (-1)^{\ell(w) + \ell(u) + k_{2}} \be^{w\vpi_{k_{1}}} Q_{k_{1}} [\cO^{u}] \\ 
&= (-1)^{\ell(w) + \ell(u) + \deg(Q_{k_{1}})} \be^{w\vpi_{k_{1}}} Q_{k_{1}} [\cO^{u}]. 
\end{align}
We set 
\begin{equation}
\A_{1}(w, \Gamma(k_{1}))_{u} := \{ A \in \A_{1}(w, \Gamma(k_{1})) \mid \mcr{\wend(w, A)} = u \}
\end{equation}
for $u \in W^{J}$. Then, for $u \in W^{J}$, we deduce from Theorems~\ref{thm:2-step_2}\,(2) and \ref{thm:2-step_3}\,(1)\,(b) that 
\begin{equation}
C_{w}^{u, \alpha_{k_{1}}^{\vee}} = \be^{w\vpi_{k_{1}}} \sum_{A \in \A_{1}(w, \Gamma(k_{1}))_{u}} (-1)^{|A|} = (-1)^{\ell(u) + \ell(w) + \deg(Q_{k_{1}})} |\A_{1}(w, \Gamma(k_{1}))_{u}| \be^{w\vpi_{k_{1}}}, 
\end{equation}
and hence that 
\begin{equation}
N_{s_{k_{1}}, w}^{u, \alpha_{k_{1}}^{\vee}} = (-1)^{1 + \ell(w) + \ell(u) + \deg(Q_{k_{1}})} |\A_{1}(w, \Gamma(k_{1}))_{u}| \be^{w\vpi_{k_{1}} - \vpi_{k_{1}}}. 
\end{equation}
Since $w\vpi_{k_{1}} - \vpi_{k_{1}} \in -Q^{+}$ and hence $\be^{w\vpi_{k_{1}} - \vpi_{k_{1}}} \in \bZ_{\ge 0}[\be^{\gamma} - 1 \mid \gamma \in -\Delta]$, we conclude that 
\begin{equation} \label{eq:positivity_2-step_1}
(-1)^{1 + \ell(w) + \ell(u) + \deg(Q_{k_{1}})} N_{s_{k_{1}}, w}^{u, \alpha_{k_{1}}^{\vee}} \in \bZ_{\ge 0}[\be^{\gamma} - 1 \mid \gamma \in -\Delta], 
\end{equation}
as desired. 
Equation~\eqref{eq:positivity_2-step_1}, together with the positivity property of $N_{u, v}^{w, 0}$ for $u, v, w \in W^{J}$, implies Corollary~\ref{cor:positivity_2-step} under the assumptions of Theorems~\ref{thm:2-step_2}\,(2) and \ref{thm:2-step_3}\,(1)\,(b). 

Next, we consider the structure constants $N_{s_{k_{1}}, w}^{u, \xi}$ with $\xi \not= 0$ under the assumption of Theorem~\ref{thm:2-step_3}\,(2)\,(b). 
We maintain the setting of Lemma~\ref{lem:2-step_cancellation-free_3-2-2-A2}. Take $A \in \A_{2}(w, \Gamma(k_{1}))$, and set 
$u := \mcr{\wend(w, A)}$, $u_{0} := \wend(w, A)$. 
Then, by the proof of Lemma~\ref{lem:2-step_cancellation-free_3-2-2-A2}, we have $u = u_{0}\sigma_{1}\sigma_{2}$, where 
$\sigma_{1} = (1, k_{1}, k_{1}-1, \ldots, 2)$ and $\sigma_{2} = (k_{2}+1, k_{2}+2, \ldots, n)$. 
Therefore, we see that 
\begin{align}
&(-1)^{|A|} \be^{w\vpi_{k_{1}}} Q_{k_{1}}Q_{k_{2}} [\cO^{\mcr{\wend(w, A)}}] \\ 
&= (-1)^{\ell(u_{0}) - \ell(w)} \be^{w\vpi_{k_{1}}} Q_{k_{1}}Q_{k_{2}} [\cO^{u}] \\ 
&= (-1)^{(\ell(u_{0}\sigma_{1}\sigma_{2}) - \ell(\sigma_{1}) - \ell(\sigma_{2})) - \ell(w)} \be^{w\vpi_{k_{1}}} Q_{k_{1}}Q_{k_{2}} [\cO^{u}] \\ 
&= (-1)^{\ell(u) - (k_{1}-1) - (n-k_{2}-1) - \ell(w)} \be^{w\vpi_{k_{1}}} Q_{k_{1}}Q_{k_{2}} [\cO^{u}] \\ 
&= (-1)^{\ell(w) + \ell(u) + (n - k_{1} + k_{2})} \be^{w\vpi_{k_{1}}} Q_{k_{1}}Q_{k_{2}} [\cO^{u}] \\ 
&= (-1)^{\ell(w) + \ell(u) + \deg(Q_{k_{1}}Q_{k_{2}})} \be^{w\vpi_{k_{1}}} Q_{k_{1}}Q_{k_{2}} [\cO^{u}]. 
\end{align}
We set 
\begin{equation}
\A_{2}(w, \Gamma(k_{1}))_{u} := \{ A \in \A_{2}(w, \Gamma(k_{1})) \mid \mcr{\wend(w, A)} = u \}
\end{equation}
for $u \in W^{J}$. Then, for $u \in W^{J}$, we deduce from Theorem~\ref{thm:2-step_3}\,(2)\,(b) that 
\begin{equation}
C_{w}^{u, \alpha_{k_{1}}^{\vee} + \alpha_{k_{2}}^{\vee}} = \be^{w\vpi_{k_{1}}} \sum_{A \in \A_{2}(w, \Gamma(k_{1}))_{u}} (-1)^{|A|} = (-1)^{\ell(w) + \ell(u) + \deg(Q_{k_{1}}Q_{k_{2}})} |\A_{2}(w, \Gamma(k_{1}))_{u}| \be^{w\vpi_{k_{1}}}, 
\end{equation}
and hence that 
\begin{equation} \label{eq:positivity_2-step_2_structure_constant}
N_{s_{k_{1}}, w}^{u, \alpha_{k_{1}}^{\vee} + \alpha_{k_{2}}^{\vee}} = (-1)^{1 + \ell(w) + \ell(u) + \deg(Q_{k_{1}}Q_{k_{2}})} |\A_{2}(w, \Gamma(k_{1}))_{u}| \be^{w\vpi_{k_{1}} - \vpi_{k_{1}}}.
\end{equation}
Again, since $\be^{w\vpi_{k_{1}} - \vpi_{k_{1}}} \in \bZ_{\ge 0}[\be^{\gamma} - 1 \mid \gamma \in -\Delta]$, we conclude that 
\begin{equation} \label{eq:positivity_2-step_2}
(-1)^{1 + \ell(w) + \ell(u) + \deg(Q_{k_{1}}Q_{k_{2}})} N_{s_{k_{1}}, w}^{u, \alpha_{k_{1}}^{\vee} + \alpha_{k_{2}}^{\vee}} \in \bZ_{\ge 0}[\be^{\gamma} - 1 \mid \gamma \in -\Delta], 
\end{equation}
as desired. 
Equation~\eqref{eq:positivity_2-step_2}, together with the positivity property of $N_{u, v}^{w, 0}$ for $u, v, w \in W^{J}$, implies Corollary~\ref{cor:positivity_2-step} under the assumption of Theorem~\ref{thm:2-step_3}\,(2)\,(b). 

It remains to consider the structure constants $N_{s_{k_{1}}, w}^{u, \xi}$ with $\xi \not= 0$ under the assumption of Theorem~\ref{thm:2-step_3}\,(3) and consider the structure constants $N_{s_{k_{1}}, w}^{u, \xi}$ for $\xi \not= 0$. 
The positivity property in the case $\xi = \alpha_{k_{1}}^{\vee}$ has already been proved by equation \eqref{eq:positivity_2-step_1}. 
Hence it suffices to consider the case $\xi = \alpha_{k_{1}}^{\vee} + \alpha_{k_{2}}^{\vee}$. 
We maintain the setting of Lemma~\ref{lem:2-step_cancellation-free_3-3-A23}. 
We set 
\begin{equation}
\A_{23}(w, \Gamma(k_{1}))_{u} := \{ A \in \A_{2}(w, \Gamma(k_{1})) \sqcup \A_{3}(w, \Gamma(k_{1})) \mid \mcr{\wend(w, A)} = u \}
\end{equation}
for $u \in W^{J}$. Then, by the same argument as that for equation \eqref{eq:positivity_2-step_2_structure_constant}, we deduce from Theorem~\ref{thm:2-step_3}\,(3) that 
\begin{equation}
N_{s_{k_{1}}, w}^{u, \alpha_{k_{1}}^{\vee} + \alpha_{k_{2}}^{\vee}} = (-1)^{1 + \ell(w) + \ell(u) + \deg(Q_{k_{1}}Q_{k_{2}})} |\A_{23}(w, \Gamma(k_{1}))_{u}| \be^{w\vpi_{k_{1}} - \vpi_{k_{1}}}, 
\end{equation}
and hence conclude that 
\begin{equation} \label{eq:positivity_2-step_3}
(-1)^{1 + \ell(w) + \ell(u) + \deg(Q_{k_{1}}Q_{k_{2}})} N_{s_{k_{1}}, w}^{u, \alpha_{k_{1}}^{\vee} + \alpha_{k_{2}}^{\vee}} \in \bZ_{\ge 0}[\be^{\gamma} - 1 \mid \gamma \in -\Delta], 
\end{equation}
as desired. Equations~\eqref{eq:positivity_2-step_1}, \eqref{eq:positivity_2-step_3}, together with the positivity property of $N_{u, v}^{w, 0}$ for $u, v, w \in W^{J}$, implies Corollary~\ref{cor:positivity_2-step} under the assumption of Theorem~\ref{thm:2-step_3}\,(3). 
This completes the proof of the corollary. 
\end{proof}

\appendix

\section{Another proof of the existence of the multiplicative surjection $\Phi_{J}$} \label{sec:A}
In this appendix, we mainly use the notation of Section~\ref{sec:qk}. 
In addition, we set $QK_{T}^{\mathrm{poly}}(G/B) := K_{T}(G/B) \otimes_{\bZ[\Lambda]} \bZ[\Lambda][Q]$, 
where $\bZ[\Lambda][Q]$ is the polynomial ring with coefficients in $\bZ[\Lambda]$ in the (Novikov) variables $Q_i = Q^{\alpha_i^{\vee}}$, $i \in I$; 
also, for an arbitrary subset $J \subset I$, we set 
$QK_{T}^{\mathrm{poly}}(G/P_{J}) := K_{T}(G/P_{J}) \otimes_{\bZ[\Lambda]} \bZ[\Lambda][Q_{K}]$, 
with $K := I \setminus J$, where $\bZ[\Lambda][Q_{K}]$ is the polynomial ring with coefficients in $\bZ[\Lambda]$ in the variables $Q_{k}$, $k \in K$. 
It is known (see \cite{katqkg}) that there exists a surjective $\mathbb{Z}[\Lambda]$-algebra homomorphism $\Phi_{J}$ from $QK_{T}^{\mathrm{poly}}(G/B)$ onto $QK_{T}^{\mathrm{poly}}(G/P_{J})$ such that $\Phi_{J}(Q^{\xi}[\mathcal{O}^{w}]) = Q^{[\xi]^J}[\mathcal{O}_{J}^{\lfloor w \rfloor}]$ for $w \in W$ and $\xi \in Q^{\vee,+}$, where $[\xi]^{J} := \sum_{k \in I \setminus J} c_{k} \alpha_{k}^{\vee}$ for $\xi = \sum_{i \in I} c_{i} \alpha_{i}^{\vee} \in Q^{\vee,+}$.
In this appendix, based on results in \cite{newprf}, we give another (short) proof of the existence of such a $\bZ[\Lambda]$-algebra homomorphism. 
First of all, we note that $QK_{T}^{\mathrm{poly}}(G/B)$ is a $\bZ[\Lambda]$-subalgebra of $QK_{T}(G/B) = K_{T}(G/B) \otimes_{\bZ[\Lambda]} \bZ[\Lambda][\![Q]\!]$ 
by \cite[Corollary~1.2]{newprf}. 

Let us briefly recall the main result of \cite{newprf}. 
Following \cite{newprf}, let $Gr_{G}$ denote Pressley-Segal's model of the 
affine Grassmannian associated to a simple and simply-connected complex Lie group $G$; more precisely, let $Gr_{G}$ be the space of polynomial based loops in a (fixed) maximal compact subgroup of $G$, equipped with an ind-variety structure  (see \cite[Chapter~8]{psloop} for details). 
We denote by $K^{T}(Gr_{G})$ the $T$-equivariant $K$-homology (in the topological sense) of the affine Grassmannian $Gr_{G}$, equipped with the Pontryagin product $\odot$ coming from the group product on the topological group $Gr_{G}$. 
Then, we have two bases. One is a basis (called the localization basis) $\mathcal{O}_{\xi} := [\mathcal{O}_{x_{\xi}}]$, $\xi \in Q^{\vee}$, of $K^{T}(Gr_{G})$ over $\mathrm{Frac}(\bZ[\Lambda])$, where $x_{\xi}$ is the $T$-fixed point of $Gr_{G}$ corresponding to the cocharacter of $T$ associated to $\xi \in Q^{\vee}$. 
More precisely, if we consider the $\bZ[\Lambda]$-algebra $\bigoplus_{\xi \in Q^{\vee}} \mathrm{Frac}(\bZ[\Lambda]) \mathcal{O}_{\xi}$ equipped with the product $\odot$ defined by $\mathcal{O}_{\xi_1} \odot \mathcal{O}_{\xi_2} := \mathcal{O}_{\xi_1 + \xi_2}$, $\xi_1, \xi_2 \in Q^{\vee}$, then we have an injective $\bZ[\Lambda]$-algebra homomorphism $K^{T}(Gr_{G}) \hookrightarrow \bigoplus_{\xi \in Q^{\vee}} \mathrm{Frac}(\bZ[\Lambda]) \mathcal{O}_{\xi}$ which fixes every $\mathcal{O}_{\xi}$. 
Another is indeed a basis of $K^{T}(Gr_{G})$ over $\bZ[\Lambda]$ given as follows. 
Let $W_{\mathrm{af}} = W \ltimes Q^{\vee}$ be the affine Weyl group of $G$, and let $W_{\mathrm{af}}^{-}$ denote the set of minimal-length coset representatives for $W_{\mathrm{af}}/W$. We know from \cite[Section 3]{lash} that an element $w t_{\xi} \in W_{\mathrm{af}}$, 
with $w \in W$ and $\xi \in Q^{\vee}$, lies in $W_{\mathrm{af}}^{-}$ if and only if $\xi \in Q^{\vee}$ is anti-dominant and $w$ is of minimal length in its coset $w W_{\xi}$ in $W/W_{\xi}$, where $W_{\xi} \subset W$ is the stabilizer of $\xi$ in $W$; note that if $\xi \in Q^{\vee}$ is anti-dominant, then $\xi \in - Q^{\vee,+}$. 
In particular, if $\xi \in Q^{\vee}$ is regular anti-dominant, then $w t_{\xi} \in W_{\mathrm{af}}^{0}$ for all $w \in W$. 
For each $w t_{\xi} \in W_{\mathrm{af}}^{0}$, there exists a complex cell (called an affine Schubert cell) in $Gr_{G}$ containing the $T$-fixed point $x_{w \xi} \in Gr_{G}$ of finite dimension; the class of the structure sheaf of the Zariski closure of this cell is denoted by $\mathcal{O}_{w t_{\xi}}$, and is called the affine Schubert class associated to $w t_{\xi} \in W_{\mathrm{af}}^{-}$. 
Then we know that the classes $\mathcal{O}_{w t_{\xi}}$, $w t_{\xi} \in W_{\mathrm{af}}^{-}$, form a $\bZ[\Lambda]$-basis of $K^{T}(Gr_{G})$. 

Now the main result of \cite{newprf} is stated as follows. 

\begin{thm}[{\cite[Theorem~1.1]{newprf}}] \label{thm:K-Peterson}
Let $J$ be an arbitrary subset of $I$. Then, 
there exist a $\mathbb{Z}[\Lambda]$-algebra homomorphism 
\begin{equation}
\Psi_{J} : K^{T}(Gr_{G}) \to QK_{T}(G/P_{J})[(Q^{\vee,+})^{-1}],
\end{equation}
where $QK_{T}(G/P_{J})[(Q^{\vee,+})^{-1}]$ denotes the localization of $QK_{T}(G/P_{J})$ with respect to the monomials in the Novikov variables $Q_i = Q^{\alpha_i^{\vee}}$, $i \in I$. 
Moreover, $\Psi_{J}(\mathcal{O}_{wt_{\xi}}) = Q^{[\xi]^{J}}[\mathcal{O}_{J}^{\lfloor w \rfloor}]$ for each $wt_{\xi} \in W_{\mathrm{af}}^{-}$, 
where $[\xi]^{J} = \sum_{k \in I \setminus J} c_{k} \alpha_{k}^{\vee}$ for $\xi = \sum_{i \in I} c_{i} \alpha_{i}^{\vee} \in Q^{\vee}$, 
and $[\mathcal{O}_{J}^{\lfloor w \rfloor}]$ denotes the (opposite) Schubert class in $K_{T}(G/P_{J})$ associated to the minimal-length coset representative 
$\lfloor w \rfloor \in W^{J}$ for the coset $w W_{J}$ in $W/W_{J}$. 
\end{thm}

Note that in the case $J = \emptyset$, i.e., $P_{J} = B$, the $\bZ[\Lambda]$-algebra homomorphism $\Psi := \Psi_{\emptyset}$ is injective since the affine Schubert classes $\mathcal{O}_{w t_{\xi}}$, $w t_{\xi} \in W_{\mathrm{af}}^{-}$, form a $\bZ[\Lambda]$-basis of $K^{T}(Gr_{G})$ and 
$\Psi([\mathcal{O}_{wt_{\xi}}]) = Q^{\xi}[\mathcal{O}^{w}]$. 

We will construct a surjective $\mathbb{Z}[\Lambda]$-algebra homomorphism $\Phi_{J}$ from $QK_{T}^{\mathrm{poly}}(G/B)$ to $QK_{T}^{\mathrm{poly}}(G/P_{J})$ such that $\Phi_{J}(Q^{\xi}[\mathcal{O}^{w}]) = Q^{[\xi]^J}[\mathcal{O}_{J}^{\lfloor w \rfloor}]$ for $w \in W$ and $\xi \in Q^{\vee,+}$, where $[\xi]^{J} = \sum_{k \in I \setminus J} c_{k} \alpha_{k}^{\vee}$ for $\xi = \sum_{i \in I} c_{i} \alpha_{i}^{\vee} \in Q^{\vee,+}$. 
We first note that for each element $v \in QK_{T}^{\mathrm{poly}}(G/B) = K_T(G/B) \otimes_{\bZ[\Lambda]} \bZ[\Lambda][Q]$, there exists a sufficiently regular anti-dominant coroot $\eta \in -Q^{\vee,+}$ such that $Q^{\eta} \, v \in QK_{T}(G/B)[(Q^{\vee,+})^{-1}]$ lies in the image of the map $\Psi$, i.e., 
$Q^{\eta} \, v \in \Psi(K^{T}(Gr_{G}))$; by the injectivity of $\Psi$, there exists a unique $u \in K^{T}(Gr_{G})$ such that $\Psi(u) = Q^{\eta} \, v$. 
Indeed, we may assume that $v = Q^{\xi} [\mathcal{O}^w]$ for some $w \in W$ and $\xi \in Q^{\vee,+}$ since each $v \in QK_{T}^{\mathrm{poly}}(G/B)$ is, by its definition, a finite linear combination with coefficients in $\bZ[\Lambda]$ of such elements. 
Hence we can take a sufficiently regular anti-dominant coroot $\eta \in Q^{\vee}$ such that $\xi + \eta \in Q^{\vee}$ is also regular anti-dominant; note that we have $\eta \in - Q^{\vee,+}$ since $\eta \in Q^{\vee}$ is anti-dominant. We set $u := \mathcal{O}_{w t_{\xi + \eta}}$, where $w t_{\xi + \eta}$ lies in $W_{\mathrm{af}}^{-}$ since $\xi + \eta \in Q^{\vee}$ is regular anti-dominant. Then it follows that $\Psi(u) = Q^{\eta} \, v$ by Theorem~\ref{thm:K-Peterson}. 
Now we define $\Phi_{J}(v) := Q^{[-\eta]^{J}} \Psi_{J}(u) \in QK_{T}(G/P_{J})$. We can easily verify that the element $Q^{[-\eta]^{J}} \Psi_{J}(u)$ does not depend on the choice of (a sufficiently regular anti-dominant coroot) $\eta \in -Q^{\vee,+}$, and hence that $\Phi_{J}$ is a well-defined surjective $\mathbb{Z}[\Lambda]$-module homomorphism from $QK_{T}^{\mathrm{poly}}(G/B)$ onto $QK_{T}^{\mathrm{poly}}(G/P_{J})$. 
Indeed, if $v = Q^{\xi} [\mathcal{O}^{w}]$ with $w \in W$ and $\xi \in Q^{\vee,+}$, then $\Phi_{J}(Q^{\xi} [\mathcal{O}^{w}]) = Q^{[\xi]^{J}}[\mathcal{O}_{J}^{\lfloor w \rfloor}]$. 

Also, for $v_1, v_2 \in QK_{T}^{\mathrm{poly}}(G/B)$, we can take sufficiently regular anti-dominant coroots $\eta_1, \eta_2 \in -Q^{\vee,+}$ such that $Q^{\eta_1} \, v_1, Q^{\eta_2} \, v_2 \in \Psi(K^{T}(Gr_{G}))$; hence there exist uniquely $u_1, u_2 \in K^{T}(Gr_{G})$ such that $\Psi(u_1) = Q^{\eta_1} \, v_1$ and $\Psi(u_2) = Q^{\eta_2} \, v_2$. 
Since $\Psi = \Psi_{\emptyset}$ is a $\bZ[\Lambda]$-algebra homomorphism, we have $Q^{\eta_1 + \eta_2}(v_1 \cdot v_2) = (Q^{\eta_1} \, v_1) \cdot (Q^{\eta_2} \, v_2) = \Psi(u_1) \cdot \Psi(u_2) = \Psi(u_1 \odot u_2)$ in $QK_{T}(G/B)[(Q^{\vee,+})^{-1}]$, 
where $u_1 \odot u_2 \in K^{T}(Gr_{G})$.
Therefore, we see that $\Phi_{J}(v_1 \cdot v_2) = Q^{[-\eta_1-\eta_2]^J} \Psi_{J}(u_1 \odot u_2) = Q^{[-\eta_1-\eta_2]^J} (\Psi_{J}(u_1) \cdot \Psi_{J}(u_2)) = 
(Q^{[-\eta_1]^J} \Psi_{J}(u_1)) \cdot (Q^{[-\eta_2]^J} \Psi_{J}) = \Phi_{J}(v_1) \cdot \Phi_{J}(v_2)$ in $QK_{T}(G/P_{J})[(Q^{\vee,+})^{-1}]$ 
since $\Psi_{J}$ is a $\bZ[\Lambda]$-algebra homomorphism. This proves that the map $\Phi_{J}$ is a $\bZ[\Lambda]$-algebra homomorphism from $QK_{T}^{\mathrm{poly}}(G/B)$ to $QK_{T}^{\mathrm{poly}}(G/P_{J})$, as desired. 

Finally, since $[\mathcal{O}^{s_i}] = 1 - \mathbf{e}^{-\varpi_i}[\mathcal{O}_{G/B}(-\varpi_i)]$ in $K_{T}(G/B)$ for all $i \in I$ and $[\mathcal{O}^{s_k}] = 1 - \mathbf{e}^{-\varpi_k}[\mathcal{O}_{G/P_{J}}(-\varpi_k)]$ in $K_{T}(G/P_{J})$ for all $k \in K = I \setminus J$, it follows that $\Phi_{J}([\mathcal{O}_{G/B}(-\varpi_k)]) = [\mathcal{O}_{G/P_{J}}(-\varpi_k)]$, and hence that $\Phi_{J}([\mathcal{O}^{w}] \cdot [\mathcal{O}_{G/B}(-\varpi_k)]) = \Phi_{J}([\mathcal{O}^{w}]) \cdot \Phi_{J}([\mathcal{O}_{G/B}(-\varpi_k)]) = [\mathcal{O}_{J}^{\lfloor w \rfloor}] \cdot [\mathcal{O}_{G/P_{J}}(-\varpi_k)]$ for all $k \in K = I \setminus J$. 

Thus, we have given a proof of the following fact; cf. Theorem~\ref{thm:qksurj}, due to Kato (\cite{katqkg}). 

\begin{cor}
Let $J$ be an arbitrary subset of $I$. Then, 
there exist a surjective $\mathbb{Z}[\Lambda]$-algebra homomorphism 
\begin{equation}
\Phi_{J} : QK_{T}^{\mathrm{poly}}(G/B) \to QK_{T}^{\mathrm{poly}}(G/P_{J})
\end{equation}
such that 
$\Phi_{J}(Q^{\xi}[\mathcal{O}^{w}]) = Q^{[\xi]^J}[\mathcal{O}_{J}^{\lfloor w \rfloor}]$ for $w \in W$ and $\xi \in Q^{\vee,+}$, where $[\xi]^{J} := \sum_{k \in I \setminus J} c_{k} \alpha_{k}^{\vee}$ for $\xi = \sum_{i \in I} c_{i} \alpha_{i}^{\vee} \in Q^{\vee,+}$, 
and $\lfloor w \rfloor \in W^{J}$ denotes the minimal-length coset representative for the coset $w W_{J}$ in $W/W_{J}$. 
Also, for each $k \in K = I \setminus J$, the following equality holds for all $w \in W$: 
\begin{equation}
\Phi_{J}([\mathcal{O}^{w}] \cdot [\mathcal{O}_{G/B}(-\varpi_k)]) = [\mathcal{O}_{J}^{\lfloor w \rfloor}] \cdot [\mathcal{O}_{G/P_{J}}(-\varpi_k)]. 
\end{equation}
\end{cor}

\section{Weihong Xu's conjecture about a cancellation-free parabolic Chevalley formula in type $A$ (with Weihong Xu)}

In this appendix, we mention the relation between our results and a conjecture due to Weihong~Xu, which is expected to be a cancellation-free Chevalley formula in type $A$ for an arbitrary subset $J \subset I$. 

Let $G$ be of type $A_{n-1}$. Take $1 \le k_{1} < k_{2} < \cdots < k_{m} \le n-1$, and set $J := I \setminus \{k_{1}, \ldots, k_{m}\}$. 
In this case, the partial flag manifold $G/P_{J}$ is isomorphic to the $m$-step flag manifold $\Fl(k_{1}, \ldots, k_{m}; n)$, defined as: 
\begin{equation}
\Fl(k_{1}, \ldots, k_{m}; n) := \left\{ (V_{1}, \ldots, V_{m}) \ \middle| \ \parbox{20em}{$V_{l}$, $l = 1, \ldots, m$, is a subspace of $\bC^{n}$ such that $\dim V_{l} = k_{l}$, and $V_{1} \subset V_{2} \subset \cdots \subset V_{m}$} \right\}. 
\end{equation}
For a directed path 
\begin{equation}
\bp: w_{0} \xrightarrow{\gamma_{1}} w_{1} \xrightarrow{\gamma_{2}} \cdots \xrightarrow{\gamma_{r}} w_{r}
\end{equation}
in $\QB(W)$, we define $\ell(\bp) \ge 0$, $\wend(\bp) \in W$, and $\wt(\bp) \in Q^{\vee, +}$ by 
\begin{align}
\ell(\bp) &:= r, \\ 
\wend(\bp) &:= w_{r}, \\ 
\wt(\bp) &:= \sum_{\substack{1 \le k \le r \\ w_{k-1} \rightarrow w_{k} \text{is a quantum edge}}} \gamma_{k}^{\vee}. 
\end{align}
Also, for $1 \le a \le n-1$, the quantum $a$-Bruhat graph $\QB_{a}(W)$ is defined to be the subgraph of $\QB(W)$ having only those edges whose labels are of the form $(i, j)$ such that $i \le a < j$. 
In addition, we define a total order $\vtl$ on $\Phi^{+}$ as follows: 
for $1 \le i < j \le n$ and $1 \le k < l \le n$, we define $(i, j) \vtl (k, l)$ if ($j > l$) or ($j = l$ and $i < k$). 

Xu formulated the following conjecture on a cancellation-free Chevalley formula for $QK(G/P_{J})$, the non-equivariant quantum $K$-theory of $G/P_{J}$, and checked it for all partial flag manifolds with \(n\leq 8\) and \(m\leq 4\) using a computer program.  
\begin{conj}[Weihong Xu] \label{conj:Xu}
In $QK(G/P_{J})$, for $w \in W^{J}$, the following cancellation-free formula holds:  
\begin{equation} \label{eq:Xu}
[\cO^{s_{k_{l}}}] \cdot [\cO^{w}] = \sum_{\bp} (-1)^{\ell(\bp) - 1} Q^{[\wt(\bp)]^{J}} [\cO^{\mcr{\wend(\bp)}}], 
\end{equation}
where the sum on the right-hand side is over all non-empty paths $\bp$ in $\QB_{k_{l}}(W)$ of the form 
\begin{equation}
\bp: w = w_{0} \xrightarrow{(i_{1}, j_{1})} w_{1} \xrightarrow{(i_{2}, j_{2})} \cdots \xrightarrow{(i_{r}, j_{r})} w_{r}
\end{equation}
such that 
\begin{enumerate}
\item $(i_{1}, j_{1}) \vtl (i_{2}, j_{2}) \vtl \cdots \vtl (i_{r}, j_{r})$, 
\item for each $0 \le t \le r$ (regarding as $k_{0} = 0$ and $k_{n+1} = n$) and an edge $v \xrightarrow{(i, j)} w$ in $\bp$, 
\begin{itemize}
\item there does not exist any paths of the form $v \xrightarrow{(i, j')} w'$ in $\QB_{k_{l}}(W)$ such that $k_{t}+1 \le j < j' \le k_{t+1}$, 
\item there does not exist any paths of the form $v \xrightarrow{(i', j)} w'$ in $\QB_{k_{l}}(W)$ such that $k_{t}+1 \le i' < i \le k_{t+1}$,
\end{itemize}
\item if there are two edges $\xrightarrow{(i, j)}$ and $\xrightarrow{(i, j')}$ in $\bp$ such that $(i, j) \vtl (i, j')$, then there exists $1 \le t \le n-1$ such that $j' \le k_{t} < j$, 
\item if there are two edges $\xrightarrow{(i, j)}$ and $\xrightarrow{(i', j)}$ in $\bp$ such that $(i, j) \vtl (i', j)$, then there exists $1 \le t \le n-1$ such that $i \le k_{t} < i'$. 
\end{enumerate}
\end{conj}

We now compare Xu's conjectural formula in the case $m = 2$ with our cancellation-free Chevalley formula for two-step flag manifolds. 
For $w \in W^{J}$, we obtain the following formula in $QK(G/P_{J})$ by applying the surjection $\Phi_{J}$ to equation~\eqref{qkchev-f} and specializing at $\be^{\mu} = 1$ for $\mu \in \Lambda$: 
\begin{align}
[\cO(-\vpi_{k_{1}})] \cdot [\cO^{w}] &= \sum_{A \in \A(w, \Gamma(k_{1}))} (-1)^{|A|} Q^{[\dn(w, A)]^{J}} [\cO^{\mcr{\wend(w, A)}}] \\ 
&= \sum_{\bp} (-1)^{\ell(\bp)} Q^{[\wt(\bp)]^{J}} [\cO^{\mcr{\wend(\bp)}}], 
\end{align}
where the sum $\sum_{\bp}$ is over all (possibly empty) directed paths in $\QB_{k_{1}}(W)$ satisfying (1) in Conjecture~\ref{conj:Xu}. 
By the formula $[\cO^{s_{k_{1}}}] = 1 - [\cO(-\vpi_{k_{1}})]$ in $QK(G/P_{J})$, we deduce that 
\begin{equation}
[\cO^{s_{k_{1}}}] \cdot [\cO^{w}] = \sum_{\bp} (-1)^{\ell(\bp)-1} Q^{[\wt(\bp)]^{J}} [\cO^{\mcr{\wend(\bp)}}], 
\end{equation}
where the sum $\sum_{\bp}$ is over all non-empty directed paths in $\QB_{k_{1}}(W)$ satisfying (1) in Conjecture~\ref{conj:Xu}. 
Here, we can construct certain involutions among non-empty directed paths satisfying (1) but not (2), and those satisfying (1) but not (3) or (4). 
Furthermore, we can verify that such involutions agree with those constructed in Section~\ref{sec:2-step} by direct calculation. 
Hence we that equation~\eqref{eq:Xu} coincides with our cancellation-free Chevalley formula (Theorems~\ref{thm:2-step_1}, \ref{thm:2-step_2}, and \ref{thm:2-step_3}). 
We can also consider the product $[\cO^{s_{k_{2}}}] \cdot [\cO^{w}]$ by using the diagram automorphism $\omega$ and the result above for the product $[\cO^{s_{k_{1}}}] \cdot [\cO^{w}]$. 
In addition, we can verify that Xu's conjectural formula \eqref{eq:Xu} coincides with our Chevalley formula for Grassmannians of type $A$ (Theorem~\ref{qkchev-a}) in the same way as above. 

\bibliographystyle{alpha}

\newcommand{\etalchar}[1]{$^{#1}$}

\end{document}